\documentclass[a4paper]{amsart}

\usepackage{amsmath,amssymb}
\usepackage{extarrows}

\usepackage{tikz-cd}
\usetikzlibrary{decorations.pathmorphing}

\numberwithin{equation}{section}
\newtheorem{theorem}[equation]{Theorem}
\newtheorem{corollary}[equation]{Corollary}
\newtheorem{lemma}[equation]{Lemma}

\newtheorem{proposition}[equation]{Proposition}

\theoremstyle{remark}
\newtheorem{remark}[equation]{Remark}
\theoremstyle{definition}
\newtheorem{definition}[equation]{Definition}
\newtheorem{notation}[equation]{Notation}
\newtheorem{example}[equation]{Example}

\newcommand*{\kokoni}[1]{\noindent\makebox[0pt][l]{#1}}

\newcounter{proofsec}
\expandafter\let\expandafter\oldproof\csname\string\proof\endcsname

\renewenvironment{proof}[1][\proofname]{\begin{oldproof}[#1]
    \setcounter{proofsec}{-1}}{\end{oldproof}}
\newcommand{\proofsec}{\refstepcounter{proofsec}\noindent\textbf{\theproofsec.}~}

\renewcommand{\epsilon}{\varepsilon}

\newcommand{\Alg}{\ccat{Alg}}

\newcommand*{\bigresto}[1]{\bigr|_{#1}}
\newcommand{\bigtensor}{\bigotimes}
\newcommand{\blank}{(\;)}
\newcommand{\Bord}{\ccat{Bord}}
\newcommand{\boundary}{\partial}

\newcommand{\cat}{\mathcal}
\newcommand{\Cat}{\ccat{Cat}}
\newcommand{\CAT}{\ccat{CAT}}

\newcommand{\CCat}{\closedcat{Cat}}
\newcommand{\CCAT}{\closedcat{CAT}}
\newcommand{\ccat}{\mathrm}
\newcommand{\CCocor}{\closedcat{Cocorr}}
\newcommand{\CCorr}{\closedcat{Corr}}
\newcommand{\closedcat}{\mathbf}

\DeclareMathOperator*{\colim}{colim}

\newcommand{\close}{\widebarv}
\newcommand{\Cocor}{\ccat{Cocorr}}
\newcommand{\codiag}{\nabla}
\newcommand{\Com}{\ccat{Com}}

\newcommand{\compose}{\circ}

\newcommand{\Corr}{\ccat{Corr}}

\newcommand{\diag}{\Delta}

\newcommand{\ddeloop}{\field{\deloop}}
\newcommand{\deloop}{B}

\renewcommand{\dot}{\bullet}

\DeclareMathOperator{\End}{End}
\renewcommand{\equiv}{\sim}

\newcommand{\equivwith}{\simeq}
\newcommand{\equivto}{\xrightarrow{\equiv}}

\newcommand{\field}{\mathbb}

\newcommand{\R}{\field{R}}

\newcommand{\Fin}{\ccat{Fin}}
\newcommand{\wasreteori}{\Theta}

\newcommand{\from}{\leftarrow}
\newcommand{\Fun}{\ccat{Fun}}
\newcommand{\FFun}{\closedcat{Fun}}
\newcommand{\Gpd}{\ccat{Gpd}}
\newcommand{\gy}{T}
\newcommand{\gyzen}{A}

\newcommand{\HH}{\mathrm{HH}}
\newcommand{\Hom}{\mathrm{Hom}}

\newcommand{\id}{\mathrm{id}}

\newcommand{\Ini}{\ccat{Init}}

\newcommand{\into}{\hookrightarrow}
\newcommand{\inv}{\check}
\newcommand*{\kakowanu}{}
\newcommand{\kara}{\varnothing}
\newcommand{\kocoprod}{\amalg}

\newcommand{\kore}{\textbf}

\newcommand{\lift}{\widetildev}
\newcommand{\longequivto}{\xlongrightarrow{\equiv}}
\newcommand{\longfrom}{\longleftarrow}
\newcommand{\longto}{\longrightarrow}

\newcommand*{\maekara}[2]{#1_{\rightarrow{#2}}}
\newcommand*{\maeoki}[2]{\vphantom{#2}#1\!{#2}}

\newcommand{\Map}{\mathrm{Map}}

\DeclareMathOperator{\Mul}{Mul}

\newcommand{\noloc}{\reflectsymbol{\colon}}

\DeclareMathOperator{\Ob}{Ob}

\newcommand{\op}{\mathrm{op}}

\newcommand{\Operad}{\ccat{Op}}

\newcommand{\Ord}{\ccat{Ord}}

\newcommand{\pr}{\mathrm{pr}}

\newcommand{\pt}{\mathrm{pt}}
\newcommand{\ptn}{\uparrow}
\newcommand{\ptp}{\downarrow}

\newcommand*{\reflectsymbol}[1]{\reflectbox{\ensuremath{#1}}}

\newcommand*{\resto}[1]{\mathclose{|}_{#1}}

\newcommand*{\sakihe}[2]{#1_{\leftarrow{#2}}}
\newcommand{\Set}{\ccat{Set}}

\newcommand{\Simp}{\mathbf{\Delta}}
\newcommand{\simp}{\Delta}

\DeclareMathOperator{\Spec}{Spec}

\newcommand{\sub}{\subset}

\newcommand{\ten}{*}
\newcommand{\tensor}{\otimes}
\newcommand{\Tensor}{\bigotimes}
\newcommand{\Theory}{\ccat{Th}}

\newcommand*{\tsketas}[2]
{#1\makebox[0pt][l]{$\vphantom{#1}#2$}}
\newcommand*{\keepheight}[2]{#1{#2}\vphantom{{#2}}}

\newcommand{\union}{\cup}
\newcommand{\Union}{\bigcup}

\newcommand{\unity}{{\boldsymbol{1}}}
\newcommand{\univ}{\mathbb{U}}
\newcommand{\univalg}{\univ}
\newcommand{\widebar}{\overline}
\newcommand*{\widebarv}[1]{\keepheight{\widebar}{#1}}

\newcommand*{\widetildev}[1]{\keepheight{\widetilde}{#1}}

\title{Higher theories of algebraic structures}
\author[Matsuoka, Takuo]{Takuo Matsuoka}
\email{motogeomtop@gmail.com}

\subjclass[2010]{Primary: 18D50;
Secondary:
18D05
, 18D35
, 57R56
, 18D20
, 18D10
, 18G55
}

\keywords{Theorization,
Categorification,
Lax structure,
Enrichment,
Grading,
Many objects\slash Color,
Coherence,
Higher associativity,
Multicategory%
}

\let\oldsection\section
\renewcommand*{\section}[1]{\oldsection{#1}\setcounter{subsection}{-1}\setcounter{subsubsection}{-1}\setcounter{equation}{-1}}
\let\oldsubsection\subsection
\renewcommand*{\subsection}[1]{\oldsubsection{#1}\setcounter{subsubsection}{-1}}

\begin{document}
\setcounter{section}{-1}

\begin{abstract}
The notion of (symmetric) coloured operad or ``multicategory'' can be
obtained from the notion of commutative algebra through a certain
general process which we call ``theorization'' (where our term comes
from an analogy with William Lawvere's notion of algebraic theory).
By exploiting the inductivity in the structure of higher
associativity, we obtain the notion of ``$n$-theory'' for every
integer $n\ge 0$, which inductively \emph{theorizes} $n$ times, the
notion of commutative algebra.
As a result, (coloured) morphism between $n$-theories is a ``graded''
and ``enriched'' generalization of ($n-1$)-theory.
The inductive
hierarchy of those \emph{higher theories} extends in particular, the
hierarchy of higher categories.
Indeed, theorization turns out to produce more general kinds of
structure than the process of
categorification in the sense of Louis Crane does.
In a part of low ``theoretic'' order of this hierarchy, graded and
enriched $1$- and $0$-theories vastly generalize symmetric, braided,
and many other kinds of enriched
multicategories and their algebras in various places.

We make various constructions of\slash with higher theories, and obtain some
fundamental notions and facts.
We also find iterated theorizations of more general kinds of algebraic
structure including (coloured) properad of Bruno Vallette
and various kinds of topological field theory (TFT).
We show that a ``TFT'' in the extended context can reflect a datum of
a very different type from a TFT in the conventional sense, despite
close formal similarity of the notions.

This work is intended to illustrate use of simple understanding of
higher coherence for associativity.
\end{abstract}

\maketitle 

\tableofcontents

\section{Introduction}
\label{sec:introduction}

\subsection{Higher theories}

\subsubsection{}

Among variations of the notion of operad, the symmetric, the planar
and the braided (see Fiedorowicz \cite{fiedoro}) versions are
particularly simple to describe, and are very commonly worked with.
Over an operad of one of these types, an \emph{algebra} can be
considered respectively
in a symmetric, associative and braided (see Joyal and Street
\cite{joyal-street}) monoidal category.

In general, each operad governs algebras over it, and this role is
important.
In other words, the notion of operad is important since it gives a way
to do \emph{universal algebra} in (e.g., symmetric) monoidal
categories.
For this role, we consider an operad as analogous to Lawvere's
\emph{algebraic theory} \cite{lawvere}.

By visiting the conceptual origin of the notion of operad, one finds
that the notion of symmetric operad arises naturally through a certain
process, which we call ``theorization'' (inspired by Lawvere's
notion), from the notion of commutative algebra.
Moreover, the same process can be started from the notion of
$\cat{U}$-algebra for any operad $\cat{U}$ in sets or groupoids, to
produce a new kind of algebraic structure which we call
``$\cat{U}$-graded'' operad (Section \ref{sec:graded-operad}).
It turns out that planar and braided operads are $\cat{U}$-graded
operads where $\cat{U}=E_1, E_2$ respectively.

The notion of $\cat{U}$-graded operad in groupoids has a deceptively
simple description, namely, a $\cat{U}$-graded operad in groupoids can
be described
as an operad $\cat{X}$ in groupoids equipped with a morphism
$P\colon\cat{X}\to\cat{U}$.
(This $\cat{X}$ equipped with $P$ corresponds to a $\cat{U}$-graded
operad in \emph{sets} if the maps
induced by $P$ on the groupoids of operations have everywhere,
homotopy fibre with a discrete homotopy type.
The details will be discussed in Remark \ref{rem:graded-by-groupoid}.)
This might hide the notion of theorization from a non-obsessed mind.
However, we have come to think that theorization is an important
notion.

One purpose of this long introductory section is to introduce the
notion of theorization, which will be a generalization with its own
mathematical content, of the notion of \emph{categorification} in the
sense of Crane \cite{crane-frenke,crane} (about which the most
influential pioneer may have been Grothendieck), but will be at least
as informal a notion as categorification.
Even though a very precise understanding of the notion of theorization
is not technically necessary for the body of the article, at least a
rough understanding will be essential for understanding the ideas of
our work.
In the body, we shall use the language of theorization to navigate the
reader through ideas.

Let us, however, start with a sketch of what we actually do in this
work.
After that, the main purpose of this introduction will be to introduce
the idea of theorization, and describe more topics to be covered in
the body.

\subsubsection{}

In this work, we introduce and study `higher order' generalizations of
(coloured) operads, which we call ``higher theories'' of algebras.
Higher theories will be obtained by iterating the process of
theorization starting from the notion of coloured operad.
Introduction of these objects leads to (among other things) a
framework for a natural explanation and a vast generalization of the
fact which we formulate below as Proposition
\ref{prop:enriching-multicategory}.
Let us describe it now.

The proposition will be about places where the notion of
$\cat{U}$-graded operad can be enriched, but the reader may assume
that $\cat{U}=E_1$ or $E_2$ (see above).
As another notice, colours in operads will not play an essential
role for a while, so we shall consider just uncoloured operads
everywhere till we start taking colours explicitly into
consideration.

For a symmetric monoidal category $\cat{A}$, let us denote by
$\Operad_\cat{U}(\cat{A})$ the category of $\cat{U}$-graded operads in
$\cat{A}$.
$\Operad_\cat{U}$ is a category-valued functor on the $2$-category
$\Alg_\Com(\Cat)$ of symmetric monoidal categories, where $\Cat$
denotes the $2$-category of categories (with a fixed limit for size),
equipped with the symmetric monoidal structure given by the direct
product operations.

For an operad $\cat{U}$, let us mean by an \kore{($\cat{U}\tensor
  E_1$)-monoidal} category, an associative monoidal object in the
$2$-category of $\cat{U}$-monoidal categories, or equivalently, a
$\cat{U}$-monoidal object in the $2$-category of associative monoidal
categories.
Note that there is a forgetful functor
$\Alg_\Com\bigl(\Cat\bigr)\to\Alg_\cat{U}\bigl(\Alg_\Com(\Cat)\bigr)\to\Alg_\cat{U}\bigl(\Alg_{E_1}(\Cat)\bigr)$
from the symmetric monoidal categories to ($\cat{U}\tensor
E_1$)-monoidal categories, where we have used the canonical functor
$\cat{C}\to\Alg_\cat{U}(\cat{C})$ exsisting for every
\emph{coCartesian} symmetric monoidal $2$-category $\cat{C}$ (namely,
a $2$-category $\cat{C}$ closed under the finite coproducts, made
symmetric monoidal by the finite coproduct operations) obtained by
letting every operation on a given object, say $X$, of $\cat{C}$, be
the codiagonal map of $X$.

The formulation of the proposition is as follows.
(See Remark \ref{rem:dimension-for-example} for a technical point.)

\begin{proposition}\label{prop:enriching-multicategory}
For every symmetric operad $\cat{U}$, the functor $\Operad_\cat{U}$ on
$\Alg_\Com(\Cat)$ has an extention to
$\Alg_\cat{U}\bigl(\Alg_{E_1}(\Cat)\bigr)$ (in a manner which is
functorial in $\cat{U}$).
\end{proposition}

The meaning of Proposition is that there is a natural notion of
$\cat{U}$-graded operad in every ($\cat{U}\tensor E_1$)-monoidal
category, such that the notion of
$\cat{U}$-graded operad in a symmetric monoidal category $\cat{A}$
coincides naturally with the notion of $\cat{U}$-graded operad in the
($\cat{U}\tensor E_1$)-monoidal category underlying $\cat{A}$.
Our definition will generalize the familiar notions of
\begin{itemize}
\item associative algebra in a \emph{associative} monoidal category,
\item planar operad in a \emph{braided} monoidal category,
\item braided operad in a \emph{$E_3$-monoidal} infinity $1$-category
\end{itemize}
(in addition to vacuously, the notion of symmetric operad in a
symmetric monoidal category).

\begin{remark}\label{rem:dimension-for-example}
In order to actually have these examples, Proposition needs to be
interpreted in the framework of sufficiently high dimensional category
theory.
(Infinity $1$-category theory is sufficient.)
However, let us not emphasize this technical point in this introduction,
even though our work will eventually be about higher category theory.
\end{remark}

\begin{remark}
It seems to be a reasonable guess that the notion of $\cat{U}$-graded
multicategory
had been known before our work even in the form enriched in a
symmetric monoidal category, and if it had in the enriched form, then
we expect Proposition~\ref{prop:enriching-multicategory}
to have also been known.
We simply have failed to confirm this from the literature.
\end{remark}

In fact, we introduce in this work much more general notion of
``grading'', generally for higher theories, and find quite general
but natural places where the notions of graded higher theory can be
enriched.
An explanation of Proposition
\ref{prop:enriching-multicategory} from the general perspective to
be so acquired, will be given in Section
\ref{sec:enrichment-explanation}.
In fact, all of these will result from extremely simple ideas, which
we would like to describe with their main consequences.

\subsubsection{}
\label{sec:theoization-sketch}
Our starting point is the idea that an operad and more generally, a
\emph{coloured} operad or ``\emph{multicategory}'' (see Lambek
\cite{lambek} or Section \ref{sec:coloured-lax-algebra}) is analogous
to an algebraic theory for its role of governing algebras over it.
We extend this by defining, for every integer $n\ge 0$, a notion of
``$n$-theory'', where a $0$-theory is an algebra (commutative etc.), a
$1$-theory will
be a multicategory (symmetric etc.), and, for $n\ge 2$, each
$n$-theory will come with a natural notion of algebra over it, in such
a way that an ($n-1$)-theory coincides
precisely with an algebra over the terminal $n$-theory, generalizing
from the case $n=1$, the fact that, e.g., a commutative algebra is an
algebra over the terminal symmeric multicategory $\Com$.

\begin{remark}
While algebraic theory in Lawvere's sense is about a kind of algebraic
structure which makes sense in any \emph{Cartesian} symmetric monoidal
category, the notion of algebra over an $n$-theory can be enriched (in
particular) in any, e.g., symmetric, monoidal category (Definition
\ref{def:algebra}).
In this work, we only consider kinds of algebraic structure which are
definable at least in every symmetric monoidal category.
\end{remark}

The stated objective will be achieved by defining an $n$-theory
inductively as a \emph{theorized}
form of an ($n-1$)-theory, where the theorization we consider is a
``coloured'' version of it.
Theorization in the coloured sense, of commutative algebra, will be
symmetric multicategory.

Let us describe what about the notion of multicategory we would like
to generalize.

We first note that the notion of symmetric monoidal category is a
\emph{categorification} of the notion of commutative algebra, and the
process of categorification plays here the role of enabling us to
define the notion of commutative algebra \emph{in an} arbitrary
\emph{symmetric monoidal category}.
A commutative algebra in a symmetric monoidal category $\cat{A}$ is
simply a lax symmetric monoidal functor $\unity\to\cat{A}$, where
$\unity$ denotes the unit symmetric monoidal category.

Next, the notion of multicategory generalizes the notion of symmetric
monoidal category in the sense that there is a functor which
associates to a given symmetric monoidal category $\cat{A}$, a
multicategory $\wasreteori\cat{A}$ from the structure of which
$\cat{A}$ can be recovered (up to canonical equivalence).
Indeed, we can define $\wasreteori\cat{A}$ as the multicategory in
which
\begin{itemize}
\item an object is an object of $\cat{A}$,
\item a multimap $X\to Y$, where $X$ is a family of objects indexed by
  a finite set, say $S$, is a map $\Tensor_SX\to Y$ in $\cat{A}$,
\item multimaps compose by composition in $\cat{A}$ in the obvious
  manner.
\end{itemize}
Moreover, the notion of commutative algebra makes sense also in a
multicategory $\cat{M}$ so that, in the case
$\cat{M}=\wasreteori\cat{A}$, we get back the notion of commutative
algebra in $\cat{A}$.
Indeed, a commutative algebra \emph{in $\cat{M}$} is a functor
$\Com\to\cat{M}$ of multicategories.
We note that the notion of functor of multicategories is `primitive'
in that it is not a lax notion of morphism between categorified
structures, as the notion of lax symmetric monoidal functor is.

Now, a standard definition is that, for a multicategory $\cat{U}$, a
\emph{$\cat{U}$-algebra} in $\cat{M}$ is a functor
$\cat{U}\to\cat{M}$ of multicategories, generalizing the notion of
commutative algebra in $\cat{M}$.
Even in the special case where $\cat{M}=\wasreteori\cat{A}$, this
gives a vast generalization of the notion of commutative algebra in
$\cat{A}$.
The notion for a general $\cat{M}$, is a version which is
\emph{enriched} in a more general manner.

Now we would like to get back to the general context, and sketch the
relation between the notions, of $n$-theory and of ($n-1$)-theory.
(The generalization of what has been discussed so far, will be as
summarized by the diagrams \eqref{eq:theorization} and
\eqref{eq:enrichment} below, of which
the case already discussed is the case $n=1$.)

Firstly, there is an appropriate categorification of the notion of
($n-1$)-theory, which enables us to generalize the notion of
($n-1$)-theory to the notion of ($n-1$)-theory \emph{in a categorified
  ($n-1$)-theory}.
An \kore{($n-1$)-theory in a categorified ($n-1$)-theory $\cat{C}$}
will be a lax functor $\unity^{n-1}_\Com\to\cat{C}$ (possibly with
``colours'' in a suitable sense, which is not important in this
introduction, but is to be commented on after Theorem
\ref{thm:functor-of-theorization-intro} below) of categorified
($n-1$)-theories,
where $\unity^{n-1}_\Com$ denotes the terminal ($n-1$)-theory.

Next, the notion of $n$-theory will generalize the notion of
categorified ($n-1$)-theory in the similar manner as the manner how
the notion of
multicategory generalized the notion of symmetric monoidal category.
(General discussions of this will be given in Sections
\ref{sec:theorization-general}, \ref{sec:basic-construction}.)
Moreover, an ($n-1$)-theory in a categorified ($n-1$)-theory $\cat{C}$
will be equivalent as a datum to a (possibly coloured) functor
$\unity^n_\Com\to\wasreteori\cat{C}$, where the target here is the
$n$-theory corresponding to $\cat{C}$.

In fact, $\unity^n_\Com$ is equivalent to
$\wasreteori\unity^{n-1}_\Com$ in a suitable sense, and we shall have
the following.

\begin{theorem}[Theorem \ref{thm:functor-of-theorization}]
\label{thm:functor-of-theorization-intro}
Let $\cat{B}$ and $\cat{C}$ be categorified ($n-1$)-theories.
Then a functor $\wasreteori\cat{B}\to\wasreteori\cat{C}$ of
uncategorified $n$-theories, is
equivalent as a datum to a \textbf{lax} functor $\cat{B}\to\cat{C}$.
\end{theorem}

The mentioned ``colouring'' of a (lax) functor (which the reader may
ignore in this introduction) means the replacement of the source
$\unity^{n-1}_\Com$ by a suitable $\cat{B}$ (Section
\ref{sec:colour-system}).

The notion of coloured functor $\unity^n_\Com\to\cat{V}$ of
$n$-theories thus generalizes the notion of ($n-1$)-theory, to the
similar notion \emph{in an $n$-theory} $\cat{V}$.
In general, given a specific kind of algebraic structure such as
``($n-1$)-theory'' here, a ``theorized'' instance of that kind of
structure, such as $\cat{V}$ here, will be a general and natural place
along which we hope to enrich the kind of structure.

Functor of higher theories will again be a ``primitive'' notion,
unlike lax functor of categorified higher theories.

Since the notion of $n$-theory generalizes the notion of categorified
($n-1$)-theory, we inductively find within the hierarchy of
$n$-theories as $n$ varies, the hierarchy of $n$-categories or
iterated categorifications of category, as in fact a very small part
of it.
Quite a variety of other hierarchies of iterated categorifications,
such as the hierarchy of symmetric monoidal $n$-cateogories, of
operads in $n$-categories,
and so on, also form very small parts of the same hierarchy.

\begin{remark}
It has not been clear whether there is an analogous hierarchy starting
from Lawvere theory.
Possible iterated \emph{categorifications} of Lawvere's notion are
algebraic theories enriched in $n$-categories, but we are looking for
a larger hierarchy than iterated categorifications.
\end{remark}

\subsubsection{}
\label{sec:grading-intro}

Let us sketch the basic structures which we obtain further.

Let $n\ge 1$, and let $\cat{U}$ and $\cat{V}$ be \emph{unenriched}
$n$-theories, i.e., those enriched in sets (or in groupoids).
Then it is appropriate to refer to a (coloured) functor
$\cat{U}\to\cat{V}$ as a \kore{$\cat{U}$-algebra in $\cat{V}$}.
(We have already mentioned this terminology for $n=1$, where we did
not want ``colours''.
For a general $n$, we shall treat an algebra over an enriched
$n$-theory using a kind of Day convolution, to handle the colours in a
simple manner.
See Section \ref{sec:convolution}.)
In the case where $\cat{U}$ is terminal, this notion of
$\cat{U}$-algebra was the (enriched) notion of ($n-1$)-theory.
Thus, we have obtained as various kinds of structure which are
analogous to the structure of an ($n-1$)-theory, as the choices there
are of an $n$-theory $\cat{U}$.

The situation can be summarized by the diagram
\begin{equation}\label{eq:theorization}
\begin{tikzcd}
\text{$n$-theory}
&\ni
&\cat{U}\arrow{d}{\text{control}}
&\tsketas{\unity}{^n_\Com}
\arrow{d}{\text{control}}\\
\text{($n-1$)-theory}\arrow[maps
to,dashed]{rr}{\substack{\text{``grade'' by $\cat{U}$}\\
    \text{(see below)}}}
\arrow[maps to,dashed]{u}{\text{theorize}}
&&\text{$\cat{U}$-algebra}&\text{%
($n-1$)-theory\kokoni{.}}
\end{tikzcd}
\end{equation}

To be detailed, various enriched notions are related as depicted in
the following diagram, where $\cat{U}$, $\cat{V}$ vary through
$n$-theories, and $\cat{C}$ varies through categorified
($n-1$)-theories.
(For example, the ``unenriched'' ($n-1$)-theory is where
$\cat{C}=\ddeloop^{n-1}\Set$ for a certain construction $\ddeloop$
(Section \ref{sec:delooping}), where $\Set$ denotes the Cartesian
symmetric monoidal category of sets.)
\begin{equation}\label{eq:enrichment}
\begin{tikzcd}[column sep=large]
\substack{\displaystyle\text{Functor}\\
  \displaystyle\cat{U}\to\cat{V}}
&\substack{\displaystyle\text{$\cat{U}$-algebra}\\
  \displaystyle\text{in $\cat{V}$}}
\arrow[maps to]{l}[swap]{\substack{\text{forbid}\\
    \text{colours}}}
\arrow[maps to]{d}[swap]{\cat{V}=\wasreteori\cat{C}}
\arrow[maps to]{r}{\cat{U}=\unity^n_\Com}
&\substack{\displaystyle\text{($n-1$)-theory}\\
  \displaystyle\text{in $\cat{V}$}}
\arrow[maps to]{d}[swap]{\cat{V}=\wasreteori\cat{C}}\\
&\substack{\displaystyle\text{$\cat{U}$-algebra}\\
  \displaystyle\text{in $\cat{C}$}}
\arrow[maps to]{r}[swap]{\cat{U}=\unity^n_\Com}
&\substack{\displaystyle\text{($n-1$)-theory}\\
  \displaystyle\text{in $\cat{C}$}}
\arrow[maps to]{r}{\substack{\text{forbid}\\
    \text{colours}}}
&\substack{\displaystyle\text{Lax functor}\\
  \displaystyle\unity^{n-1}_\Com\to\cat{C}.}
\end{tikzcd}
\end{equation}

On the other hand, there is a notion of \kore{$\cat{U}$-graded
  $n$-theory}, which theorizes the notion of $\cat{U}$-algebra, so the
relation between the former and the latter notions is analogous to
(and in this particular case, generalizes in
fact; see below) the relation between the notions, of $n$-theory and
of ($n-1$)-theory, described in Section \ref{sec:theoization-sketch}.
The term ``graded $n$-theory'' comes from the following proposition.
The previous, un-``graded'' version of an $n$-theory will be called a
\kore{symmetric} $n$-theory.

\begin{proposition}[Proposition \ref{prop:graded-is-overlying}]
\label{prop:graded-is-overlying-intro}
An unenriched $\cat{U}$-graded $n$-theory is equivalent as a datum to
an unenriched symmetric $n$-theory $\cat{X}$ equipped with a functor
$\cat{X}\to\cat{U}$ of symmetric $n$-theories.
\end{proposition}

In the case where $\cat{U}$ is terminal, the notion of
$\cat{U}$-graded $n$-theory, which theorizes the notion of
$\unity^n_\Com$-algebra $=\text{($n-1$)-theory}$, coincides indeed
with the notion of symmetric $n$-theory.

One might think that, for a general $\cat{U}$, it makes sense to refer
to a $\cat{U}$-algebra also as a
\kore{$\cat{U}$-graded ($n-1$)-theory}.
In fact, there are even ``lower'' order notions of graded theory as
follows.

If $n\ge 2$, then as there is a notion of algebra over an
($n-1$)-theory, there is also a notion of algebra over a \emph{monoid}
$\cat{X}$ over an $n$-theory $\cat{U}$, where by a
\kore{monoid}, we mean an algebra enriched in sets.
In the case where $\cat{X}$ is the terminal monoid
$\unity^{n-1}_\cat{U}$, the notion of $\cat{U}$-algebra theorizes the
notion of $\cat{X}$-algebra, generalizing the manner how the notion of
($n-1$)-theory theorized the notion of ($n-2$)-theory.
It is appropriate to refer to $\unity^{n-1}_\cat{U}$-algebra also as a
\kore{$\cat{U}$-graded ($n-2$)-theory}.

We can further iterate a similar construction of a new notion, and
obtain for every integer $m$ such that $0\le m\le n-2$, the notion of
\emph{$\cat{U}$-graded $m$-theory}, which is theorized by the
notion of $\cat{U}$-graded ($m+1$)-theory.

There is also a notion of \kore{$\cat{U}$-graded ($n+1$)-theory},
which theorizes the notion of $\cat{U}$-graded $n$-theory.
We in fact obtain the following fundamental results.

\begin{theorem}[Theorems \ref{thm:graded-theory-is-monoid},
  \ref{thm:grade-theorized-theory}]
\label{thm:graded-theory-is-monoid-intro}
A $\cat{U}$-graded $n$-theory is equivalent as a datum to an algebra
over the ($n+1$)-theory $\wasreteori\cat{U}$.
A $\cat{U}$-graded ($n+1$)-theory is equivalent as
a datum to a $\wasreteori\cat{U}$-graded ($n+1$)-theory.
\end{theorem}

The notion of $\cat{U}$-graded ($n+1$)-theory can further be theorized
iteratively, but the resulting notion of \emph{$\cat{U}$-graded
  $m$-theory} for $m\ge n+2$ coincides with the notion of
$\wasreteori^m_n\cat{U}$-graded $m$-theory, where
$\wasreteori^m_n\cat{U}$ denotes the $m$-theory obtained by applying
the construction $\wasreteori$ $m-n$ times on $\cat{U}$ (Section
\ref{sec:forget-to-theorization}).
(Theorem also suggests that $\cat{U}$-graded $m$-theory in the case
$n=0$ should mean $\wasreteori\cat{U}$-graded $m$-theory for $m\ge
0$.)

In the case where $\cat{U}$ is terminal, the hierarchy of
$\cat{U}$-graded higher theories coincides with the original hierarchy
of symmetric higher theories.

We further consider higher theories graded by a graded theory, and so
on, and shall find out that this does not add much to what have
already been described.
See Section \ref{sec:grade-by-graded-intro} for some more on this.

In this work, we shall investigate relationship among various
mathematical structures related to these objects, do and investigate
various fundamental constructions, as well as consider a
generalization and its relation to other subjects of mathematics.
See Section \ref{sec:further-development}.

\subsubsection{}
\label{sec:enrichment-explanation}

The notion of $\cat{U}$-graded $1$-theory in a $\cat{U}$-graded
$2$-theory gives a generalization of Proposition
\ref{prop:enriching-multicategory}.
Let us see this.
(The discussion below will be slightly imprecise even though
essentially correct.
A more technical account can be found in Section
\ref{sec:enriched-theory}.)

Given a ($\cat{U}\tensor E_1$)-monoidal category $\cat{A}$, one
can `categorically deloop' $\cat{A}$ using the $E_1$-monoidal
structure (see B\'enabou \cite[2.2]{benabou} or the review in Section
\ref{sec:coloured-lax-algebra}) to obtain a $\cat{U}$-monoidal
$2$-category $\deloop\cat{A}$, and hence a $\cat{U}$-graded
$1$-theory $\wasreteori\deloop\cat{A}$ enriched in categories.
Since this is a categorified $\cat{U}$-graded $1$-theory, the notion
of $\cat{U}$-graded operad in $\wasreteori\deloop\cat{A}$ makes sense,
which we declare to be the notion of \emph{$\cat{U}$-graded operad in
  $\cat{A}$}.
This naturally generalizes the notion with the same name from the case
where $\cat{A}$ is symmetric monoidal.
Moreover, it is easy to check when $\cat{U}$ is one of the most
familiar operads, that this coincides with the usual notion.
However, this notion of operad, including the coloured version of it,
is nothing but the notion of $1$-theory in the $\cat{U}$-graded
$2$-theory $\wasreteori(\wasreteori\deloop\cat{A})$.

\subsubsection{}

Our method for iterative theorization will be by a technology of
producing
from a given kind of associative operation, a new kind of associative
operation, which is based on fundamental understanding of the higher
structure of associativity.
See Section \ref{sec:inductivity}.
The author indeed expects a hierarchy of iterated theorizations to
exist starting from a kind of algebraic structure which can be
expressed as defined by an ``associative'' operation, much more
generally than has been discussed so far.
Our construction of the hierarchy of $n$-theories, and its graded
generalization, may be showing the meaningfulness of a general
question of the existence of iterated theorizations for kinds of
``algebraic'' structure, and this may be our deepest contribution at
the conceptual level.

Indeed, we shall consider in Section \ref{sec:generalization}, a
modest generalization of symmetric higher theories, which are obtained
by iteratively theorizing some algebraic structures in which the
operations may have multiple inputs \emph{and} multiple outputs, such
as various versions of topological field theories.
See Section \ref{sec:further-list} for brief descriptions of the
examples to be discussed in this work.

\subsubsection{}
For the rest of this introduction, we shall describe the details of
the idea of theorization, and more results of this work, and then give
the outline of the body of the paper, and some notes for the reader
for proceeding.

\subsection{The conceptual origin of the notion of operad}
\label{sec:theorization}

\subsubsection{}
 
Theorization will be a process which produces a new kind
of algebraic structure from a given kind.
In order to start a discussion of the idea of theorization, we would
like to be able to talk about \emph{kinds} of algebraic structure.

An example of a kind of structure definable in a symmetric monoidal
category, is a symmetric monoidal functor from a fixed symmetric
monoidal category, say $\cat{B}$.
Thus, let us mean by a \kore{$\cat{B}$-algebra} in a symmetric
monoidal category $\cat{A}$,
simply a symmetric monoidal functor $F\colon\cat{B}\to\cat{A}$.
It is a `representation' of $\cat{B}$ in $\cat{A}$ (or an
$\cat{A}$-valued point of the `affine scheme' $\Spec\cat{B}$).

Concretely, by describing the structure of $\cat{B}$ using a
collection of generating objects and generating maps (as well as
decomposition of each of the source and the target of every generating
map into a monoidal product of generating objects), one may obtain a
presentation of the form of the structure of a $\cat{B}$-algebra in
terms of structure maps and equations satisfied by the structure
maps.
For example, the coCartesian symmetric monoidal category
$\cat{B}=\Fin$ of finite sets, is generated under the symmetric
monoidal multiplication
operations $\tensor=\kocoprod$, by the terminal object $\ten$, and one
map $\codiag_S\colon\ten^{\tensor S}\to\ten$ for each finite set $S$.
It follows that the datum of a symmetric monoidal functor
$F\colon\Fin\to\cat{A}$ can be described as the object $A=F(\ten)$
of $\cat{A}$ equipped with one operation $F(\codiag_S)\colon
A^{\tensor S}=F(\ten^{\tensor S})\to A$ for each finite set $S$,
satisfying suitable equations resulting from the relations one has in
$\Fin$.
Thus, we have obtained a presentation of the form of the structure of
a $\Fin$-algebra, as the form of datum for defining a commutative
algebra.

For a general $\cat{B}$, if objects of $\cat{B}$ are generated under
the monoidal multiplication by a family
$b=(b_\lambda)_{\lambda\in\Lambda}$ of objects, then a
$\cat{B}$-algebra defined by a symmetric monoidal functor
$F\colon\cat{B}\to\cat{A}$, can be considered similarly as structured
on the family $Fb$ of objects of $\cat{A}$, by structure maps
satisfying equations imposed by the structure of $\cat{B}$.
For example, the universal $\cat{B}$-algebra defined by
$\id\colon\cat{B}\to\cat{B}$, is structured on the family $b$ of
objects of $\cat{B}$, where the
structure maps will be the chosen generating maps of $\cat{B}$.

Algebra over a symmetric monoidal category in our sense, is simply the
most obvious formalization of kind of structure which can be presented
as defined by structure maps satisfying some specific equations.
Lawvere's theory is based on this idea.
Indeed, a \emph{multi-sorted}, i.e., ``coloured'', Lawvere theory is
essentially a Cartesian symmetric monoidal category $\cat{B}$ which is
given a nice collection of generating objects.
A \emph{PROP} \cite{mac-algebra} or ``\emph{category of operators}''
\cite{homotopy-everything}, with colours \cite{homotopy-invariant}, is
similar.

An algebra over a multicategory is also covered.
Indeed, given a multicategory $\cat{U}$, one can freely generate from
it a symmetric monoidal category, say $L\cat{U}$, so a
$\cat{U}$-algebra in a symmetric monoidal category $\cat{A}$ will
be equivalent as a datum to an $L\cat{U}$-algebra in $\cat{A}$.

Even though we are not particularly interested in the notion of
algebra over a general symmetric monoidal category, this can be the
starting point for our purpose of finding kinds of
algebraic structure which generalize nicely.
For example, recall that the basic idea of categorification is that a
\emph{categorification} of a certain kind of algebraic structure, is a
kind of structure on category obtained by replacing structure maps by
functors, and structural equations by suitably coherent isomorphisms,
forming a part of the structure.
We have a canonical categorification of $\cat{B}$-algebra, which we
shall call a \kore{$\cat{B}$-monoidal category}, and it is simply a
symmetric monoidal functor $\cat{B}\to\Cat$, where $\Cat$ denotes the
Cartesian symmetric monoidal category enriched in groupoids, of
categories (with a fix limit for size), where for
$\cat{X},\cat{Y}\in\Cat$, we let
$\Map_\Cat(\cat{X},\cat{Y})$ be the groupoid formed by functors
$\cat{X}\to\cat{Y}$ and
isomorphisms between them.

\begin{remark}\label{rem:groupoid-terminology}
This is a technical remark. 

Here and everywhere else in this introduction, a functor which we
consider
to a category enriched in groupoids (or in categories) is a functor in
the usual ``weakened'' sense (which is sometimes called a
\emph{pseudo}-functor, see Grothendieck \cite[Section 8]{sga1-6},
\cite{gro-semi-bour}).
Even though we shall not need to look into the details of this till we
enter the body, a symmetric monoidal structure on such a functor can
also be defined in an appropriate manner.

In fact, it should be understood that every categorical term in this
introduction is used in the similarly appropriate sense when there is
enrichment of the relevant categorical structures in groupoids or
categories, where erichment itself should be understood to be done in
the standard ``weakened'' manner.
See B\'enabou \cite{benabou}.
(The reader who is comfortable with homotopy theory may instead
replace all sets\slash groupoids with infinity groupoids, and understand
everything as enriched in infinity groupoids, and this will trivialize
the process of categorification since infinity $1$-categories (of size
up to a fixed limit) are already forming a (larger) infinity
$1$-category.)
\end{remark}

However, we notify the reader of a circularity here.
Namely, we have used one particular categorification of the notion of
commutative algebra, i.e., the notion of symmetric monoidal category,
to categorify kinds of structure which are similar to commutative
algebra.
Even though the result obtained is not bad, one may not be able to
expect that the same framework would also be the most useful for
categorifying very different kinds of structure.

For example, for categorification of the notion of multicategory, a
method which takes account of the categorical dimensionality appears
to lead to a cleaner and less redundant presentation of the result
than the method of reformulating the notion of multicategory as
algebra over a symmetric monoidal category, and then applying the
previous definition.
(It appears simpler to treat a multicategory as an algebra over a
categorically \emph{$2$-dimensional} algebraic structure, e.g., the
terminal ``$2$-theory'', than to treat it as an algebra over a
multicategory or a symmetric monoidal
category, which can naturally be seen as $1$-dimensional structures.)

Therefore, the general idea which we have described of a kind of
algebraic structure and its categorification, seems more important for
practical purposes, than precise but limited formulations of the
notions in particular contexts, such as $\cat{B}$-algebra and
$\cat{B}$-monoidal category.
Nevertheless, we hope that the examples above was clarifying on our
view on algebraic structures.

Finally, we remark that, in order to consider a categorification of
the notion of $\cat{B}$-algebra, the target of the symmetric monoidal
functor on $\cat{B}$ to define a categorified structure, did not need
to be $\Cat$.
Namely, the notion of symmetric monoidal functor $\cat{B}\to\cat{A}$,
where $\cat{A}$ is \emph{any} symmetric monoidal category enriched in
groupiods, formalizes the idea of replacing structural equations for
the structure of a $\cat{B}$-algebra (in some presentation of the
notion) by coherent isomorphisms.

Therefore, it seems reasonable to expect in general, that a meaningful
categorification of a kind of structure definable in a symmetric
monoidal category, should be a kind of structure definable in any
symmetric monoidal category $\cat{A}$ enriched in groupoids.
In fact, the right way to consider categorification is perhaps as
about enrichment in groupoids, and the resulting weakening in a
coherent manner, of the structure.

While the expectation above is indeed fulfilled in the concrete cases
which we consider in this work, we shall keep concentrating on the
case $\cat{A}=\Cat$ of categorified structures for the time being,
since this will keep things simpler.
One relation between the mentioned general form of categorification
and the idea of theorization, will be seen in Section
\ref{sec:theorization-general}.
On the other hand, for the kinds of structure which we theorize in
this work, the general form of categorification can be understood in
any case, as a specific kind of structure residing in a suitably
associated theorized structure, leaving us no need to consider more
general situation than $\cat{A}=\Cat$ in this early stage.
See e.g., Corollary \ref{cor:forget-to-theorization} (or Theorem \ref{thm:functor-of-theorization-intro}).
For example, the case ``$n=0$'' of this result applies to the
categorified form of the notion of $\cat{B}$-algebra.

\subsubsection{}
\label{sec:algebra-to-operad}

In order to get to the idea of theorization, we first recall that a
basic feature expected of the categorified structure is that, if a
category $\cat{C}$ is equipped with a categorified form of a certain
kind of algebraic structure, then the original, uncategorified form of
the same structure should naturally make sense in $\cat{C}$.
For example, if $\cat{C}$ is given a monoidal structure over a
symmetric monoidal category $\cat{B}$, then a ``$\cat{B}$-algebra'' in
$\cat{C}$ means a lax $\cat{B}$-monoidal functor $\unity\to\cat{C}$,
where $\unity$ denotes the unit $\cat{B}$-monoidal category.

However, for some kinds of algebraic structure, we have more general
instances of this phenomenon.
In a symmetric monoidal category for example, the notion of algebra
makes sense over any symmetric operad or multicategory, and the same
moreover makes sense also \emph{in} any symmetric multicategory.
Indeed, an algebra over a symmetric multicategory $\cat{U}$ in a
symmetric multicategory $\cat{V}$, is simply a morphism
$\cat{U}\to\cat{V}$.
Similarly, the notion of algebra over any planar multicategory makes
sense in any
associative monoidal category, and more generally, in any planar
multicategory in the same manner.

While the notion of associative monoidal category categorifies the
notion of associative algebra, the process of theorization, which will
be more general than the process of categorification, will produce the
notion of planar multicategory from the notion of associative algebra,
and symmetric multicategory from commutative algebra.
In general, theorization will produce from a given kind of algebraic
structure, a new kind of algebraic structure generalizing its
categorification, in such a manner that the original notion of algebra
reduces to the notion of algebra over the terminal one among the
theorized objects (meaning symmetric multicategories, for commutative
algebras, so generalizing the simple fact that an commutative algebra
is an algebra over the terminal symmetric multicategory).

Let us thus recall how one may naturally arrive at the notion of
symmetric operad, starting from the notion of commutative algebra (and
we are suggesting that the same procedure will produce the notion
of planar operad from the notion of associative algebra, for example).
Specifically, let us try to find the notion of symmetric operad (in
sets) out of
the desire of generalizing the notion of commutative algebra to the
notions of certain other kinds of algebra which makes sense in a
symmetric monoidal category.
Indeed, one of the most important roles of a multicategory is
definitely the role of governing algebras over it.

The way how we generalized the notion of commutative algebra is as
follows.
Recall, as we have already seen, that the structure of a commutative
algebra on an object $A$ of a symmetric monoidal category, was given
by a single $S$-ary
operation $A^{\tensor S}\to A$ for every finite set $S$ which,
collected over all $S$, had appropriate consistency.
We get a generalization of this by allowing not just a single $S$-ary
operation, but a \emph{family} of $S$-ary operations parametrized by a
set prescribed for $S$.
This ``set of $S$-ary operations'' for each $S$, is the first bit of
the datum defining an operad in sets.
Having this, we next would like to compose these operations just as we
can compose multiplication operations of a commutative algebra,
and the composition should have appropriate consistency.
A symmetric multicategory is simply a more general version of this,
with many objects, or ``colours''.
(We shall take a look at colours in a theorized structure in Section
\ref{sec:theorization-formulation}.)

\begin{remark}
In this formulation of an operad, part of the composition structure
makes the set of $S$-ary operations functorial with respect to
bijections of ``$S$''.
This gives the ``action of the symmetric group'' in another common
formulation of an operad.
\end{remark}

A similar procedure can be imagined once a kind of ``algebraic''
structure in a broad sense is specified as a specific kind of
system of operations, in place of commutative or associative algebra.
Inspired by Lawvere's notion of algebraic theory, we call a
multicategory also a (\kore{symmetric}) \kore{$1$-theory}, and
then generally call \kore{theorization}, a process similar to the
process above through which we have obtained $1$-theories from the
notion of commutative algebra.
The result of such a process will also be called a \kore{theorization}.
Thus, the notion of $1$-theory is a \emph{theorization} of the notion
of commutative algebra.
We shall see in Sections \ref{sec:theorization-general} and
\ref{sec:basic-construction}, how the process of theorization indeed
generalizes the process of categorification.

\begin{remark}\label{rem:presentation-vs-unenrich}
Recall that operad was a kind of structure which made sense in any
symmetric monoidal category, the meaning of which for us was that the
form of datum to define an operad, could be presented in terms of
structure maps and structural equations.
In general, it seems reasonable to expect that a meaningful
theorization of a given kind of algebraic structure should have a
similar presentation, and in particular, should make sense in any
symmetric monoidal category $\cat{A}$.
This ability of presentation will be important when we would like to
theorize a theorized kind of structure once again, even though we
shall till that time, stick normally to the case where $\cat{A}$ is
the Cartesian symmetric monoidal category $\Set$ of sets, in
order to keep our exposition simpler.
\end{remark}

\subsection{Theorization of algebra}
\label{sec:theorize-algebra}

\subsubsection{}
\label{sec:graded-operad}

As the simplest example of a theorization process next to the one
which we have seen in the previous section, let us consider
theorization of the
notion of $\cat{U}$-algebra for a symmetric operad $\cat{U}$ in sets
(see Remark \ref{rem:graded-by-groupoid} below for the case of an
operad in groupoids).
By using the same method as in the previous section, we shall obtain a
theorization of the notion of $\cat{U}$-algebra, which we call
\emph{$\cat{U}$-graded operad} (in the uncoloured version).
Let us assume for simplicity, that $\cat{U}$ is an uncoloured operad.

Recall that the structure of a $\cat{U}$-algebra on an object $A$ of a
symmetric monoidal category, is defined by an associative action on
$A$, of operations in $\cat{U}$.
If $u$ is an $S$-ary operation in $\cat{U}$ for a finite set $S$, then
it should act as an $S$-ary operation $A^{\tensor S}\to A$.
Now, to theorize the notion of $\cat{U}$-algebra given by an action of
the operators of $\cat{U}$, means to modify the definition
of this structure by replacing an action of every operator $u$ in
$\cat{U}$, by a choice of the set ``of operations of shape (so to
speak) $u$''.
We call an element of this set an \kore{operation} of \kore{degree}
$u$.

Thus the datum of a $\cat{U}$-graded operad $\cat{X}$ in sets should
include,
for every operation $u$ in $\cat{U}$, a set whose element we shall
call an operation in $\cat{X}$ of degree $u$.
If the operation $u$ is $S$-ary in $\cat{U}$, then we shall say that
any operation of degree $u$ in $\cat{X}$ has \kore{arity} $S$.

There should further be given a consistent way to compose operations
in $\cat{X}$, according to the way how operations in
$\cat{U}$ compose, namely, in a manner which respects the degrees of
the operations.
These will be a complete set of data for a \kore{$\cat{U}$-graded
  operad} $\cat{X}$ in sets.

There is also a coloured version of this, which we call
$\cat{U}$-graded \kore{$1$-theory}, \kore{multicategory} or
\kore{coloured} operad, and this is a theorization of
$\cat{U}$-algebra in a more general sense.
From the general discussions of theorization in Sections
\ref{sec:theorization-general} and \ref{sec:basic-construction},
$\cat{U}$-graded multicategory will turn out to be also a
generalization of $\cat{U}$-monoidal category, generalizing the way
how symmetric multicategory generalizes symmetric monoidal category.

\subsubsection{}
\label{sec:graded-multicat-as-lying-over}

By reflecting on what we have done above, we immediately find that a
$\cat{U}$-graded operad in sets is in fact exactly a symmetric operad
$\cat{Y}$ in sets equipped with a morphism
$P\colon\cat{Y}\to\cat{U}$.
The relation between $\cat{X}$ above and $\cat{Y}$ here is that an
$S$-ary operation in $\cat{Y}$ is an $S$-ary operation in $\cat{X}$
of \emph{any} degree.
The map $P$ maps an operation in $\cat{Y}$ to the degree which the
operation had when it was in $\cat{X}$.
Conversely, given an $S$-ary operation $u$ in $\cat{U}$, an operation
in $\cat{X}$ of degree $u$ is an $S$-ary operation in $\cat{Y}$ which
lies over $u$.
For example, $\cat{U}$, lying terminally over itself, indeed
corresponds in this manner, to the terminal $\cat{U}$-graded operad,
which has exacly one operation of each degree.

\begin{remark}\label{rem:graded-by-groupoid}
In the case where $\cat{U}$ is an operad in groupoids, the set of
$S$-ary operations in $\cat{X}$ of degree $u$ should be functorial in
$u$ on the groupoid of $S$-ary operations in $\cat{U}$.
The $S$-ary operations in $\cat{Y}$ will then be the corresponding
groupoid projecting to the groupoid of $S$-ary operations in $\cat{U}$
(obtained by the Grothendieck construction \cite[Section 8]{sga1-6} or
the (homotopy) colimit in groupoids).
For $\cat{U}=E_2$, the mentioned functoriality corresponds to the
action of the pure braid groups in a braided operad.
See Fiedorowicz \cite{fiedoro}.

In general, a symmetric operad $\cat{Y}$ in groupoids equipped with
$P\colon\cat{Y}\to\cat{U}$, corresponds to a $\cat{U}$-graded operad
$\cat{X}$ in \emph{groupoids}.
$\cat{X}$ is in sets if for every operation $u$ in $\cat{U}$, the
groupoid of operations in $\cat{X}$ of degree $u$, obtained as the
(homotopy) fibre over $u$ in $\cat{Y}$, is a \emph{homotopy
  $0$-type}, namely, a groupoid in which every pair of maps $f,g\colon
x\equivto y$ between the same pair of objects are equal.
\end{remark}

\begin{remark}\label{rem:different-categorification}
Inclusion of operads in groupoids in the discussion leads to a subtle
situation.
For example, if $\cat{U}=E_2$, then a $\cat{U}$-algebra in a symmetric
monoidal category is simply a commutative algebra.
Therefore, we are considering both $E_2$-graded multicategory and
symmetric multicategory as theorizations of commutative algebra.
The crucial difference between the two theorizations is between the
categorifications being generalized, namely, braided monoidal category
and symmetric monoidal category.
\end{remark}

Similarly, a $\cat{U}$-graded multicategory enriched in sets is a
symmetric multicategory
in which multimaps are graded by multimaps of $\cat{U}$.
In other words, it is just a symmetric multicategory (enriched in
sets) equipped with a morphism to $\cat{U}$.
Now, given a $\cat{U}$-graded multicategory $\cat{X}$, an
$\cat{X}$-algebra
in a $\cat{U}$-graded $1$-theory $\cat{Y}$ will be just a functor
$\cat{X}\to\cat{Y}$ of $\cat{U}$-graded $1$-theories.

\begin{example}\label{ex:initially-graded}
A multicategory graded by the initial operad $\Ini$, is a
multicategory 
with only unary multimaps, which is equivalent as a datum to a
category.
\end{example}

A similar theorization of the notion of $\cat{U}$-algebra, can be
defined for a \emph{coloured} symmetric operad $\cat{U}$ enriched in
sets, and a $\cat{U}$-graded $1$-theory enriched in sets will be again
a symmetric multicategory $\cat{X}$ (enriched in sets) equipped with a
functor to $\cat{U}$.
In other words, $\cat{X}$ will be such that not only multimaps in it
are graded, but objects are also graded by objects of $\cat{U}$.
For an object $u$ of $\cat{U}$, an object of $\cat{X}$ of
\emph{degree} $u$, will be just an object of $\cat{X}$ lying over
$u$.

\begin{example}\label{ex:graded-by-category}
Recall, as noted in Example \ref{ex:initially-graded}, that a category
$\cat{C}$ can be considered as a symmetric
multicategory having only unary multimaps.
If we consider $\cat{C}$ as a multicategory in this way, then a
$\cat{C}$-graded $1$-theory enriched in sets is a category equipped
with a functor to
$\cat{C}$, and this theorizes $\cat{C}$-algebra, or functor on
$\cat{C}$ (which one might also call a left $\cat{C}$-module).
On the other hand, a categorification of $\cat{C}$-algebra is a
category valued functor $\cat{C}\to\Cat$, and among the theorizations,
categorifications correspond to op-fibrations over $\cat{C}$.
\end{example}

Suppose given a category $\cat{C}$ and two functors
$\cat{F},\cat{G}\colon\cat{C}\to\Cat$, corresponding respectively to categories
$\cat{X},\cat{Y}$ lying over $\cat{C}$, mapping down to $\cat{C}$ by
op-fibrations.
Note that, by Example \ref{ex:graded-by-category}, $\cat{F}$ and $\cat{G}$ are
categorified $\cat{C}$-modules, and $\cat{X}$ and $\cat{Y}$ as
categories over
$\cat{C}$, are the corresponding $\cat{C}$-graded $1$-theories.

In this situation, the relation between maps $\cat{F}\to\cat{G}$ and
maps $\cat{X}\to\cat{Y}$, is as follows.
Namely, a functor $\phi\colon\cat{X}\to\cat{Y}$ of categories over
$\cat{C}$ (see Remark \ref{rem:groupoid-terminology}, to be
technical), corresponds to a map $\cat{F}\to\cat{G}$ if and only if
$\phi$ preserves coCartesian maps, and an arbitrary functor $\phi$
over $\cat{C}$ only corresponds to a \emph{lax} map
$\cat{F}\to\cat{G}$ (defined with respect to the standard
\emph{$2$-category} structure on $\Cat$).

This is an instance of Theorem
\ref{thm:functor-of-theorization-intro}.

\subsection{Theorization in general}
\label{sec:theorization-general}

\subsubsection{}
\label{sec:bimodule}
For the idea for theorization of a general algebraic structure, the
notion of \emph{profunctor}\slash\emph{distributor}\slash\emph{bimodule} is
useful.
For categories $\cat{C},\cat{D}$, a
\kore{$\cat{D}$--$\cat{C}$-bimodule} (in the category $\Set$)
is a functor $\cat{C}^\op\times\cat{D}\to\Set$.
The category of $\cat{D}$--$\cat{C}$-bimodules contains the opposite
of the functor category $\Fun(\cat{C},\cat{D})$
as a full subcategory, where a functor $F\colon\cat{C}\to\cat{D}$ is
identified with the bimodule $\Map_\cat{D}(F-,-)$.
Let us say that this bimodule is \kore{corepresented} by $F$.
By symmetry, the category of bimodules also contains
$\Fun(\cat{D},\cat{C})$.
However, for the purpose of theorization, we treat $\cat{C}$ and
$\cat{D}$ asymmetrically, and mostly consider only
\emph{co}representation of bimodules.
Bimodules compose by tensor product, to make categories form a
$2$-category, extending the $2$-category formed with (opposite)
functors as $1$-morphisms, by the identification of a functor with the
bimodule corepresented by it.

\subsubsection{}
\label{sec:theorization-formulation}

We would like to consider the general idea of theorization, while
having in mind as an example, the case of the notion of
$\cat{B}$-algebra for a symmetric monoidal category $\cat{B}$.

We shall find that theorization is in fact more general than
\emph{lax} (or ``op-lax''; see below) categorification, where, by a
\kore{lax categorification}
of a kind of structure, we mean a \emph{relaxation} of a specific
categorification of the same kind of structure in the sense that it is
a specific generalization of the specified categorification such that,
in a lax categorified structure, a non-invertible map (going in a
specified
direction) is allowed in place of every one of some specified
structure isomorphisms in some presentation of the categorified form
of structure of the kind.
Let us denote by $\CCat$, the \emph{$2$-category} of categories and
functors.
Then the notion of $\cat{B}$-algebra has a canonical lax
categorification, which we shall call \emph{lax} $\cat{B}$-monoidal
category, where a \kore{lax $\cat{B}$-monoidal category} is by
definition, a lax functor $\cat{B}\to\CCat$ which is given a datum of
(not lax) commutation with the symmetric monoidal structures.
In the case where $\cat{B}$ is $\Fin$, this coincides with the
standard notion of lax symmetric monoidal category.

There is another lax categorification of the notion of
$\cat{B}$-algebra, which we shall call \emph{op-lax}
$\cat{B}$-monoidal category.
An \kore{op-lax} $\cat{B}$-monoidal category is similar to a lax
$\cat{B}$-monoidal category except that the laxness of the functor
$\cat{B}\to\CCat$ should be opposite, namely, the non-invertible
structure maps should go in the opposite direction.
In other words, it should be an symmetric monoidal op-lax functor in
one of the common conventions, in which a \emph{lax} functor
$\unity\to\CCat$ from the unit category $\unity$, is a category
equipped with a monad, rather than a comonad, on it (see B\'enabou
\cite{benabou}).

The idea of \kore{theorization} which we have described in Sections
\ref{sec:theorization}, \ref{sec:theorize-algebra},
can now be expressed as that it is a \emph{virtualization} of an
op-lax categorification, where by a \kore{virtualization} of op-lax
$\cat{B}$-monoidal structure, or any kind of structure as far as the
following makes sense, we mean a specific generalization of the
structure such that, in a virtualized structure, non-corepresentable
bimodules are
allowed in place of some specified structure \emph{functors}.
(Note here also that the structure maps on bimodules should be
understood to be in the opposite direction to the structure maps on
functors, owing to the
contravariance of the corepresentation of bimodules by functors.)

To be cautious, specification of a theorization of $\cat{B}$-algebra
for example, includes specification of the notion of ``structure
functor'' for $\cat{B}$-monoidal structure, which should at least
include specification of a collection of generating objects of
$\cat{B}$.
Specifically, given a family $b=(b_\lambda)_{\lambda\in\Lambda}$ of
objects of $\cat{B}$ which generates all objects under the monoidal
multiplication, we consider for a family
$\cat{X}=(\cat{X}_\lambda)_\lambda$ of categories, a symmetric
monoidal functor $\cat{F}\colon\cat{B}\to\Cat$ with $\cat{F}b=\cat{X}$
as a structure on $\cat{X}$, and then see a theorization as a more
general kind of structure which we can consider on $\cat{X}$ (although
a slightly more precise understanding will be that the structure is on
the collections of colours to be described shortly, as would be
concluded from the discussions of Section \ref{sec:action-of-identity}
below).

For example, theorization of algebra over a multicategory $\cat{U}$,
can be considered as the case where $\cat{B}$ is freely generated by
$\cat{U}$, and the generating objects which we choose will be the
indecomposable objects, i.e., the objects which come from $\cat{U}$.
See Section \ref{sec:graded-as-theorization} below.

Note in particular, that the idea above does not determine the
theorization from a categorification uniquely, so there is a question
on which theorization if any, we would like.
In the case of algebra over a multicategory $\cat{U}$, we achieve the
following through theorization.
Firstly, our categorification is $\cat{U}$-monoidal category, and this
already allows us to define a $\cat{U}$-algebra internal in a
$\cat{U}$-monoidal category $\cat{A}$ as a lax $\cat{U}$-monoidal
functor $\unity\to\cat{A}$, where $\unity$ denotes the unit
$\cat{U}$-monoidal category.
Now, through the process of theorization, this notion of
$\cat{U}$-algebra becomes generalized to the notion of functor of
$\cat{U}$-graded multicategories.
See Example \ref{ex:algebra-as-functor}.

In general, given a specific kind of structure in a specific context,
and some reasonable categorification of it, we would like a similar
extension of the notion through theorization.
It is a theorization which allows this that we would like, and
existence of such a theorization appears to be usually a non-trivial
question.

Another thing to note is that the notion of theorization which we have
formulated is the ``coloured''
version which we did not discuss in detail in Section
\ref{sec:theorization} or \ref{sec:theorize-algebra}.
A \kore{colour} in the theorized structure is an object of a category
in the underlying family of categories, e.g.,
$(\cat{X}_\lambda)_{\lambda\in\Lambda}$ above.
This generalizes the colours in a multicategory. 
See Section \ref{sec:graded-as-theorization} below.

To be more detailed, in the case discussed above, one can consider an
object of $\cat{X}_\lambda$ as a colour having \kore{degree}
$\lambda$.

\begin{remark}
As mentioned in Remark \ref{rem:presentation-vs-unenrich}, it is
better to have a presentation of the form of datum for a theorized
structure, so in particular, we have a definition of a theorized
structure \emph{enriched} in a symmetric monoidal category.
In practice, it is perhaps not difficult usually, to write down a
presentation by looking at the process of the theorization carefully.
Essentially, one simply needs to run the virtualization process using
the enriched version of bimodules, even though, to be rigourous, there
is a minor issue here that bimodules in a general symmetric monoidal
category do not necessarily compose, so we need to work actually in a
\emph{$2$-theory} formed by enriched categories and bimodules.
However, we shall not worry about this in this introduction, and
shall mainly consider only \emph{un}enriched theorized structures.

In the concrete situations which we treat in the body, another,
simpler method for theorization (which demands more concrete data as
an input) will in fact give a simpler solution for enriching the
theorized structure in a symmetric monoidal category, as will be seen
in Section \ref{sec:coloured-lax-structure}.
\end{remark}

\subsubsection{}
\label{sec:graded-as-theorization}
For a multicategory $\cat{U}$, let us try to interpret
$\cat{U}$-graded multicategory as a ``theorization'' of
$\cat{U}$-algebra in the defined sense.

Firstly, we consider the structure of a $\cat{U}$-algebra as a
structure on family of objects (of a symmetric monoidal category)
indexed by the objects of $\cat{U}$.
Namely, we consider a $\cat{U}$-algebra $A$ in a symmetric monoidal
category $\cat{A}$, as consisting of
\begin{itemize}
\item for every object $u\in\Ob\cat{U}$, an object $A(u)$ of
  $\cat{A}$,
\item for every finite set $S$ and an $S$-ary operation $f\colon
  u\to u'$ in $\cat{U}$, where $u=(u_s)_{s\in S}$ is a family of
  objects of $\cat{U}$ indexed by $S$, a map $Af\colon A(u)\to A(u')$
  in $\cat{A}$, where $A(u):=\Tensor_{s\in S}A(u_s)$,
\end{itemize}
and then consider the latter as a structure on the family
$\Ob A:=\bigl(A(u)\bigr)_{u\in\Ob\cat{U}}$ of objects of $\cat{A}$.
The structure is thus an action of every multimap $f$ in $\cat{U}$ on
the relevant members of the family $\Ob A$.

We would like next to obtain from a $\cat{U}$-graded multicategory
$\cat{X}$, a family similar to $\Ob A$ above, \emph{of categories}, to
underlie $\cat{X}$.
For this, we take the family
$\Ob\cat{X}:=(\cat{X}_u)_{u\in\Ob\cat{U}}$, where $\cat{X}_u$ denotes
the category formed by the objects of $\cat{X}$ of 
degree $u$, and maps (i.e., unary multimaps) between them of degree
$\id_u$.

The rest of the structure of $\cat{X}$ can then be considered as a
lax associative action of the rest of the multimaps $f$ in $\cat{U}$,
on these categories $\cat{X}_u$, each $f$ acting as the bimodule
formed by the multimaps in $\cat{X}$ of degree $f$, so $\cat{X}$ can
be interpreted as obtained by putting a theorized $\cat{U}$-algebra
structure on the family $\Ob\cat{X}$.
(See Section \ref{sec:action-of-identity} below to be more precise.)

\begin{remark}\label{rem:contravariance}
Lax associativity of an action through bimodules generalizes
\emph{op}-lax associativity of an action through functors.
\end{remark}

If a $\cat{U}$-graded multicategory $\cat{X}$ is seen as a theorized
structure in this manner, then a colour in this theorized structure is
an object of a category $\cat{X}_u$, where $u$ is any object of
$\cat{U}$.
In other words, it is an object of the multicategory $\cat{X}$.

\subsubsection{}
\label{sec:action-of-identity}
We would like to give a minor and technical remark.

For a multicategory $\cat{U}$, a natural theorization which our
definition expects of
$\cat{U}$-algebra would appear to be lax $\cat{U}$-algebra in the
$2$-category mentioned above formed by categories and bimodules
between them.
(Note Remark \ref{rem:contravariance}.)
This does not coincide with our desired theorization, which is
$\cat{U}$-graded multicategory.

Indeed, for an object $u$ of $\cat{U}$, if a category, say
$\cat{X}_u$, is associated to $u$, and the identity map of $u$ acts on
$\cat{X}_u$ in our $2$-category of bimodules, then this acton gives
another category, say $\cat{Y}_u$, with objects the objects of
$\cat{X}_u$, and a map, say $F\colon\cat{X}_u\to\cat{Y}_u$, of the
structures of categories on the same collection of objects.
However, the theorization in the idea described in the previous
sections, is not where $\id_u$ acts on an already existing category
$\cat{X}_u$, but where the structure of $\cat{X}_u$ itself as a
category, is the action of $\id_u$.

In other words, we usually do not just want to consider a relaxed
structure in the $2$-category of categories and bimodules, but we
would further like to require that the resulting map corresponding to
$F$ in the example above, associated to each member of the
underlying family of categories, to be an isomorphism.

\begin{remark}
In the example of Section \ref{sec:graded-as-theorization}, there is
another category
structure on the objects of $\cat{X}_u$, in which a map is a (unary)
map in $\cat{X}$ of degree an \emph{arbitrary} endomorphism of $u$ in
$\cat{U}$, rather than just the identity.
This clearly does not interfere with the remark here.
\end{remark}

\subsubsection{}
\label{sec:coloured-lax-structure}
For a final remark, for the kinds of structure which we know how to
theorize, we actually have a more economical description of the
theorizations than we have given above.
This uses the construction of a $2$-category by `categorically
delooping' an associative monoidal category $\cat{A}$.
See B\'enabou \cite[2.2]{benabou} or the review in Section
\ref{sec:coloured-lax-algebra}.
Note that, if $\cat{A}$ is a symmetric monoidal category, then the
resulting $2$-category, which we shall denote by $\deloop\cat{A}$,
inherits a symmetric monoidal structure.

To turn to the description of the theorization, for the case of the
notion of algebra over a multicategory $\cat{U}$ for example, a
$\cat{U}$-graded multicategory enriched in a symmetric monoidal
category $\cat{A}$, can be described as a coloured
version of a lax $\cat{U}$-algebra in the symmetric monoidal
$2$-category $\deloop\cat{A}$.
We refer the reader to Section \ref{sec:coloured-lax-algebra} for the
case $\cat{U}=\Com$ of this.
In general, it will be seen in Section \ref{sec:theorem} that, for an
$n$-theory $\cat{U}$, our theorization of $\cat{U}$-algebra will, in
the $\cat{A}$-enriched form, be an (``$n$-tuply'') coloured lax
$\cat{U}$-algebra enriched in $\deloop\cat{A}$.

This idea of \emph{coloured lax structure enriched in
  $\deloop\cat{A}$}, is actually less redundant than the idea of
theorization which we have expressed above for a general situation,
and does not produce the issue
discussed in Section \ref{sec:action-of-identity}, either.
This idea is the one along which we actually theorize kinds of
algebraic structure which we can theorize so far.
See the definitions in the body.
In particular, the simplest case will be observed explicitly in
Proposition \ref{prop:theory-as-theorization}.

However, compared to our previous formulation of the idea of
theorization, the formulation of the notion of ``coloured lax
structure'' enriched in a symmetric monoidal $2$-category, would rely
more on the manner how we generate the relevant structure (e.g., the
symmetric monoidal category $\cat{B}$ for the case of
``$\cat{B}$-algebra'').
Since the author does not know what exact data are needed for
theorizing a kind of ``algebraic structure'' in general, he does not
know a general definition of a coloured lax structure.
To formulate this notion for a given kind of structure, seems
essentially equivalent to theorizing the kind of structure.

\begin{remark}
On the other hand, we have an `uncoloured theorization' as soon as we
have a lax version of the kind of structure.
However, unless we can further find a reasonable common generalization
of this uncoloured ``theorization'' and the categorification, there
might not be a reasonable notion of algebra over an uncoloured
``theorized'' object of the kind.
\end{remark}

\subsection{A basic construction}
\label{sec:basic-construction}

The idea of theorization was such that a categorified structure was an
instance of the theorized form of the same structure.
Let us suppose given a kind of structure and a theorization of it.
Then, for a categorified structure $\cat{X}$, let us denote by
$\wasreteori\cat{X}$, the theorized structure corresponding to
$\cat{X}$.
Concretely, $\wasreteori\cat{X}$ is obtained by replacing as needed,
structure functors of $\cat{X}$ with bimodules copresented by them.
The construction $\wasreteori$ generalizes the usual way to construct
a multicategory from a monoidal category.
Let us say that $\wasreteori\cat{X}$ is \kore{represented} by
$\cat{X}$.

\begin{remark}
Recall that a theorized structure in general could have colours.
This flexibility is playing an important role here.
Indeed, if the categorified structure $\cat{X}$ is structured on a
family, say $\Ob\cat{X}=(\Ob_\lambda\cat{X})_{\lambda\in\Lambda}$, of
categories, where $\Lambda$ denotes a collection which is suitably
specified in the chosen presentation of the form of the structure in
question, then any object of $\Ob_\lambda\cat{X}$ for any
$\lambda\in\Lambda$,
is being a colour in the theorized structure $\wasreteori\cat{X}$.

Thus the coloured version of the notion of theorization is indeed
necessary in order for every
categorified structure to be an instance of a theorized structure.
\end{remark}

As in the case of monoidal structure, $\wasreteori$ is usually only
faithful, but not full.
Indeed, for (families of) categories $\cat{X},\cat{Y}$, each equipped
with a categorified structure, a morphism
$\wasreteori\cat{X}\to\wasreteori\cat{Y}$ is equivalent as a datum to
a \emph{lax} morphism $\cat{X}\to\cat{Y}$.
See Section \ref{sec:functor-of-theorization}.

\begin{example}\label{ex:algebra-as-functor}
Let $\cat{U}$ be a symmetric multicategory, and let $\cat{A}$ be a
$\cat{U}$-monoidal category.
Then a $\cat{U}$-algebra in $\cat{A}$, namely, a lax
$\cat{U}$-monoidal functor $\unity\to\cat{A}$, where $\unity$ denotes
the unit $\cat{U}$-monoidal category, is equivalent as a datum to a
functor $\wasreteori\unity\to\wasreteori\cat{A}$ of $\cat{U}$-graded
multicategories, which is by definition, a $\cat{U}$-algebra in
$\wasreteori\cat{A}$.
\end{example}

\begin{remark}
The functor $\wasreteori$ has a left adjoint (which in fact can be
described in a concrete manner).
In the example above, $\cat{U}$-algebra in $\cat{A}$ is thus
equivalent to a $\cat{U}$-monoidal functor to $\cat{A}$ from the
$\cat{U}$-monoidal category freely generated from the terminal
$\cat{U}$-graded $1$-theory $\wasreteori\unity$ (which thus has a
concrete description).
\end{remark}

\subsection{Theorization of category}
\label{sec:theorize-category}

\subsubsection{}

For illustration of the general definition, let us describe a natural
theorization of the notion of category, which we shall call
``categorical theory'' here.

In order to define this, we first note that category can be understood
as a kind of structure on family of sets.
Indeed, for a category $\cat{X}$, there is the family
$\Map_\cat{X}:=\bigl(\Map_\cat{X}(x,y)\bigr)_{x,y}$ of sets of morphisms,
parametrized by pairs $x,y$ of objects of $\cat{X}$, so the structure of
$\cat{X}$ can be understood as defined on this family $\Map_\cat{X}$
of sets, by the composition operations.
Moreover, this presentation immediately leads to a generalization of
the notion to the notion of category enriched in a symmetric monoidal
category.

Now, if we choose and fix a collection as the collection of objects
for our (enriched) categories, then $2$-category with the same
collection of objects can
be considered as a categorification of category with that collection
of objects.
Categorical theory will be a theorization of category whose associated
categorification is $2$-category.

The description of a \kore{categorical theory} (enriched in sets) is
as follows.
Firstly, it, like a $2$-category (our categorification) has objects,
$1$-morphisms, and sets of $2$-morphisms.
$1$-morphisms do not compose, however.
Instead, for every nerve
$f\colon x_0\xrightarrow{f_1}\cdots\xrightarrow{f_n}x_n$ of
$1$-morphisms and a $1$-morphism $g\colon x_0\to x_n$, one has the
notion of (\kore{$n$-ary}) \kore{$2$-multimap} $f\to g$.
The $2$-morphisms, which were already mentioned, are just unary
$2$-multimaps.
There are given unit $2$-morphisms and associative composition for
$2$-multimaps, analogously to the similar operations for multimaps in
a planar multicategory.

A $2$-category in particular represents a categorical theory, in which
a $2$-multimap $f\to g$ is a $2$-morphism $f_n\compose\cdots\compose
f_1\to g$ in the $2$-category.
Between two $2$-categories, a natural map of the represented
categorical theories is not precisely a functor, but is a lax functor
of the $2$-categories.

As we have also suggested, a categorical theory is also a
generalization of a planar multicategory.
Indeed, planar multicategory was a theorization of associative
algebra.
The relation between the notions of planar multicategory and of
categorical theory, is parallel to the relation between the notions of
associative monoid and of category.
Namely, categorical theory is a `many objects' (or ``coloured'')
version of planar multicategory, where the word ``object'' here
refers to one at a deeper level than the many objects which
a multicategory (as a ``coloured'' operad) may already have are at.

One thing one should note then, is that, while the $2$-multimaps
in a categorical theory is generalizing the multimaps in a planar
multicategory here, the $1$-morphisms in a categorical theory is
generalizing the \emph{objects} of a planar multicategory, and no
longer have the characteristic of operators like the $1$-morphisms in
a category.
Indeed, a $1$-morphism in a categorical theory and an object of a
planar multicategory are both ``colours'', and we have also mentioned
earlier that there is no operations of composition given for
$1$-morphisms in a categorical theory.

What we said above in comparison of the structures of a categorical
theory and of a planar multicategory, is that a categorical theory
$\cat{X}$ has one more layer of `colouring' under the $1$-morphisms,
given by the collection of the objects of $\cat{X}$.

\begin{example}\label{ex:planar-deloop}
Let $\cat{A}$ be an associative monoidal category.
Then the categorical theory corresponding to the planar multicategory
$\wasreteori\cat{A}$, is represented by the $2$-category
$\deloop\cat{A}$.
(To be technical, the former categorical theory is ``simply coloured''
in the sense that it has only one layer of colours, and it is
equivalent to the `simply coloured part' at the base object, of the
categorical theory $\wasreteori\deloop\cat{A}$.
See Section \ref{sec:delooping-comparison} for an explanation in a
similar situation.)
\end{example}

\begin{example}\label{ex:morita-theory}
The ``Morita'' $2$-category due to B\'enabou \cite{benabou} of
associative algebras and bimodules in a monoidal category $\cat{A}$
with nicely behaving colimits, is well-defined as a categorical theory
when $\cat{A}$ is more generally, an arbitrary planar multicategory.
There is a forgetful functor from this ``Morita'' categorical theory
to the categorical theory of Example \ref{ex:planar-deloop}.
\end{example}

\subsubsection{}

Following the general pattern about theorization, there is a notion of
category in a categorical theory, and, as an uncoloured version of it,
monoid in a categorical theory, which generalizes monad in a
$2$-category.
A monad in a $2$-cateogry $\cat{X}$ was a lax functor to $\cat{X}$
from the unit $2$-category (see B\'enabou \cite{benabou}), which can
also be considered as a map between the categorical theories
represented by these $2$-categories.
See Section \ref{sec:basic-construction}. 
However, the latter is a monoid in the target categorical theory
$\wasreteori\cat{X}$ by definition.

\begin{example}
Let $\cat{C}$ be a category enriched in groupoids, and let $M$ be a
monad \emph{on} $\cat{C}$.
Then there is a categorical theory as follows.
\begin{itemize}
\item An object is an object of $\cat{C}$.
\item A map $x\to y$ is a map $Mx\to y$ in $\cat{C}$.
\item Given a sequence of objects $x_0,\ldots,x_n$ and maps $f_i\colon
  Mx_{i-1}\to x_i$ in $\cat{C}$ and $g\colon Mx_0\to x_n$, the set
  $\Mul(f;g)$ of $2$-multimaps $f\to g$, is the set of
  \emph{commutative} diagrams
  \[\begin{tikzcd}[column sep=large]
    M^nx_0\arrow{d}[swap]{m}\arrow{r}{M^{n-1}f_1}
    &\cdots\arrow{r}{f_n}
    &\tsketas{x}{_n}\arrow{d}{=}\\
    Mx_0\arrow{rr}{g}&&\tsketas{x}{_n}
  \end{tikzcd}\]
  in $\cat{C}$ (i.e., the set of isomorphisms filling the rectangle),
  where $m$ denotes the multiplication operation on $M$.
\item Composition is done in the obvious manner.
\end{itemize}
A monoid in this categorical theory is exactly an $M$-algebra in
$\cat{C}$.
\end{example}

More generally, a category in a categorical theory can be described as
a coloured version of a map of categorical theories.
A basic example is a category in the categorical theory obtained by
considering a planar multicategory $\cat{U}$ as a categorical theory.
This is equivalent as a datum to a category enriched in the planar
multicategory $\cat{U}$.

\subsection{Iterating theorization}
\label{sec:theorize-theory}

\subsubsection{}

In Section \ref{sec:theorize-category}, we have theorized the notion
of category by considering a category as a structure on the family of
sets consisting of the sets of maps.
Recall from Example \ref{ex:initially-graded} that a category was an
`initially graded' $1$-theory.
Section \ref{sec:graded-as-theorization} shows thus
that category is a theorization of $\Ini$-algebra, or bare, i.e.,
unstructured, object.
(It is not difficult to see in the similar manner, that the versions
enriched in a symmetric monoidal category, of the relevant notions
also coincide.)
For these reasons, we shall call a categorical theory also an
``$\Ini$-graded $2$-theory''.

One can similarly consider the structure of a categorical theory as a
structure on the sets of its $2$-multimaps, and then try to theorize
the notion of categorical theory after fixing the collections of
objects and of $1$-morphisms.
It turns out that there is indeed an interesting theorization in this
case, which in particular generalizes the notion of $3$-category.
One might call the resulting theorized object a categorical $2$-theory
or an initially graded $3$-theory.

One might ask whether it is possible to iterate theorization in a
similar manner here, or starting from some specific additional
structure, rather than `no structure at all'.
We should of course ask possibility of interesting theorizations.
One condition for this has been mentioned in Section
\ref{sec:theorization-formulation}.
Other desirable things may include abundance of natural examples, and
reasonable properties.

In the case where the answer to the question is affirmative, just as
the original structure could be expressed as the structure of an
algebra over the terminal object among the (unenriched) theorized
objects of the same kind, the theorization similarly becomes the
structure of an algebra over the
terminal object among the twice theorized objects, and so on, so all
the structures can be described using their iterated theorizations.
Moreover, by considering an algebra over a non-terminal theory, an
algebra over it, and so on, one obtains various general structures,
which can all be treated in a unified manner.

The question asked above is non-trivial.
However, we introduce the notion of $n$-theory in this work, which
will inductively be an interesting theorization of ($n-1$)-theory.
The hierarcy as $n$ varies, of $n$-theories will be an infinite
hierarchy of iterated theorizations which extends the various standard
hierarchies of iterated categorifications, in particular, the hierarcy
of $n$-categories in the ``initially graded'' case.

While our higher theories will be in general a completely new
mathematical objects, we have already found very classical objects of
mathematics among $2$-theories.
Namely, while we have seen that a categorical theory was an
``initially graded'' $2$-theory, we have also noted in Section
\ref{sec:theorize-category}, that planar multicategories were
among categorical theories.

We have also seen non-classical objects among $2$-theories in Section
\ref{sec:theorize-category}.
Less exotic examples of higher theories comes from the construction of
Section \ref{sec:basic-construction}.
Namely, a higher categorified instance of a \emph{lower} theorized
structure leads to a higher theorized structure through the iterated
application of the construction $\wasreteori$ (details of which can be
found in Section \ref{sec:forget-to-theorization}).
As an object, this is less interesting among the general higher
theories for the very reason that it is represented by a lower
theory.
However, the \emph{functors} between these theories are interesting in
that it is much more general than functors which we
consider between the original lower theories.
Namely, a functor between such higher theories amounts to highly
\emph{relaxed} functor between the higher categorified lower theories,
as follows from an iteration of the remark in Section
\ref{sec:basic-construction}.

In Section \ref{sec:delooping} (as will be previewed in Section
\ref{sec:deloop-preview}), we discuss a general construction which we
call ``delooping'', through which we obtain an ($n+1$)-theory which
normally fails to be representable by a categorified $n$-theory.
Another construction, which is closely related to the classical Day
convolution, will also be discussed in Section \ref{sec:convolution},
and this also produces similar examples.

\subsubsection{}

We have sketched in Section \ref{sec:grading-intro}, how to generalize
the hierarchy of higher theories to the hierarchy graded by a higher
theory.
This will be done in Section \ref{sec:grading}.

\subsection{Further developments}
\label{sec:further-development}

\subsubsection{}

Let us preview a few more highlights of our work.

\subsubsection{}
\label{sec:grade-by-graded-intro}
In Section \ref{sec:grading-intro}, we have hinted at a notion of
algebra over a monoid (i.e., algebra enriched in sets)
over an $n$-theory enriched in sets (if $n\ge 2$), a notion of algebra
over unenriched such (if $n\ge 3$), and so on.
For a $\cat{U}$-graded $m$-theory $\cat{X}$ enriched
in sets, we indeed define
the notion of $\cat{X}$-algebra, and more generally, of
$\cat{X}$-graded $\ell$-theory, as well as the notion of higher
theory graded by unenriched such, and so on.
We can actually give simple definitions of all these, using
Theorem \ref{thm:graded-theory-is-monoid-intro} as a general
principle (Section \ref{sec:iterated-monoid}).
At the end, every structure (which is enriched in sets or groupoids)
will come with a hierarchy of higher theories ``graded'' by it.
We obtain a generalization of Proposition
\ref{prop:graded-is-overlying-intro} with these new notions
(Proposition \ref{prop:graded-projection}).
We also treat some basic questions concern the general enriched
versions of all the notions in Section \ref{sec:enrichment}.

\subsubsection{}
\label{sec:deloop-preview}

An $n$-theory enriched in a symmetric monoidal category $\cat{A}$, can
also be described as a (suitable) $n$-theory in the ($n+1$)-theory
$\wasreteori^{n+1}\deloop^n\cat{A}$ obtained by applying the
construction $\wasreteori$ $n+1$ times (Section
\ref{sec:forget-to-theorization}) to the symmetric monoidal
($n+1$)-category $\deloop^n\cat{A}$ (the $n$-th iterated categorical
deloop of $\cat{A}$).
In Section \ref{sec:delooping}, we generalize the categorical
delooping construction for symmeric monoidal higher category, to a
certain construction $\ddeloop$ which
produces a symmetric $n$-theory from a symmetric ($n-1$)-theory.
This is a generalization of the categorical delooping in such a manner
that, for a symmetric monoidal category $\cat{A}$, there is a natural
`equivalence'
$\wasreteori^{n+1}\deloop^n\cat{A}\equivwith\ddeloop^n\wasreteori\cat{A}$
(we refer the reader to Section~\ref{sec:delooping-comparison} for
the precise relation).
It follows that a natural notion of $n$-theory enriched in a symmetric
multicategory $\cat{M}$, which is not necessarily of the form
$\wasreteori\cat{A}$, is $n$-theory in the ($n+1$)-theory
$\ddeloop^n\cat{M}$, and, as $n$ increases, the notion iteratively
`theorizes' the previous notions in a suitable sense
(see Remark~\ref{rem:lax-functor-general} and
Proposition~\ref{prop:theory-as-theorization}).

Incidentally, if $\cat{M}$ is not of the form $\wasreteori\cat{A}$,
then $\ddeloop^n\cat{M}$ is usually \emph{not} representable by a
categorified $n$-theory.

\subsubsection{}
\label{sec:graded-theory-as-lift-intro}

For a symmetric monoidal $n$-category $\cat{C}$, we construct a
certain ($n+1$)-theory $\gyzen_n\cat{C}$ and a functor
$\gyzen_n\cat{C}\to\wasreteori^{n+1}\deloop^n\Set$ of
($n+1$)-theories.
The use of this is the following.
Namely, while we have already mentioned that an $n$-theory $\cat{X}$
enriched in sets can be considered as an $n$-theory in
$\wasreteori^{n+1}\deloop^n\Set$, the construction above allows us
to understand an $\cat{X}$-graded $m$-theory enriched in a symmetric
monoidal category $\cat{A}$, where $0\le m\le n-1$, as an appropriate
lift of the theory to $\gyzen_n\deloop^m\cat{A}$ (a more general
statement being as Corollary \ref{cor:ab-ba}, where
$\gyzen_n=\wasreteori_n\gyzen_*\wasreteori^n_0$ in the notation
there).
There is also a version of this for $m=n$ (Corollary
\ref{cor:graded-theory-as-lift}), which will have an application in
our work.

\subsubsection{}
\label{sec:further-list}

We also touch on more topics, such as the following.
\begin{itemize}
\item Pull-back and push-forward constructions which changes gradings,
  and their properties (Sections \ref{sec:iterated-monoid},
  \ref{sec:right-adjoint-of-restriction}).
  The result Corollary \ref{cor:graded-theory-as-lift} mentioned in
  Section \ref{sec:graded-theory-as-lift-intro}, will be used for the
  construction of the push-forward `on the right side'.
\item Some other basic constructions such as a construction for higher
  theories related to Day's convolution \cite{day}.
  The Day type construction leads very easily to a notion of algebra
  over an
  \emph{enriched} higher theory (Section \ref{sec:convolution}).
\item Hierarchies of iterated theorizations associated to more general
  systems of operations, with multiple inputs \emph{and} multiple
  outputs, such as operations of `shapes' of bordisms as in various
  versions of a topological field theory (Section
  \ref{sec:generalization}).
  Examples also include iterated theorizations of the notion due to
  Vallette of (coloured) \emph{properad} \cite{vallette}.
  See Example \ref{ex:properad}.
\end{itemize}

The last topic leads to vast generalizations of the relevant versions
of the notion of topological field theory.
We obtain a simple but in a way exotic example (Section
\ref{sec:example-different-nature}) in addition to examples of a more
expected type.

\subsubsection{}

We also enrich everything we consider in this work, in the Cartesian
symmetric monoidal infinity $1$-category of infinity groupoids,
instead of in sets.
Fortunately, this does not add any difficulty to the discussions.
However, we invite the reader who does not wish to deal with homotopy
theory, to Section \ref{sec:no-homotopy}.

\subsection{Outline}

We shall give the definition of an ungraded higher theory with a
minimal degree of enrichment, in \textbf{Section
  \ref{sec:symmetric-theory}}.
We shall follow up the definition in \textbf{Section
  \ref{sec:variant-construction}} with discussions of simple subjects
such as a planar variant, a less coloured variant, the construction
$\wasreteori$, and the generalized ``delooping'' construction.
We shall then discuss algebras and graded higher (and ``lower'')
theories over a higher theory in \textbf{Section \ref{sec:grading}}.
We shall then discuss in \textbf{Section \ref{sec:enrichment}},
various topics about enrichment of the notion of higher theory,
including a construction for higher theories which is
related to the Day convolution.
Finally, we shall discuss iterated theorizations of more general
algebraic structure in \textbf{Section \ref{sec:generalization}}.
\textbf{Appendix} is for comparison of this work with a related
important work \cite{bd} of Baez and Dolan.

\subsection{Notes for the reader}

\subsubsection{}
\label{sec:no-homotopy}

For practical purposes, it seems best to build our theory in the
framework of homotopy theory.
Thus, we let the infinity $1$-category of infinity
groupoids be the default place where we enrich any categorical or
algebraic structure.

For the reader who does not wish to deal with homotopy theory, our
terminology in the body will be such that it could be read as if we
are working in the framework of the classical, discrete category
theory.
For example, we shall say ``category'' to in fact mean ``infinity
$1$-category''.
So such a reader would be comfortable with interpreting what we write
in the normal, classical manner, and then ignoring what is redundant
in such an interpretation.

\begin{remark}
Here are two cautions.
One is that, when we say ``groupoid'' (while in fact meaning infinity
groupoid), this can often be interpreted as set in the classical
context, but sometimes, it will be better to interpret
it honestly as ``groupoid'' (i.e., $1$-groupoid).
The other is that there is fear that some of the examples we give may
degenerate to trivialities in the non-homotopical interpretation.
\end{remark}

While we also welcome the reader who does this, our method for
theorization is by dealing with the structure of the coherence for
higher associativity, which is also the key for higher category
theory, so our expectation is that the
reader who interprets our work in the classical context, would
eventually find the ``homotopical'' interpretation more natural.
(The subject of the present work does not select any particular model
for higher category theory.
In order to communicate to as wide audience as we can, we shall try to
make clear which data are being used from the theory of higher
categories when we use them, so the hope is that even the reader who
has no more than basic ideas on the essence of the higher category
theory, would find our exposition largely accessible.)

\subsubsection{}

As we have mentioned,
we adopt the convention that all terms should be interpreted in
homotopical\slash infinity $1$-categorical sense.
Namely, categorical terms are used in the sense enriched in the
infinity $1$-category of infinity groupoids, and algebraic terms are
used freely in the sense generalized in accordance with the enriched
categorical structures.

However, we do welcome the reader who prefers to work in the classical
setting, to interpret the terms in the usual, non-homotopical manner.
In this case, categorical terms (e.g., multicategory) should be
understood in the sense enriched in the category of sets (or sometimes
better groupoids) unless otherwise specified, and Remark
\ref{rem:groupoid-terminology} would continue in effect.

For example, by a \emph{$1$-category}, we officially mean an
\emph{infinity} $1$-category, while also welcoming the classical
interpretation.
We often call a $1$-category (namely an infinity $1$-category) simply
a \kore{category}.
More generally, for an integer $n\ge 0$, by an \emph{$n$-category},
we mean an \emph{infinity} $n$-category.

\subsubsection{}

Unless otherwise noted, we do not consider non-unital associative
algebraic structures.
Moreover, we normally treat unitality as part of associativity.
(Specification of the unit will be a nullary operation.)

\subsubsection{}

In our notations, we shall freely put a non-negative integer (or a
variable, such as ``$n$'', for a non-negative integer) as a
superscript to a letter in order to avoid excessive use of multiple
subscripts.
Other than the exceptions listed below, and unless otherwise noted,
such a superscript will be a label just like a subscript, put on the
right upper corner in order to preserve rooms for subscripts.
In particular, there will be only few occasions where we need to take
a power of a thing, in which case, we shall indicate so.

Major exceptions are as follows.
\begin{itemize}
\item ``$\simp^n$'', ``$d^i$'', ``$\R^n$'' will respectively denote
  the $n$-dimensional symplex, the $i$-th simplicial coface operator,
  $n$-dimensional Euclidean space.
\item ``$f^{-1}$'' for a map or a morphism $f$, will denote the
  inverse of $f$, or the inverse image by the map $f$.
\item ``$\deloop^n$'' and ``$\ddeloop^n$'' will denote the $n$-fold
  applications of the (``delooping'') constructions ``$\deloop$'' and
  ``$\ddeloop$'' respectively (which will have been defined).
\item ``$\wasreteori^n_m$'' will denote the instance of a certain
  construction (defined in this work) which applies to an
  ``$m$-theory'', and produces an ``$n$-theory''.
\item ``$\unity^n_\cat{U}$'' and ``$\univalg^n$'' will respectively
  denote the terminal ($\cat{U}$-graded, unenriched uncoloured)
  $n$-theory and the $n$-dimensional ``universal'' monoid.
\end{itemize}

All other exceptions will be noted at the relevant places.

\section{Symmetric higher theories}
\label{sec:symmetric-theory}

\subsection{Introduction}

After giving a small number of preliminary definitions, we shall give
in this section, the definition of a higher theory with least amount
of structure.

\subsection{The definition}
\label{sec:definition}

\subsubsection{}

Let
\begin{itemize}
\item $\Ord$ denote the category (in the classical, discrete sense) of
  finite ordinals (including the empty set $\kara$),
\item $\Simp$ denote the category (again in the discrete sense) of
  combinatorial simplices, in other words, non-empty finite ordinals.
\end{itemize}

For example, we have objects $[0]=\{0\}$ and $[1]=\{0<1\}$ of $\Simp$,
and the maps in $\Simp$ called the \emph{coface operators}
$d^i\colon[0]\to[1]$, for $i=0,1$, where $d^0(0)=1$, $d^1(0)=0$.

The following is about all of the combinatorics which we need for
this work.
Namely, there is a functor $[-]\colon\Ord\to\Simp^\op$ defined as
follows.

For an object $I\in\Ord$, we define
\[
[I]:=[1]\mathbin{\maeoki{^{d^0}}{\union}^{d^1}_{[0]}}\cdots\mathbin{\maeoki{^{d^0}}{\union}^{d^1}_{[0]}}[1]\qquad(\text{$I$-fold;
  e.g., $[\kara]=[0]$}).
\]
In other words, $[I]=\Union_{i\in I}[1]$ is obtained by gluing for
every pair $i<i+1$ of adjacent elements of $I$, $1$ in the $i$-th
component $[1]$, with $0$ in the ($i+1$)-th component $[1]$.

For a map $\phi\colon I\to J$ in $\Ord$, we note that
$[I]=\Union_{j\in J}[\phi^{-1}j]$, where adjacent components are glued
(similarly to before) at the respective maximal and minimal elements.
We define $[\phi]\colon[I]\from[J]$ in $\Simp$ to be the map obtained
by gluing over $j\in J$, the maps $[1]\to[\phi^{-1}j]$ in $\Simp$
preserving the minimum and the maximum.

\begin{remark}
Let
\begin{itemize}
\item $\Fin$ denote the category (in the discrete sense) of finite sets,
\item $\Fin_*$ denote the category (discrete sense) of pointed finite
  sets.
\end{itemize}
Namely, an object of $\Fin_*$ is a finite set $S$ equipped with a
``base point'' $\ten\in S$, and a morphism is a map $f$ for which we
have $f(*)=*$.

Then there is a commutative square
\[\begin{tikzcd}
\Ord\arrow{d}[swap]{\mathrm{forget}}\arrow{r}{[-]}
&\tsketas{\Simp}{^\op}\arrow{d}{\simp^1/\boundary\simp^1}\\
\Fin\arrow{r}{\blank_+}&\tsketas{\Fin}{_*,}
\end{tikzcd}\]
where $\simp^1$ denotes the simplicial $1$-simplex $\Simp^\op\to\Fin$,
with its boundary $\boundary\simp^1$, and $\blank_+$ is the functor
which externally adds a base point to every finite set.
\end{remark}

\begin{remark}
The functor $[-]\colon\Ord\to\Simp^\op$ has a right adjoint
$\lift{\simp}^1\colon\Simp^\op\to\Ord$, which, as a functor, is the
lift of $\simp^1\colon\Simp^\op\to\Fin$ obtained by putting the
natural total order on the set of all faces of $\simp^1$ of each
fixed dimension.
In particular,
\[
[I]=\Hom_\Simp([0],[I])\equivwith\Hom_\Ord(I,\lift{\simp}^1_0)
\]
as a functor of $I\in\Ord$.
\end{remark}

\begin{notation}\label{notation:skip-0}
In the following, we usually write the elements of an ordinal $I$ as
$1<2<\cdots$ in the ascending order, and then the elements of $[I]$ as
$0<1<2<\cdots$.
\end{notation}

\subsubsection{}

Next, we explain our terminology and notations conerning families,
nerves and operations on them.
Here we introduce only a minimal amount of it; more will be introduced
later during various other definitions.

For $S\in\Fin$, we mean by an \kore{$S$-family} a family
of mathematical objects indexed by the elements of $S$.

Let $\phi\colon T\to S$ be a map in $\Fin$.
Then from an $S$-family $x=(x_s)_{s\in S}$, we obtain a $T$-family
$\phi^*x:=(x_{\phi t})_{t\in T}$.

For $I\in\Ord$, we mean by an \kore{$I$-nerve} in a category, a pair
consisting of
\begin{itemize}
\item a $[I]$-family of objects $x=(x_i)_{i\in[I]}$, and
\item an $I$-family of maps $f=(f_i)_{i\in I}$, where $f_i\colon
  x_{i-1}\to x_i$.
\end{itemize}
Such an $f$ is also called an $I$-nerve \kore{connecting} the
$[I]$-family $x$.

Let $\phi\colon I\to J$ be a map in $\Ord$.
Then from an $[I]$-family $x$, we obtain an $[J]$-family
$\phi_!x:=[\phi]^*x$, and from an $I$-nerve $f$ as above connecting
$x$, we obtain an $J$-nerve $\phi_!f$ connecting $\phi_!x$, defined by
\[
(\phi_!f)_j=f_{[\phi](j)}\cdots f_{[\phi](j-1)+1}.
\]
Note that $\{[\phi](j-1)+1<\cdots<[\phi](j)\}=\phi^{-1}j\sub I$.

\begin{definition}
Let $I\in\Ord$.
Then an $[I]$-family $J$ in either $\Ord$ or $\Fin$, is said to be
\kore{elemental} if $J_{[\pi](1)}=\ten$, the terminal object, where
$\pi$ denotes the unique map $I\to\{1\}$, so $[\pi](1)$ is the maximum
of $[I]$.
\end{definition}

\subsubsection{}
\label{sec:coloured-lax-algebra}

The notion of multicategory is a theorization of the notion of
commutative algebra.
Indeed, a multicategory (enriched in groupoids) is a virtualized form
of a op-lax symmetric monoidal category.

We can also formulate the notion of coloured lax commutative algebra
as follows.

\begin{definition}\label{def:coloured-lax-commutative-algebra}
A \kore{coloured lax commutative algebra} $U$ in a symmetric monoidal
$2$-category $\cat{A}$ consists of the following data.
\begin{itemize}
\item[($0$)] A collection $\Ob U$, whose member is called an
  \emph{object} of $U$, and, for every object $u\in\Ob U$, an object
  $U(u)\in\Ob\cat{A}$.
\item[($1$)] For every finite set $S$, every $S$-family
  $u_0=(u_{0s})_{s\in S}$ of objects of $U$, and every object
  $u_1$, a map $m^U_1(u_0;u_1)\colon U(u_0)\to U(u_1)$, where
  $U(u_0):=\Tensor_{s\in S}U(u_{0s})$.
\item[($2$)] Suppose given
  \begin{itemize}
  \item a finite ordinal $I$,
  \item an elemental $[I]$-family $S=(S_i)_{i\in[I]}$ of finite sets,
    and an $I$-nerve $\phi=(\phi_i)_{i\in I}$ in $\Fin$ connecting
    $S$,
  \item for every $i\in[I]$, an $S_i$-family $u_i=(u_{is})_{s\in S_i}$
    in $\Ob U$.
  \end{itemize}
  Then a $2$-morphism $m^U_2\colon\pi_!m^U_1[u]\to
  m^U_1(u_0;u_{[\pi](1)})$ in $\cat{A}$, where
  \begin{itemize}
  \item $m^U_1[u]:=\bigl(m^U_1(u_{i-1};u_i)\bigr)_{i\in I}$, where
    $m^U_1(u_{i-1};u_i):=\Tensor_{s\in
      S_i}m^U_1(u_{i-1}\resto{s};u_{is})$, where $u_{i-1}\resto{s}$
    denotes the restriction of $u_{i-1}$ to $(\phi_i)^{-1}s\sub
    S_{i-1}$, namely, $m^U_1[u]$ is an $I$-nerve in $\cat{A}$
    connecting the $[I]$-family $U(u):=\bigl(U(u_i)\bigr)_{i\in[I]}$ of
    objects,
  \item $\pi$ denotes the unique map $I\to\{1\}$.
  \end{itemize}
\item[($\infty$)] A datum of \emph{coherence} for the structure.
\end{itemize}
\end{definition}

We shall not write down the details of ($\infty$) since the explicit
form of it is not so important here.
A more general case is treated in Definition \ref{def:theory}, in
particular, Sections \ref{sec:associative-n+2},
\ref{sec:coherence-n+ell}, below.

Now, given a monoidal category $\cat{A}$, by its \emph{categorical
  deloop}, we mean the $2$-category $\deloop\cat{A}$ with a chosen
``base'' object, in which all objects are equivalent, and the
endomorphism monoidal category of the base
object is given an equivalence with $\cat{A}$.
(Note that this determines $\deloop\cat{A}$ uniquely.)

For a symmetric monoidal category $\cat{A}$, we can consider a
multicategory \emph{enriched in $\cat{A}$} as, by definition, a
coloured lax commutative algebra $\cat{U}$ in the \emph{symmetric
  monoidal} $2$-category $\deloop\cat{A}$ (with the induced symmetric
monoidal structure) such that, for every
object $u\in\Ob\cat{U}$, $\cat{U}(u)$ is the unit (i.e., the base)
object of $\deloop\cat{A}$.
For a multicategory $\cat{U}$, we denote the object
$m^\cat{U}_1(u_0;u_1)\in\cat{A}$ by $\Mul_\cat{U}(u_0;u_1)$.
In the case where $\cat{A}$ is the Cartesian symmetric monoidal
category $\Gpd$ of groupoids, $\Mul_\cat{U}(u_0;u_1)$ is the groupoid
of multimaps $u_0\to u_1$ in $\cat{U}$.
For a general $\cat{A}$, the object $\Mul_\cat{U}(u_0;u_1)$ of
$\cat{A}$ is made to behave as if it were formed by (generally
fantastical) multimaps $u_0\to u_1$.

For example, in the case where $\cat{A}$ is the Cartesian symmetric
monoidal category of categories (with a fixed limit for the size), a
multicategory enriched in $\cat{A}$ is the obvious categoried form of
a multicategory.
This is not a useless notion, and the notion of coloured lax
commutative algebra can more generally be defined in a categorified
multicategory, for example.
A more general notion of enrichment will be the subject of Section
\ref{sec:enrichment}.

\subsubsection{}
\label{sec:inductivity}
Starting from commutative algebra and multicategory, there is an
infinite hierarchy of iterated theorizations.
The idea for its construction is simple.
We consider the structure of a multicategory as given by an
associative system of ``composition'' operations for multimaps.
In general, we would like to produce from one kind of associative
system of operations, another kind of associative system of
operations.
This can be done using the following simple observation on the
inductivity of the structure of coherent associativity.

Suppose that we have a collection $m$ of operations (e.g., $m^U_1$ of
Definition \ref{def:coloured-lax-commutative-algebra}) which, if made
coherently associative, would define (in perhaps a special case)
$n$-th theorized version of a multicategory (e.g., a multicategory if
$n=0$).
Suppose further that we actually have a collection $m'$ of (at least
lax) associativity maps for the operations $m$, but that we still do
not have a coherence datum for these maps $m'$, so $m$ is not yet
coherently (lax) associative.
An example of $m'$ is $m^U_2$ for $m=m^U_1$ in Definition
\ref{def:coloured-lax-commutative-algebra}.

In this situation, datum of coherence for the lax associativity has the
following interpretation.
Consider $m'$ as itself a new collection of \emph{operations}.
Then a coherence datum we are looking for amounts precisely to a datum
of coherent \emph{associativity} for the collection $m'$ \emph{of
  operations}.

A ($n+1$)-theorization of multicategory is then obtained by
formalizing the structure given by the collection $m'$ of operations
and its coherent associativity.
See Definition \ref{def:theory} below for the details.

\begin{definition}\label{def:theory}
Let $n\ge 2$ be an integer.
A \kore{symmetric $n$-theory} $\cat{U}$ (which will often be called
simply an ``\kore{$n$-theory}'') enriched in a symmetric monoidal
category $\cat{A}$, consists of data of the forms
specified below as ($0$), ($1$), ($2$) (or just $(0)$, $(1)$ if
$n=2$), ``($k$)'' for every integer $k$ such that $3\le k\le n-1$,
($n$), ($n+1$), ($n+2$), and ``($n+\ell$)'' for every integer $\ell\ge
3$.

We refer to a multicategory enriched in $\cat{A}$ also as a
(symmetric) \kore{$1$-theory} enriched in $\cat{A}$.
We refer to a commutative algebra in $\cat{A}$ also as a
(symmetric) \kore{$0$-theory} (``enriched'') in $\cat{A}$.
\end{definition}

The case where $\cat{A}$ is the Cartesian symmetric monoidal category
$\Gpd$ of groupoids, of $n$-theories, will play important roles, so we
let $\Gpd$ be the default place where an $n$-theory is to be enriched,
and consider an $n$-theory enriched in groupoids as an
\kore{unenriched} $n$-theory.
An \kore{$n$-theory} in the narrower sense will mean an ``unenriched''
$n$-theory.

An \kore{$n$-theory} in the broader sense will mean an enriched (or
unenriched) $n$-theory.
Later, enrichment of an $n$-theory will be considered in a more
general place than a symmetric monoidal category.

Specification of the forms of data will occupy the rest of this
section.

\subsection{Objects of a higher theory}
\label{sec:0}

The form of datum ($0$) for Definition \ref{def:theory} is as follows.
\begin{itemize}
\item[($0$)] A collection $\Ob\cat{U}$, whose member will be called an
  \kore{object} of $\cat{U}$.
\end{itemize}

\subsection{Multimaps in a higher theory}
\label{sec:1}

The form of datum ($1$) for Definition \ref{def:theory} is as follows.
\begin{itemize}
\item[($1$)] Suppose given
  \begin{itemize}
  \item[($0'$)] an elemental $[1]$-family $I=(S,\ten)$ of finite sets,
    and a $\{1\}$-nerve $(\pi\colon S\to \ten)$ in $\Fin$ connecting
    $I$, whole of which is determined by a free choice of $S$,
  \item[($0''$)] an $S$-family $u_0=(u_{0s})_{s\in S}$ of objects of
    $\cat{U}$, and an object $u_1$.
  \end{itemize}
  Then a collection $\Mul^\pi_\cat{U}(u_0;u_1)$ or
  $\Mul^\pi_\cat{U}[u]$ for short, whose member will be called an
  ($S$-ary) (\kore{$1$-})\kore{multimap} $u_0\to u_1$ in $\cat{U}$.
\end{itemize}

Observe that a datum of this form extends as follows.
Suppose given
\begin{itemize}
\item $\psi\colon S\to T$ in $\Fin$,
\item $u_0$ as above,
\item a $T$-family $u_1$ of objects of $\cat{U}$.
\end{itemize}
Then we let $\Mul^\psi_\cat{U}[u]$ denote the collection
of all $T$-families $v=(v_t)_{t\in T}$, where $v_t$ is a
multimap $u_0\resto{t}\to u_{1t}$ in $\cat{U}$, where the
source here is the restriction of the $S$-family $u_0$ to
$\psi^{-1}t\sub S$, so $v_t$ is $\psi^{-1}t$-ary.

\subsection{$2$-multimaps in an $n$-theory, in the case $n\ge 3$}
\label{sec:2}

\subsubsection{}

The form of datum ($2$) for Definition \ref{def:theory} is as
follows.
\begin{itemize}
\item[($2$)] Suppose given
  \begin{itemize}
  \item[($1'$)] an elemental $[1]$-family $I\sp{1}=(I\sp{1}_0,\{1\})$
    in $\Ord$, and an $\{1\}$-nerve $(\pi\colon I\sp{1}_0\to\{1\})$,
  \item[($0'$)] an elemental $[I\sp{1}_0]$-family
    $I\sp{0}=(I\sp{0}_i)_{i\in[I\sp{1}_0]}$ in $\Fin$, and a
    $I\sp{1}_0$-nerve $\phi\sp{0}$ connecting $I\sp{0}$, namely,
    $\phi\sp{0}=(\phi\sp{0}_i)_{i\in I\sp{1}_0}$, where
    $\phi\sp{0}_i\colon I\sp{0}_{i-1}\to I\sp{0}_i$,
  \item[($0''$)] an $I\sp{0}$-family $u\sp{0}$ in $\Ob\cat{U}$,
    namely, $u\sp{0}=(u\sp{0}_i)_{i\in[I\sp{1}_0]}$, where $u\sp{0}_i$
    is an $I\sp{0}_i$-family in $\Ob\cat{U}$,
  \item[($1''$)]
    \begin{itemize}
    \item a \emph{$\phi\sp{0}$-nerve} $u\sp{1}_0$ of multimaps in
      $\cat{U}$, \emph{connecting} $u\sp{0}$, which \emph{by
        definition} means that $u\sp{1}_0=(u\sp{1}_{0i})_{i\in
        I\sp{1}_0}$, where
      $u\sp{1}_{0i}\in\Mul^{\phi\sp{0}_i}_\cat{U}(u\sp{0}_{i-1};u\sp{0}_i)$,
    \item $u\sp{1}_1\in\Mul^{\pi_!\phi\sp{0}}[\pi_!u\sp{0}]$.
    \end{itemize}
  \end{itemize}
  Then a collection $\Mul^\pi_\cat{U}[u\sp{0}](u\sp{1}_0;u\sp{1}_1)$
  or $\Mul^\pi_\cat{U}[u]$ for short, whose member will be called a
  \kore{$2$-multimap} $u\sp{1}_0\to u\sp{1}_1$ in $\cat{U}$.
\end{itemize}

\subsubsection{}

We can extend a datum of this form from above for a similar input
datum with the requirement that the $[I\sp{1}_0]$-family
$I^0$ be elemental discarded.
The idea is to treat such an input datum as an $I^0_{[\pi(1)]}$-family
of \emph{elemental} data, to obtain an $I^0_{[\pi(1)]}$-family of
outputs.

Thus, for input data similar to ($1'$) through ($1''$) above with $I^0$
\emph{non}-elemental, we let $\Mul^\pi_\cat{U}[u]$ denote the
collection whose member is an $I^0_{[\pi](1)}$-family $(v_s)_{s\in
  I^0_{[\pi](1)}}$ of $2$-multimaps in $\cat{U}$, where
$v_s\in\Mul^\pi_\cat{U}\bigl[u^0\resto{s}\bigr]\bigl(u^1_0\bigresto{s};u^1_{1s}\bigr)$,
where
\begin{itemize}
\item $u^0\resto{s}:=\bigl(u^0_i\bigresto{s}\bigr)_{i\in[I^1_0]}$,
  where $u^0_i\bigresto{s}$ is the restriction of $u^0_i$ to
  $I^0_{is}:=(\sakihe{\phi^0}{i})^{-1}s\sub I^0_i$, where
  $\sakihe{\phi^0}{i}:=\phi^0_{[\pi](1)}\cdots\phi^0_{i+1}$,
\item
  $u^1_0\bigresto{s}:=\bigl(\bigl(u^1_0\bigresto{s}\bigr)_i\bigr)_{i\in
    I^1_0}$, where
  $\bigl(u^1_0\bigresto{s}\bigr)_i:=\bigl(u^1_{0it}\bigr)_{t\in I^0_{is}}$,
\end{itemize}
so
$\bigl(u^1_0\bigresto{s}\bigr)_i\in\Mul^{(\phi^0_{/s})_i}\bigl(u^0_{i-1}\bigresto{s};u^0_i\bigresto{s}\bigr)$,
where $(\phi^0_{/s})_i:=\phi^0_i\bigresto{I^0_{i-1,s}}\colon
I^0_{i-1,s}\to I^0_{i,s}$, so $\phi^0_{/s}$ is an $I^1_0$-nerve in
$\Fin$ connecting the \emph{elemental} $[I^1_0]$-family
$(I^0_{is})_{i\in[I^1_0]}$.
In other words, $u^1_0\bigresto{s}$ is a $\phi^0_{/s}$-nerve of
multimaps in $\cat{U}$ connecting $u^0\resto{s}$.

\subsubsection{}
 
We can extend the datum ($2$) further by discarding the requirement
that $I^1$ be elemental, in the similar way as above.

To do this, suppose given
\begin{itemize}
\item a map $\psi\colon I^1_0\to I^1_1$ in $\Ord$,
\item not necessarily elemental data of the form ($0'$) and
  ($0''$) of ($2$),
\item a $\phi^0$-nerve $u^1_0$ and a $\psi_!\phi^0$-nerve $u^1_1$
  respectively, of multimaps in $\cat{U}$ (which we mean to be
  interpreted following the previous step).
\end{itemize}
Then we let $\Mul^\psi_\cat{U}[u]$ denote
the collection whose member is an $I^1_1$-family $v=(v_i)_{i\in
  I^1_1}$ of $2$-multimaps in $\cat{U}$, where
$v_i\in\Mul^{\pi_i}\bigl[u^0\resto{i}\bigr]\bigl(u^1_0\bigresto{i};u^1_{1i}\bigr)$,
where
\begin{itemize}
\item $\pi_i$ denotes the unique map $I^1_{0i}:=\psi^{-1}i\to \ten$, so
  $\psi=\sum_{i\in I^1_1}\pi_i$ (where $\sum_{i\in I^1_1}$ denotes the
  functor which takes the disjoint union equipped with the
  lexicographical order),
\item $u^0\resto{i}$ denotes the restriction of $u^0$ to
  $[I^1_{0i}]\sub[I^1_0]$,
\item $u^1_0\bigresto{i}$ denotes the restriction of
  $u^1_0$ to $I^1_{0i}$,
\end{itemize}
so $u^1_0\bigresto{i}$ is a ($\phi^0\resto{i}$)-nerve connecting
$u^0\resto{i}$, where $\phi^0\resto{i}$ denotes the restriction of
$\phi^0$ to $I^1_{0i}$.

\subsection{$k$-multimaps in a higher theory}
\label{sec:k}

\subsubsection{}

The form of datum ($k$) for $3\le k\le n-1$ for Definition
\ref{def:theory}, is specified inductively as follows.
\begin{itemize}
\item[($k$)] Suppose given
  \begin{itemize}
  \item[($k-1'$)] an elemental $[1]$-family
    $I^{k-1}=(I^{k-1}_0,\{1\})$ in $\Ord$, and an $\{1\}$-nerve
    $(\pi\colon I^{k-1}_0\to\{1\})$,
  \item[($k-2'$)] an elemental $[I^{k-1}_0]$-family
    $I^{k-2}=(I^{k-2}_i)_{i\in[I^{k-1}_0]}$ in $\Ord$, and an
    $I^{k-1}_0$-nerve $\phi^{k-2}=(\phi^{k-2}_i)_{i\in I^{k-1}_0}$
    connecting $I^{k-2}$,
  \item ($k-3'$) through ($k-3''$) of ($k-1$),
  \item[($k-2''$)] a $I^{k-2}$-family
    $u^{k-2}=(u^{k-2}_i)_{i\in[I^{k-1}_0]}$ of ($k-2$)-multimaps in
    $\cat{U}$, where the $I^{k-2}_i$-family $u^{k-2}_i$ is in fact a
    $\kakowanu{\maekara{\phi^{k-2}}{i}}_!\phi^{k-3}$-nerve (see
    ($k-1''$) below) of ($k-2$)-multimaps in $\cat{U}$ (where
    $\maekara{\phi}{i}:=\phi_i\cdots\phi_1$),
    connecting $\kakowanu{\maekara{\phi^{k-2}}{i}}_!u^{k-3}$,
  \item[($k-1''$)]
    \begin{itemize}
    \item a \emph{$\phi^{k-2}$-nerve} $u^{k-1}_0$ of ($k-1$)-multimaps
      \emph{connecting} $u^{k-2}$ in $\cat{U}$, which \emph{by
        definition} means that $u^{k-1}_0=(u^{k-1}_{0i})_{i\in
        I^{k-1}_0}$, where
      $u^{k-1}_{0i}\in\Mul^{\phi^{k-2}_i}_\cat{U}[u^{\le
        k-2}\resto{i}]$ (see below), where $u^{\le k-2}\resto{i}$
      consists of
      \[
      u^{\le
        k-4}:=(u^\nu)_{0\le\nu\le
        k-4},\;{(\maekara{\phi^{k-2}}{i-1})}_!u^{k-3},\;u^{k-2}\resto{i},
      \]
      where $u^{k-2}\resto{i}$ denotes the
      restriction of $u^{k-2}$ to $\{i-1,i\}\sub[I^{k-1}_0]$,
    \item $u^{k-1}_1\in\Mul^{\pi_!\phi^{k-2}}[\pi_!u^{\le k-2}]$,
      where $\pi_!u^{\le k-2}$ consists of
      \[
      u^{\le k-4},\;(\pi_!\phi^{k-2})_!u^{k-3},\;\pi_!u^{k-2}.
      \]
    \end{itemize}
  \end{itemize}
  Then a collection $\Mul^\pi_\cat{U}[u^{\le
    k-2}](u^{k-1}_0;u^{k-1}_1)$ or $\Mul^\pi_\cat{U}[u]$ for short,
  whose member will be called
  a \kore{$k$-multimap} $u^{k-1}_0\to u^{k-1}_1$ in $\cat{U}$.
\end{itemize}

\begin{definition}
We refer to a datum of the form $(I;\pi,\phi)$, where
$I:=(I^\nu)_{0\le\nu\le k-1}$,
$\phi:=(\phi^\nu)_{0\le\nu\le k-2}$, specified by ($k-1'$)
through ($0'$) above, as the \kore{arity of a} \kore{$k$-multimap} in
a symmetric higher theory.

We refer to a datum of the form $u=(u^\nu)_{0\le\nu\le k-1}$ specified
by ($0''$) through ($k-1''$) above (by induction in $k$), as the
\kore{type of a $k$-multimap} in $\cat{U}$, of \kore{arity}
$(I;\pi,\phi)$.
\end{definition}

\begin{remark}
Even though we have not yet specified the form of the rest of data
for $\cat{U}$, note that the notion of the type of a $k$-multimap ``in
$\cat{U}$'' makes sense as soon as data of the forms ($0$) through
($k-1$) are given ``for $\cat{U}$''.
\end{remark}

A datum of the form ($k$) above extends for a similar input datum with
the elementality requirements discarded.
This can be done by induction, starting from the elemental case above,
as follows.

\subsubsection{}

Fix an integer $\nu$ such that $1\le\nu\le k-1$, and suppose as an
inductive hypothesis, that we have extended the datum ($k$)
for input data similar to ($k-1'$) through ($k-1''$) above, where the
families $I^0$ through $I^{\nu-2}$ are allowed to be non-elemental.
Then we extend the datum ($k$) for input data with the families up to
$I^{\nu-1}$ allowed to be non-elemental, as follows.

Suppose given data similar to ($k-1'$) through ($k-1''$) above, where
the families $I^0$ through $I^{\nu-1}$ are allowed to be
non-elemental (which we mean to be interpreted following the previous
inductive step).
Then we let $\Mul^\pi_\cat{U}[u]$ denote the collection whose member
is an $I^{\nu-1}_{[\pi^\nu](1)}$-family $(v_i)_{i\in
  I^{\nu-1}_{[\pi^\nu](1)}}$,
where $\pi^\nu$ denotes the unique map $I^\nu_0\to\{1\}$, and
$v_i\in\Mul^\pi_\cat{U}[u\resto{i}]$, where $u\resto{i}$ consists of
\[
u^{\le\nu-4},\;{(\maekara{\phi^{\nu-2}}{i-1})}_!u^{\nu-3},\;u^{\ge\nu-2}\resto{i}:=(u^\kappa\resto{i})_{\kappa\ge\nu-2},
\]
where if $\nu\le k-2$,
\begin{itemize}
\item $u^{\ge\nu-2,\le
    k-3}\resto{i}=(u^\kappa\resto{i})_{\nu-2\le\kappa\le k-3}$ are as
  already defined by the previous step of the induction on $k$ (see
  the case $\nu=k-1$ below for $u^{\nu-2}\resto{i}$ and
  $u^{\nu-1}\resto{i}$, and the next
  point for $u^{\ge \nu,\le k-3}\resto{i}$),
\item
  $u^{k-2}\resto{i}:=\bigl(u^{k-2}_j\bigresto{i}\bigr)_{j\in[I^{k-1}_0]}$,
  where $u^{k-2}_j\bigresto{i}$ is as already defined by the previous
  step of the induction on $k$ (see below),
\end{itemize}
and if $\nu=k-1$,
\begin{itemize}
\item $u^{\nu-2}\resto{i}$ denotes the restriction of $u^{\nu-2}$ to
  $[(\pi^\nu_!\phi^{\nu-1})^{-1}i]\sub[I^{\nu-1}_0]$,
\item
  $u^{\nu-1}\resto{i}:=\bigl(u^{\nu-1}_j\bigresto{i}\bigr)_{j\in[I^\nu_0]}$,
  where $u^{\nu-1}_j\bigresto{i}$ denotes the
  $\bigl(\maekara{(\phi^{\nu-1}_{/i})}{j}\bigr)_!\bigl(\phi^{\nu-2}\resto{i}\bigr)$-nerve
  connecting
  $\bigl(\maekara{(\phi^{\nu-1}_{/i})}{j}\bigr)_!\bigl(u^{\nu-2}\resto{i}\bigr)$,
  obtained by restricting $u^{\nu-1}_j$ to
  $I^{\nu-1}_{ji}:=(\sakihe{\phi^{\nu-1}}{j})^{-1}i\sub
  I^{\nu-1}_j$, where
  \begin{itemize}
  \item $\phi^{\nu-1}_{/i}$ denotes the $I^\nu_0$-nerve in $\Ord$
    connecting the \emph{elemental} $[I^\nu_0]$-family
    $(I^{\nu-1}_{ji})_{j\in[I^\nu_0]}$, obtained by restricting
    $\phi^{\nu-1}$,
  \item $\phi^{\nu-2}\resto{i}$ is the $I^{\nu-1}_{0i}$-nerve obtained
    by restricting $\phi^{\nu-2}$, connecting the
    $[I^{\nu-1}_{0i}]$-family $I^{\nu-2}\resto{i}$ obtained by
    restricting $I^{\nu-2}$,
  \end{itemize}
  so
  \begin{align*}
  \bigl(\maekara{(\phi^{\nu-1}_{/i})}{j}\bigr)_!\bigl(\phi^{\nu-2}\resto{i}\bigr)
    &=\bigl(\kakowanu{\maekara{\phi^{\nu-1}}{j}}_!\phi^{\nu-2}\bigr)\bigresto{I^{\nu-1}_{ji}},\\
  \bigl(\maekara{(\phi^{\nu-1}_{/i})}{j}\bigr)_!\bigl(u^{\nu-2}\resto{i}\bigr)
    &=\bigl(\kakowanu{\maekara{\phi^{\nu-1}}{j}}_!u^{\nu-2}\bigr)\bigresto{[I^{\nu-1}_{ji}]},
  \end{align*}
\end{itemize}
and for any $\nu$, $u^{k-1}_0\bigresto{i}=(u^{k-1}_{0ji})_{j\in
  I^{k-1}_0}$, and $u^{k-1}_1\bigresto{i}=(u^{k-1}_{1i})$ are as
already specified by the previous step of the induction on $k$.
Note (see below) that, by induction on $k$, $u^{k-1}_{0j}$ for $j\in
I^{k-1}_0$ is an $I^{k-2}_{[\pi](1)}$-family $(u^{k-1}_{0ji})_{i\in
  I^{k-2}_{[\pi^1](1)}}$, where
$u^{k-1}_{0ji}\in\Mul^{\phi^{k-2}_i}[u^{\le k-2}\resto{i,j}]$, and
similarly for $u^{k-1}_1$.

\subsubsection{}

Finally, the datum ($k$) extends for input data with $I^{k-1}$
non-elemental, as follows.
Suppose given
\begin{itemize}
\item a map $\psi\colon I^{k-1}_0\to I^{k-1}_1$ in $\Ord$,
\item not necessarily elemental data of the form ($k-2'$) through
  ($k-2''$) above,
\item a $\phi^{k-2}$-nerve $u^{k-1}_0$ connecting $u^{k-2}$, and a
  $\psi_!\phi^{k-2}$-nerve $u^{k-1}_1$ connecting $\psi_!u^{k-2}$
  respectively, of ($k-1$)-multimaps in $\cat{U}$.
\end{itemize}
Then we let $\Mul^\psi_\cat{U}[u]$ denote the collection whose member
is an $I^{k-1}_1$-family $v=(v_i)_{i\in I^{k-1}_1}$
of $k$-multimaps in $\cat{U}$, where $v_i\in\Mul^{\pi_i}[u\resto{i}]$,
where
\begin{itemize}
\item $\pi_i$ denotes the unique map $I^{k-1}_{0i}:=\psi^{-1}i\to \ten$,
  so $\psi=\sum_{i\in I^{k-1}_1}\pi_i$,
\item $u\resto{i}$ consists of
  \[
  u^{\le
    k-4},\;{(\maekara{\phi^{k-2}}{[\psi](i-1)})}_!u^{k-3},\;u^{k-2}\resto{i},\;u^{k-1}\resto{i},
  \]
  where
  \begin{itemize}
  \item $u^{k-2}\resto{i}$ denotes the restriction of
    $u^{k-2}$ to $[I^{k-1}_{0i}]\sub[I^{k-1}_0]$,
  \item
    $u^{k-1}\bigresto{i}:=\bigl(u^{k-1}_j\bigresto{i}\bigr)_{j\in[1]}$,
    where $u^{k-1}_0\bigresto{i}$ denotes the restriction of
    $u^{k-1}_0$ to $I^{k-1}_{0i}$, and
    $u^{k-1}_1\bigresto{i}:=(u^{k-1}_{1i})$,
  \end{itemize}
so $u^{k-1}_0\bigresto{i}$ is a ($\phi^{k-2}\resto{i}$)-nerve
connecting $u^{k-2}\resto{i}$, where $\phi^{k-2}\resto{i}$ denotes the
restriction of $\phi^{k-2}$ to $I^{k-1}_{0i}$.
\end{itemize}

\subsubsection{}

This allows us to go to the next inductive step.

\subsection{The object ``formed by $n$-multimaps in an $n$-theory''}
\label{sec:n}

The form of datum ($n$) for Definition \ref{def:theory} is as
follows.
\begin{itemize}
\item[($n$)] Suppose given the type $u$ of an $n$-multimap in
  $\cat{U}$ of arity given as $(I;\pi,\phi)$.
  Namely, suppose given a set of data similar to ($k-1'$) through
  ($k-1''$) of ``($k$)'' in Section \ref{sec:k}, but with $k$ substituted
  by $n$ (where $I^{k-2}$ should be a family in $\Fin$ if $n=2$), so
  these will be ($n-1'$) through ($n-1''$) here.
  Then an object $\Mul^\pi_\cat{U}[u^{\le n-2}](u^{n-1}_0;u^{n-1}_1)$
  of $\cat{A}$, or $\Mul^\pi_\cat{U}[u]$ for short.
  In the case where $\cat{A}$ is $\Gpd$ or some other category so that
  $\Mul^\pi_\cat{U}[u]$ can have its objects, then those objects will
  be called \kore{$n$-multimaps} $u^{n-1}_0\to u^{n-1}_1$ in
  $\cat{U}$.
  For a general $\cat{A}$, we shall call $\Mul^\pi_\cat{U}[u]$ the
  object ``of \emph{$n$-multimaps}''.
\end{itemize}

A datum of this form extends for a non-elemental input datum just as
the datum ``($k$)'' did in Section \ref{sec:k}.

To see this, for the purpose of induction starting from the elemental
case above, fix an integer $\nu$ such that $1\le\nu\le n-1$, and
suppose given data similar to ($n-1'$) through ($n-1''$)
above, where the families $I^0$ through $I^{\nu-1}$ are allowed to be
non-elemental.
Then we define $\Mul^\pi_\cat{U}[u]:=\Tensor_{i\in
  I^{\nu-1}_{[\pi^\nu](1)}}\Mul^\pi_\cat{U}[u\resto{i}]$, which makes
sense by induction on $\nu$.

Suppose given next, instead of $\pi\colon I^{n-1}_0\to\{1\}$
above, a map $\psi\colon I^{n-1}_0\to I^{n-1}_1$ in $\Ord$, and
suppose $u^{n-1}_1$ is now an $\psi_!\phi^{n-2}$-nerve of
($n-1$)-multimaps connecting $\psi_!u^{n-2}$.
Then we let $\pi_i$ for $i\in I^{n-1}_1$ denote the unique map
$\psi^{-1}i\to \ten$ (so $\psi=\sum_i\pi_i$), and define
$\Mul^\psi_\cat{U}[u]:=\Tensor_{i\in I^{n-1}_1}\Mul^{\pi
  _i}[u\resto{i}]$.

\subsection{Composition of $n$-multimaps}
\label{sec:n+1}

The form of datum ($n+1$) for Definition \ref{def:theory} is as
follows.
\begin{itemize}
\item[($n+1$)] Suppose given the arity $(I;\pi,\phi)$ of an
  ($n+1$)-multimap in a symmetric higher theory, namely
  \begin{itemize}
  \item ($k-1'$) and ($k-2'$) of ``($k$)'' in Section \ref{sec:k}, but
    with $k$ substituted by $n+1$, so these will be ($n'$) and ($n-1'$)
    here.
  \item ($n-2'$) through ($0'$) of ($n$) in Section \ref{sec:n},
  \end{itemize}
  as well as
  \begin{itemize}
  \item ($0''$) through ($n-2''$) of ($n$) in Section \ref{sec:n},
  \item ($k-2''$) of ``($k$)'' in Section \ref{sec:k}, but
    with $k$ substituted by $n+1$, so this will be ($n-1''$) here.
  \end{itemize}
  Then a map
  $m^\cat{U}_1(\pi)\colon\Mul^{\phi^{n-1}}_\cat{U}[u]\to\Mul^{\pi_!\phi^{n-1}}_\cat{U}[\pi_!u]$
  in $\cat{A}$, where
  \begin{itemize}
  \item the source of $m^\cat{U}_1(\pi)$ is the object $\Tensor_{i\in
      I^n_0}\Mul^{\phi^{n-1}_i}[u\resto{i}]$ of $\cat{A}$ (``of
    $\phi^{n-1}$-nerves of $n$-multimaps connecting $u^{n-1}$ in
    $\cat{U}$''), where $u\resto{i}$ consists of
    \[
    u^{\le
      n-3},\;{(\maekara{\phi^{n-1}}{i-1})}_!u^{n-2},\;u^{n-1}\resto{i},
    \]
  \item $\pi_!u$ consists of $u^{\le n-2}$, $\pi_!u^{n-1}$.
  \end{itemize}
  The map $m^\cat{U}_1(\pi)$ will be called the \kore{composition}
  operation for $n$-multimaps.
\end{itemize}

\begin{definition}
We refer to a datum of the form $u=(u^\nu)_{0\le\nu\le n-1}$ specified
by ($0''$) through ($n-1''$) above, as the \kore{type of a
  $\phi^{n-1}$-nerve of $n$-multimaps} in $\cat{U}$.
\end{definition}

A datum of the form ($n+1$) above extends for a non-elemental input
datum as follows.

To begin with, fix an integer $\nu$ such that $1\le\nu\le n-1$, and
suppose given data similar to ($n'$) through ($n-1''$)
above, where the families $I^0$ through $I^{\nu-1}$ are allowed to be
non-elemental.
Then we define
$m^\cat{U}_1(\pi)\colon\Mul^{\phi^{n-1}}_\cat{U}[u]\to\Mul^{\pi_!\phi^{n-1}}_\cat{U}[\pi_!u]$
as the monoidal product over $i\in I^{\nu-1}_{[\pi^\nu](1)}$ of the
maps
$m^\cat{U}_1(\pi)\colon\Mul^{\phi^{n-1}}_\cat{U}[u\resto{i}]\to\Mul^{\pi_!\phi^{n-1}}[\pi_!u\resto{i}]$,
which makes sense by induction on $\nu$.

Next, suppose given data similar to ($n'$) through ($n-1''$)
above, where the families $I^0$ through $I^{n-1}$ are allowed to be
non-elemental.
Then we define
$m^\cat{U}_1(\pi)\colon\Mul^{\phi^{n-1}}_\cat{U}[u]\to\Mul^{\pi_!\phi^{n-1}}_\cat{U}[\pi_!u]$
as the monoidal product over $i\in I^{n-1}_{[\pi](1)}$ of the maps
$m_1(\pi)\colon\Mul^{\phi^{n-1}_{/i}}[u\resto{i}]\to\Mul^{\pi_!\phi^{n-1}_{/i}}[\pi_!u\resto{i}]$.

Finally, suppose given, instead of $\pi\colon I^n_0\to\{1\}$ above,
a map $\psi\colon I^n_0\to I^n_1$ in $\Ord$.
Then we let $\pi_i$ for $i\in I^n_1$ denote the unique map
$\psi^{-1}i\to \ten$ (so $\psi=\sum_i\pi_i$), and define
$m^\cat{U}_1(\psi)\colon\Mul^{\phi^{n-1}}[u]\longto\Mul^{\psi_!\phi^{n-1}}[\psi_!u]$,
where $\psi_!u$ consists of $u^{\le n-2}$, $\psi_!u^{n-1}$, as the
monoidal product over $i\in I^n_1$ of the maps
$m_1(\pi_i)\colon\Mul^{\phi^{n-1}\resto{i}}[u\resto{i}]\to\Mul^{(\psi_!\phi^{n-1})_i}[(\psi_!u)\resto{i}]$,
where $u\resto{i}$ consists of
\[
u^{\le
  n-3},\;{(\maekara{\phi^{n-1}}{[\psi](i-1)})}_!u^{n-2},\;u^{n-1}\resto{i},
\]
so $(\psi_!u)\resto{i}$ consists of
\[
u^{\le
  n-3},\;\bigl(\maekara{(\psi_!\phi^{n-1})}{i-1}\bigr)_!u^{n-2},\;\bigl(\psi_!u^{n-1}\bigr)\bigresto{i}.
\]

\subsection{The associativity isomorphism for the composition}
\label{sec:associative-n+2}

The form of datum ($n+2$) for Definition \ref{def:theory} is as
follows.
\begin{itemize}
\item[($n+2$)] Suppose given the arity $(I;\pi,\phi)$ of an
  ($n+2$)-multimap in a symmetric higher theory, and the type $u$ of a
  $\phi^{n-1}$-nerve of $n$-multimaps in $\cat{U}$.
  Then a $2$-isomorphism
  $m^\cat{U}_2(\pi)\colon\pi_!m^\cat{U}_1(\phi^n)\equivto
  m^\cat{U}_1(\pi_!\phi^n)$ in $\cat{A}$, where $m^\cat{U}_1(\phi^n)$
  denotes the $I^{n+1}_0$-nerve obtained by indexing with $i\in
  I^{n+1}_0$, the maps
  \[
  m_1(\phi^n_i)\colon\Mul^{{(\maekara{\phi^n}{i-1})}_!\phi^{n-1}}[{(\maekara{\phi^n}{i-1})}_!u]\longto\Mul^{{(\maekara{\phi^n}{i})}_!\phi^{n-1}}[{(\maekara{\phi^n}{i})}_!u].
  \]
\end{itemize}

A datum of this form extends for a non-elemental input datum as
follows.

To begin with, fix an integer $\nu$ such that $1\le\nu\le n-1$, and
suppose given data similar to ($n+1'$) through ($n-1''$)
above, where the families $I^0$ through $I^{\nu-1}$ are allowed to be
non-elemental.
Then we define $m^\cat{U}_2(\pi)\colon\pi_!m^\cat{U}_1(\phi^n)\equivto
m^\cat{U}_1(\pi_!\phi^n)$ as the monoidal product over $i\in
I^{\nu-1}_{[\pi^\nu](1)}$ of the $2$-isomorphisms
\[
m^\cat{U}_2(\pi)\colon\pi_!m_1(\phi^n)\longequivto
m_1(\pi^n)\colon\Mul^{\phi^{n-1}}[u\resto{i}]\longto\Mul^{{\pi^n}_!\phi^{n-1}}[{\pi^n}_!u\resto{i}],
\]
which makes sense by induction on $\nu$.

Next, suppose given data similar to ($n+1'$) through ($n-1''$)
above, where the families $I^0$ through $I^{n-1}$ are allowed to be
non-elemental.
Then we define $m^\cat{U}_2(\pi)\colon\pi_!m_1(\phi^n)\equivto
m_1(\pi_!\phi^n)$ as the monoidal product over $i\in
I^{n-1}_{[\pi_!\phi^n](1)}$ of the $2$-isomorphisms
\[
m_2(\pi)\colon\pi_!m_1(\phi^n)\longequivto
m_1(\pi^n)\colon\Mul^{\phi^{n-1}_{/i}}[u\resto{i}]\longto\Mul^{{\pi^n}_!\phi^{n-1}_{/i}}[{\pi^n}_!u\resto{i}].
\]

Next, suppose given data similar to ($n+1'$) through ($n-1''$)
above, where the families $I^0$ through $I^n$ are allowed to be
non-elemental.
Then we write for $i\in I^n_{[\pi](1)}$,
$\pi^n_i:=\pi_!(\phi^n_{/i})$ (so $\pi_!\phi^n=\sum_i\pi^n_i$), and
define $m_2(\pi)\colon\pi_!m_1(\phi^n)\equivto m_1(\pi_!\phi^n)$ as
the monoidal product over $i\in I^n_{[\pi](1)}$ of the
$2$-isomorphisms
\[
m_2(\pi)\colon\pi_!m_1(\phi^n_{/i})\longequivto
m_1(\pi^n_i)\colon\Mul^{\phi^{n-1}\resto{i}}[u\resto{i}]\longto\Mul^{{\pi^n_i}_!(\phi^{n-1}\resto{i})}[({\pi^n}_!u)\resto{i}].
\]

Finally, suppose given, instead of $\pi\colon I^{n+1}_0\to\{1\}$
above, a map $\psi\colon I^{n+1}_0\to I^{n+1}_1$ in $\Ord$.
Then we let $\pi_i$ for $i\in I^{n+1}_1$ denote the unique map
$\psi^{-1}i\to \ten$ (so $\psi=\sum_i\pi_i$), and define the isomorphism
$m_2(\psi)\colon\psi_!m_1(\phi^n)\equivto m_1(\psi_!\phi^n)$ of
$I^{n+1}_1$-nerves in $\cat{A}$ as the
family indexed by $i\in I^{n+1}_1$ of the isomorphisms
\[
m_2(\pi_i)\colon{\pi_i}_!m_1\bigl(\phi^n\resto{i}\bigr)\longequivto
m_1\bigl((\psi_!\phi^n)_i\bigr).
\]

\subsection{Coherence for the associativity}
\label{sec:coherence-n+ell}

\subsubsection{}

The form of datum ($n+\ell$) for $\ell\ge 3$ for Definition
\ref{def:theory}, is specified inductively as follows.
\begin{itemize}
\item[($n+\ell$)] Suppose given the arity $(I;\pi,\phi)$ of an
  ($n+\ell$)-multimap in a symmetric higher theory, and the type $u$
  of a $\phi^{n-1}$-nerve of $n$-multimaps in $\cat{U}$.
  Then an $\ell$-isomorphism
  \[
  m^\cat{U}_\ell(\pi)\colon\pi_!m^\cat{U}_{\ell-1}(\phi^{n+\ell-2})\longequivto
  m^\cat{U}_{\ell-1}(\pi_!\phi^{n+\ell-2})
  \]
  in $\cat{A}$, where $m^\cat{U}_{\ell-1}(\phi^{n+\ell-2})$ denotes the
  $I^{n+\ell-1}_0$-nerve of ($\ell-1$)-isomorphisms obtained by
  indexing with $i\in I^{n+\ell-1}_0$ the isomorphisms
  \[
  m_{\ell-1}(\phi^{n+\ell-2}_i)\colon
  m^\cat{U}_{\ell-2}(\kakowanu{\maekara{\phi^{n+\ell-2}}{i-1}}_!\phi^{n+\ell-3})\longequivto
  m^\cat{U}_{\ell-2}(\kakowanu{\maekara{\phi^{n+\ell-2}}{i}}_!\phi^{n+\ell-3})
  \]
  of ($\ell-2$)-isomorphisms
  \[
  {\pi^{n+\ell-3}}_!m_{\ell-3}(\phi^{n+\ell-4})\longequivto
  m_{\ell-3}({\pi^{n+\ell-3}}_!\phi^{n+\ell-4})
  \]
  (where $\pi^{n+\ell-3}:=(\pi_!\phi^{n+\ell-2})_!\phi^{n+\ell-3}$) in
  $\cat{A}$,
  or $\Mul^{\phi^{n-1}}[u]\to\Mul^{{\pi^n}_!\phi^{n-1}}[{\pi^n}_!u]$
  if $\ell=3$.
\end{itemize}

A datum of this form extends for a non-elemental input datum as
follows.

The initial step is a similar induction as before.
Fix an integer $\nu$ such that $1\le\nu\le n+\ell-1$, and suppose
given data similar to ($n+\ell-1'$) through ($n-1''$) above, where the
families $I^0$ through $I^{\nu-1}$ are allowed to be non-elemental.
Then we define
$m^\cat{U}_\ell(\pi)\colon\pi_!m_{\ell-1}(\phi^{n+\ell-2})\equivto
m_{\ell-1}(\pi_!\phi^{n+\ell-2})$ as
\begin{itemize}
\item if $\nu\le n+1$, the monoidal product over
  $I^{\nu-1}_{[\pi^\nu](1)}$,
\item if $\nu\ge n+2$, the ($n+\ell+1-\nu$)-isomorphism of
  $I^{\nu-1}_{[\pi^\nu](1)}$-nerves of ($\nu-n-1$)-isomorphisms (or
  $1$-morphisms if $\nu=n+2$) in $\cat{A}$, given by the family
  indexed by $I^{\nu-1}_{[\pi^\nu](1)}$,
\end{itemize}
of $m^\cat{U}_\ell(\pi)$ defined for each $i\in
I^{\nu-1}_{[\pi^\nu](1)}$ as an instance of the previous inductive
step.

For instance, as the case $\nu=n+\ell-1$, suppose given data similar to
($n+\ell-1'$) through ($n-1''$) above, where the families $I^0$ through
$I^{n+\ell-2}$ are allowed to be non-elemental.
Then we write for $i\in I^{n+\ell-2}_{[\pi](1)}$,
$\pi^{n+\ell-2}_i:=\pi_!\phi^{n+\ell-2}_{/i}$ (so
$\pi_!\phi^{n+\ell-2}=\sum_i\pi^{n+\ell-2}_i$), and define the isomorphism
$m_\ell(\pi)\colon\pi_!m_{\ell-1}(\phi^{n+\ell-2})\equivto
m_{\ell-1}(\pi_!\phi^{n+\ell-2})$ of isomorphisms
\[
\bigl(\pi_!\phi^{n+\ell-2}\bigr)_!m_{\ell-2}\bigl(\phi^{n+\ell-3}\bigr)\longequivto
m_{\ell-2}\bigl((\pi_!\phi^{n+\ell-2})_!\phi^{n+\ell-3}\bigr)
\]
of $I^{n+\ell-2}_{[\pi](1)}$-nerves of ($\ell-2$)-isomorphisms (or
$1$-morphisms if $\ell=3$) in $\cat{A}$, as
given by the family indexed by $i\in I^{n+\ell-2}_{[\pi](1)}$ of
\begin{multline*}
m_\ell(\pi)\colon \pi_!m_{\ell-1}(\phi^{n+\ell-2}_{/i})\longequivto
m_{\ell-1}(\pi^{n+\ell-2}_i)\colon\\
{\pi^{n+\ell-2}_i}_!m_{\ell-2}\bigl(\phi^{n+\ell-3}\resto{i}\bigr)\longequivto m_{\ell-2}\bigl(\bigl((\pi_!\phi^{n+\ell-2})_!\phi^{n+\ell-3}\bigr)_i\bigr).
\end{multline*}

Finally, suppose given, instead of $\pi\colon I^{n+\ell-1}_0\to\{1\}$
above, a map $\psi\colon I^{n+\ell-1}_0\to I^{n+\ell-1}_1$ in $\Ord$.
Then we let $\pi_i$ for $i\in I^{n+\ell-1}_1$ denote the unique map
$\psi^{-1}i\to \ten$ (so $\psi=\sum_i\pi_i$), and define the isomorphism
\[
m_\ell(\psi)\colon\psi_!m_{\ell-1}(\phi^{n+\ell-2})\longequivto
m_{\ell-1}(\psi_!\phi^{n+\ell-2})
\]
of $I^{n+\ell-1}_1$-nerves of ($\ell-1$)-isomorphisms in $\cat{A}$, as
given by
the family indexed by $i\in I^{n+\ell-1}_1$ of the isomorphisms
\[
m_\ell(\pi_i)\colon{\pi_i}_!m_{\ell-1}\bigl(\phi^{n+\ell-2}\resto{i}\bigr)\longequivto
m_{\ell-1}\bigl((\psi_!\phi^{n+\ell-2})_i\bigr).
\]

\subsubsection{}

We can thus proceed to the next inductive step, and this completes
Definition \ref{def:theory}.

\section{Simple variants and basic constructions}
\label{sec:variant-construction}

\subsection{Introduction}

In this section, we shall first discuss a planar variant of higher
theories, which iteratevely theorize associative algebra.
In particular, we shall find the structure of a planar ($n-1$)-theory
formed by endomorphisms in a symmetric $n$-theory.
This will lead to a discussion of other theorized structures similarly
residing in a symmetric higher theory.
We shall further discuss less coloured variants of higher theory,
relation of higher theory to higher categorified structure, and a
construction for a higher theory which generalizes the ``delooping''
construction for a symmetric monoidal category.

\subsection{Planar theories}
\label{sec:planar-theory}

\subsubsection{}

The notion of symmetric multicategory had variants such as planar and
braided.
While a generalization of these will appear in Section
\ref{sec:grading}, the notion of \emph{planar} $n$-theory is
particularly simple to describe, and turns out to be also fundamental,
so we shall discuss it here.

The definition of a \kore{planar $n$-theory} is the same as the
definition of a symmetric $n$-theory except that one uses the
category $\Ord$ instead of $\Fin$.
Namely, $I^0$ (and $S$) appearing in the definition \ref{def:theory}
of an $n$-theory should now be in $\Ord$, and everything else is as it
makes sense under this modification.

In particular, one obtains a planar $n$-theory from a symmetric
$n$-theory $\cat{U}$ by restricting the data defining $\cat{U}$,
through the forgetful functor $\Ord\to\Fin$.

\subsubsection{}

The notion of planar $n$-theory is fundamental for the following
reason.
Given a symmetric $n$-theory $\cat{U}$, and its object $x$, the
structure of a planar ($n-1$)-theory underlies the
structure formed by endomorphisms (i.e., unary endomultimaps) of $x$.

The idea is that if one intends to take as the part ($0'$) of an
input datum for $\cat{U}$, the constant elemental $I^1_0$-nerve
$\ten\to\cdots\to\ten$, and as the part ($0''$), the constant family
at $x$,
then the rest of the required input data is of the same form as the
form for an input datum for a \emph{planar} ($n-1$)-theory.

Thus, a planar ($n-1$)-theory $\cat{V}=\cat{M}ap_\cat{U}(x,x)$ is
obtained as follows.
Suppose inductively in $k$, that an input for the datum ($k$) for the
planar ($n-1$)-theory $\cat{V}$ is given by $J^\nu$,
$\psi^\nu$, $v^\nu$ for all integers $\nu$ in the suitable range.
Then one obtains an input for the datum ($k+1$) for $\cat{U}$ by
letting
\begin{itemize}
\item $I^\nu:=J^{\nu-1}$, $\phi^\nu:=\psi^{\nu-1}$ for $\nu\ge 1$,
\item $I^0$ and $\phi^0$ constant as above, as well as $u^0$ constant
  at $x$,
\item $u^\nu:=v^{\nu-1}$ for $\nu\ge 1$,
\end{itemize}
so we use the output for this by $\cat{U}$ as the datum ($k$) for
$\cat{V}$ for the original input.

For example, the collection of the objects of $\cat{V}$ is
$\Mul^\pi_\cat{U}(x;x)$, where $\pi$ is the identity map of
$I^0_1=\{1\}$.

\subsection{More related higher theorizations.}

In Section \ref{sec:planar-theory}, we have found the structure of a
planar ($n-1$)-theory within the structure of every $n$-theory.
The case $n=1$ of this is the associative algebra of endomorphisms
within a multicategory, and by no accident, a planar $n$-theory is an
$n$-th theorization of an associative algebra.

Since we can also find other structures within the structure of a
multicategory, we can find higher theorized forms of them within the
structure of an $n$-theory.

For example, if we focus on, instead of endormophisms on a selected
object, all unary multimaps between arbitrary objects within an
$n$-theory, then we find that they naturally form a structure which is
an ($n-1$)-th theorization of the structure of a category.
We have seen examples of theorized categories in Section
\ref{sec:theorize-category}.
We have also noted there that theorized category was a `more coloured'
version of planar multicategory.
Higher theorizations of category relates to planar higher theories in
a similar manner.

For another example, if we fix one object of an $n$-theory to look at,
but allow all endomultimaps of arbitrary arities (and arbitrary higher
multimaps between them), then we find the
structure of an ($n-1$)-th theorization of an uncoloured operad.
The theorizations can have colours at `shallower' levels, and these
will be precisely defined as \emph{($n-1$)-tuply coloured}
$n$-theories in Section \ref{sec:stratum-for-colour}.

\subsection{Restricting strata for colours}
\label{sec:stratum-for-colour}

\subsubsection{}

The datum of a multicategory, or a coloured operad, enriched in
groupoids, say, consists of
collection of objects or ``colours'', and operations or ``multimaps''
which compose.
In an $n$-theory, only the multimaps of dimension $n$ are required to
compose, so only these are really considered as operations, while the
collections of lower dimensional multimaps are then considered as
forming \emph{strata of colours} whose role is to specify the types of
the operations in the top dimension.

As we could consider uncoloured operad, we sometimes want to consider
uncoloured, and only partially coloured, higher theories.
Specifically, we would like to consider situations in which the data
below some dimension are all fixed to be `trivial'.
We actually do this by simulating such a situation, instead of
defining what we mean by the ``trivial'' data.

\begin{definition}\label{def:uncoloured}
Let $n\ge 2$ be an integer.
Then a (symmetric) \kore{uncoloured $n$-theory} $\cat{U}$ enriched in
a symmetric monoidal category $\cat{A}$, consists of
data of the forms specified below as ($n$), ($n+1$) and ($\infty$).
\end{definition}

\begin{itemize}
\item[($n$)](The object ``formed by \emph{objects}''.)
  Suppose given the arity $(I;\pi,\phi)$ of an $n$-multimap in a
  symmetric higher theory, namely, ($n-1'$) through ($0'$) of ($n$) in
  Section \ref{sec:n}.
  Then an object $\Ob^\pi\cat{U}$ of $\cat{A}$, to be called the
  object of \kore{objects} of $\cat{U}$ of the specified arity.
\end{itemize}

This extends for a non-elemental input datum just as the datum ($n$)
of Section \ref{sec:n} did, and we use the resulting extended datum in
the specification of the next datum form.

\begin{itemize}
\item[($n+1$)](\emph{Composition} of objects.)
  Suppose given the arity $(I;\pi,\phi)$ of an ($n+1$)-multimap in a
  symmetric higher theory.
  Then a map
  $m^\cat{U}_1(\pi)\colon\Ob^{\phi^{n-1}}\cat{U}\to\Ob^{\pi_!\phi^{n-1}}\cat{U}$
  in $\cat{A}$, where the source here is the object $\Tensor_{i\in
    I^n_0}\Ob^{\phi^{n-1}_i}\cat{U}$.
  The map $m_1(\pi)$ will be called the \kore{composition}
  operation for objects of $\cat{U}$.
\item[($\infty$)] A datum of coherent associativity for the
  composition
  operations, corresponding to that for an $n$-theory described in
  Sections \ref{sec:associative-n+2} and \ref{sec:coherence-n+ell}.
\end{itemize}

This completes Definition \ref{def:uncoloured}.

\begin{definition}\label{def:tuply-coloured}
Let $n\ge 2$ be an integer.
We say that an $n$-theory as defined in Definition \ref{def:theory},
as \kore{$n$-tuply coloured}.

Let $m$ be an integer such that $1\le m\le n-1$.
Then a (symmetric) \kore{$m$-tuply coloured $n$-theory} $\cat{U}$
enriched in a symmetric monoidal category $\cat{A}$,
consists of data of the forms specified below as ($n-m$), ``($k$)''
for every integer $k$ such that $n-m+1\le k\le n-1$, ($n$),
($n+1$) and ($\infty$).
\end{definition}

\begin{itemize}
\item[($n-m$)](\emph{Object}.)
  Suppose given the arity $(I;\pi,\phi)$ of an ($n-m$)-multimap in a
  symmetric higher theory (or a finite set $S$ with unique map
  $\pi\colon S\to\ten$, if $n-m=1$).
  Then a collection $\Ob^\pi\cat{U}$, whose member will be called an
  \kore{object} of $\cat{U}$ of the specified arity.
\end{itemize}

This extends for a non-elemental input datum just as the datum ($k$)
of
Section \ref{sec:k} did for $k=n-m$, and we use the resulting extended
datum in the specification of the next datum form.

\begin{itemize}
\item[($k$)](\emph{($k-n+m$)-multimap}, inductively for $n-m+1\le k\le
  n-1$.)
  Suppose given
  \begin{itemize}
  \item the arity $(I;\pi,\phi)$ of a $k$-multimap in a symmetric
    higher theory,
  \item if $k-3\ge n-m$, then ($n-m''$) through ($k-3''$) of ($k-1$)
    here,
  \end{itemize}
  and
  \begin{itemize}
  \item[($k-2''$)] if $k-2\ge n-m$, then an $I^{k-2}$-family
    $u^{k-2-n+m}=(u^{k-n+m-2}_i)_{i\in[I^{k-1}_0]}$ of
    ($k-n+m-2$)-multimaps (or objects if $k-n+m-2=0$) in $\cat{U}$,
    where, if $k-n+m-2\ge 1$, then the $I^{k-2}_i$-family
    $u^{k-n+m-2}_i$ is in fact a
    ${\maekara{\phi^{k-2}}{i}}_!\phi^{k-3}$-nerve (see
    ($k-1''$) below) of ($k-n+m-2$)-multimaps in $\cat{U}$,
    connecting ${\maekara{\phi^{k-2}}{i}}_!u^{k-n-m-3}$, 
  \item[($k-1''$)] if $k-1=n-m$, then
    \begin{itemize}
    \item a \emph{$\phi^{k-2}$-nerve}
      $u^0_0$ of objects in $\cat{U}$,
      which \emph{by definition} means that $u^0_0=(u^0_{0i})_{i\in
        I^{k-1}_0}$, where $u^0_{0i}\in\Ob^{\phi^{k-2}_i}\cat{U}$,
    \item an object $u^0_1\in\Ob^{\pi_!\phi^{k-2}}\cat{U}$;
    \end{itemize}
    if $k-1\ge n-m+1$, then
    \begin{itemize}
    \item a \emph{$\phi^{k-2}$-nerve} $u^{k-1-n+m}_0$ of
      ($k-n+m-1$)-multimaps in $\cat{U}$ \emph{connecting}
      $u^{k-n+m-2}$, which \emph{by definition} means that
      $u^{k-n+m-1}_0=(u^{k-n+m-1}_{0i})_{i\in I^{k-1}_0}$, where
      $u^{k-n+m-1}_{0i}\in\Mul^{\phi^{k-2}_i}[u^{\le
        k-n+m-2}\resto{i}]$ (see below),
    \item $u^{k-n+m-1}_1\in\Mul^{\pi_!\phi^{k-2}}[\pi_!u^{\le
        k-n+m-2}]$.
    \end{itemize}
  \end{itemize}
  Then a collection
  $\Mul^\pi_\cat{U}[u^{\le k-n+m-2}](u^{k-n+m-1}_0;u^{k-n+m-1}_1)$ or
  $\Mul^\pi_\cat{U}[u]$ for short, whose member will be called a
  \kore{($k-n+m$)-multimap} $u^{k-n+m-1}_0\to u^{k-n+m-1}_1$ in
  $\cat{U}$.
\end{itemize}

This extends for a non-elemental input datum just as the datum ($k$)
of Section \ref{sec:k} did, and we use the resulting extended datum in
the specification of the next datum form.

\begin{itemize}
\item[($n$)](The object ``formed by \emph{$m$-multimaps}''.)
  Suppose given the arity $(I;\pi,\phi)$ of an $n$-multimap in a
  symmetric higher theory, and a set of data similar to ($0''$)
  through ($k-1''$) of ``($k$)'' above, but with $k$ substituted by
  $n$, so these will be ($0''$) through ($n-1''$) here.
  Then an object $\Mul^\pi_\cat{U}[u^{\le m-2}](u^{m-1}_0;u^{m-1}_1)$
  of $\cat{A}$, or $\Mul^\pi_\cat{U}[u]$ for short, to be called the
  object of \kore{$m$-multimaps} $u^{m-1}_0\to u^{m-1}_1$ in
  $\cat{U}$.
\end{itemize}

This extends for a non-elemental input datum just as the datum ($n$)
of Section \ref{sec:n} did, and we use the resulting extended datum in
the specification of the next datum form.

\begin{itemize}
\item[($n+1$)](\emph{Composition} of $m$-multimaps.)
  Suppose given
  \begin{itemize}
  \item the arity $(I;\pi,\phi)$ of an ($n+1$)-multimap in a symmetric
    higher theory,
  \item ($n-m''$) through ($n-2''$) of ($n$) above,
  \item ($k-2''$) of ($k$) above, but with $k$ substituted by
    $n+1$, so this will be ($n-1''$) here.
  \end{itemize}
  Then a map
  $m^\cat{U}_1(\pi)\colon\Mul^{\phi^{n-1}}_\cat{U}[u]\to\Mul^{\pi_!\phi^{n-1}}_\cat{U}[\pi_!u]$
  in $\cat{A}$, where the source here is the object $\Tensor_{i\in
    I^n_0}\Mul^{\phi^{n-1}_i}[u\resto{i}]$ (of $\phi^{n-1}$-nerves of
  $m$-multimaps connecting $u^{m-1}$ in $\cat{U}$).
  The map $m_1(\pi)$ will be called the \kore{composition} operation
  for $m$-multimaps in $\cat{U}$.

\item[($\infty$)] A datum of (coherent) associativity for the
  composition
  operations, corresponding to that for an $n$-theory described
  in Sections \ref{sec:associative-n+2} and \ref{sec:coherence-n+ell}.
\end{itemize}

This completes Definition \ref{def:tuply-coloured}.

\subsubsection{}

Any notion which makes sense for a general $n$-theory also makes sense
for an $n$-theory with restricted strata of colours as above.
Indeed, given the definition of a notion concerning an $n$-theory, one
obtains the definition of the corresponding notion for a less coloured
$n$-theory just by suppressing from the definition, every
specification involving colours in the lower dimensions in the
$n$-theory.
For example, at a place where one needs to choose an object of an
$n$-theory, one just does not need to make any choice with an
$m$-tuply coloured $n$-theory if $m\le n-1$.
Similarly, at a place where one needs to make some choice for every
object of the $n$-theory, one just make one choice with a not fully
coloured $n$-theory.

\subsection{Forgetting categorifications to theorizations}
\label{sec:forget-to-theorization}

\subsubsection{}

For an integer $m\ge 0$, let us denote by $\Cat_m$ the Cartesian
symmetric monoidal category of $m$-categories with a fix limit for the
size, where we let a \kore{$0$-category} mean a groupoid by
convention.

While an $n$-theory enriched in $\Cat_m$ is an instance of an enriched
$n$-theory, it is also \emph{unenriched} in the sense that we can more
generally consider $m$-categories enriched in a symmetric monoidal
category.
Thus it can be considered both as an ``unenriched'' instance of an
\emph{$m$-categorified} $n$-theory, and as an enriched
``uncategorified'' $n$-theory.
Interpolating these two views, it can also be considered as an
$\ell$-categorified $n$-theory enriched in ($m-\ell$)-categories, for
every integer $\ell$ such that $1\le\ell\le m-1$.

\subsubsection{}

Since ($n+1$)-theory is a theorization of $n$-theory, there are in
particular, $n$-theories enriched in $\Cat_{m+1}$ among
($n+1$)-theories enriched in $\Cat_m$.
Indeed, a \emph{op-lax} $n$-theory enriched in $\Cat_{m+1}$ can be
characterized among ($n+1$)-theories enriched in $\Cat_m$, as one in
which every
profunctor\slash distributor\slash bimodule (enriched in $\Cat_m$) virtually
giving the composition of
\emph{$n$-multimaps}, is corepresentable, and an $n$-theory can be
characterized among op-lax $n$-theories as
one in which every associativity map is an isomorphism.

Given an $n$-theory $\cat{U}$ enriched in ($m+1$)-categories, let us
denote by $\wasreteori_n\cat{U}$, ($n+1$)-theory enriched in
$m$-categories obtained by replacing each functor of composition
operation for $n$-multimaps in $\cat{U}$, by the bimodule
corepresented by it.
We shall say that $\wasreteori_n\cat{U}$ is \kore{represented} by
$\cat{U}$.

\begin{definition}\label{def:forget-to-theorization}
Let $n\ge 0$, $m\ge 2$ be integers, and let $\cat{U}$ be an $n$-theory
enriched in $m$-categories.
Given an integer $\ell$ such that $0\le\ell\le m$, we define an
($n+\ell$)-theory $\wasreteori^{n+\ell}_n\cat{U}$ enriched in
($m-\ell$)-categories, by the inductive relations
\[
\wasreteori^{n+\ell}_n\cat{U}=
\begin{cases}
\cat{U}&\text{if $\ell=0$},\\
\wasreteori_{n+\ell-1}\wasreteori^{n+\ell-1}_n\cat{U}&\text{if $\ell\ge 1$}.
\end{cases}
\]
\end{definition}

For example, we obtain from a symmetric monoidal $n$-category
$\cat{A}$ an uncategorified $n$-theory $\wasreteori^n_0\cat{A}$.
Even though this is an $n$-theory, it is true that this is not really
a new mathematical object since it is essentially just a symmetric
monoidal $n$-category $\cat{A}$.
An $n$-theory which \emph{fails} to be represented by an
($n-1$)-theory arises, for example, through the ``delooping'', as well
as the ``convolution'', constructions, which we shall discuss in
Sections \ref{sec:delooping} and \ref{sec:convolution} respectively.

However, considering symmetric monoidal $n$-categories as $n$-theories
means considering very different \emph{morphisms} between symmetric
monoidal $n$-categories, since a functor of these $n$-theories turns
out to be an \emph{$n$-lax} version of a symmetric monoidal functor of
the original symmetric monoidal $n$-categories.
Here and everywhere, by ``\emph{$n$-lax}'' we mean ``relaxed $n$
times''.
The way in which the structure is relaxed each time is actually quite
interesting here, and the author also finds the structure resulting
from iteration of these processes of relaxation fascinating.

Even though $n$-lax symmetric monoidal functor may still not be a
\emph{very} new notion, the construction above is certainly giving a
new meaning to this notion, in a richer environment where many new
and natural mathematical structures interact, as we shall show through
this work.

In general, for $m$-categorified $n$-theories $\cat{U},\cat{V}$, a
functor $\wasreteori^{m+n}_n\cat{U}\to\wasreteori^{m+n}_n\cat{V}$ of
uncategorified ($n+m$)-theories is equivalent as a datum to an $m$-lax
functor $\cat{U}\to\cat{V}$, as follows similarly to Therem
\ref{thm:functor-of-theorization} below.

\subsubsection{}

There are of course, less coloured versions of all the above.
In particular, in the situation of Definition
\ref{def:forget-to-theorization}, if $\cat{U}$ is $k$-tuply coloured,
where $0\le k\le n-1$, then $\wasreteori^{n+\ell}_n\cat{U}$ is
obtained as ($k+\ell$)-tuply coloured.

\subsection{Delooping a higher theory}
\label{sec:delooping}

\subsubsection{}

There is a construction, which we shall call the \emph{delooping}, of
a symmetric $n$-tuply coloured ($n+1$)-theory from a symmetric
$n$-theory.
We shall describe this construction, and then discuss its relation to
the ``categorical'' delooping.

An ($n+1$)-theory which is obtained through this construction normally
fails to be representable by an $n$-theory.
Delooped theories will be conveniently used throughout this work as
the targets of (possibly ``coloured'') functors of higher theories.

\subsubsection{}

The delooping construction relies on the following construction.

Suppose first that the arity of a $2$-multimap is specified as in
($1'$) and ($0'$) of ($2$) in Section \ref{sec:2}.
Then since our notation \ref{notation:skip-0} chooses an embedding
$I^1_0\into[I^1_0]$, it makes sense to take the coproduct
$\coprod_{I^1_0}I^0$.
Then note that a $\phi^0$-nerve $u^0_0$ as in ($2$) of Definition
\ref{def:tuply-coloured} in the case $m-n=1$, of objects of an
($n-1$)-tuply coloured $n$-theory, can be considered as a
($\coprod_{I^1_0}I^0$)-family, while an ($I^0_0$-ary) object $u^0_1$
is a ($\coprod_{I^1_1}\pi_!I^0$)-family, where $I^1_1:=\{1\}$.

Suppose next given the arity $(I;\pi,\phi)$ of a $3$-multimap in a
symmetric higher theory.
Then for every $i\in I^2_0$ and $j\in I^1_i$, we have a map
$\coprod_{(\phi^1_i)^{-1}j}{(\maekara{\phi^1}{i-1})}_!I^0\to
I^0_{[\maekara{\phi^1}{i}](j)}$ whose component for
$k\in(\phi^1_i)^{-1}j$ is the composite
\[
\phi^0_{[\maekara{\phi^1}{i}](j)}\cdots\phi^0_{[\maekara{\phi^1}{i-1}](k)+1}\colon
I^0_{[\maekara{\phi^1}{i-1}](k)}\longto
I^0_{[\maekara{\phi^1}{i}](j)}.
\]
Taking the coproduct of these over $j$, we obtain a map
$\coprod_{I^1_{i-1}}{(\maekara{\phi^1}{i-1})}_!I^0\to\coprod_{I^1_i}{(\maekara{\phi^1}{i})}_!I^0$,
which together form an $I^2_0$-nerve in $\Fin$.

Let us denote this nerve by $\int_{I^1}\phi^0$, and the
$[I^2_0]$-family of finite sets connected by it by $\int_{I^1}I^0$.
Then the datum $u^0_i$ of ($1''$) of ($3$) for Definition
\ref{def:tuply-coloured} in the case $m=n-1$, can be considered as a
$\bigl(\int_{I^1}I^0\bigr)_i$-family, so $u^0$ can be considered as a
($\int_{I^1}I^0$)-family.

\subsubsection{}

Suppose now given an $n$-theory $\cat{V}$.
Then we wish to construct a new $n$-tuply coloured ($n+1$)-theory
$\cat{U}=\ddeloop\cat{V}$ by precomposing the above constructions
to the data for $\cat{V}$.

The construction of $\cat{U}$ is as follows.
\begin{itemize}
\item[($1$)] Given a finite set $S$, we let
  $\Ob^\pi\cat{U}=\Ob\cat{V}$, where $\pi$ denotes the unique
  map $S\to\ten$.
\item[($2$)] Given the arity of a $2$-multimap as in the preliminary
  construction above, if $n=0$, then we let
  $m^\cat{U}_1(\pi)\colon\Ob^{\phi^0}\cat{U}\to\Ob^{\pi_!\phi^0}\cat{U}$
  be the multiplication map
  $(\Ob\cat{V})^{\tensor(\coprod_{I^1_0}I^0)}\to\Ob\cat{V}$.
  If $n\ge 1$, and further given a $\phi^0$-nerve $u^0_0$ of objects
  of $\cat{U}$ and an $I^0_0$-ary object $u^0_1$ as in ($2$) of
  Definition \ref{def:tuply-coloured} in the case $m=n-1$, then we let
  $\Mul^\pi_\cat{U}[u^0]=\Mul^{\pi'}_\cat{V}[u^0]$, where $\pi'$
  denotes the unique map $\coprod_{I^1_0}I^0\to\ten$, and $u^0$
  is considered as a ($\int_{I^1}I^0$)-family of objects in $\cat{V}$,
  by the preliminary construction above.
\end{itemize}

Next, if $n\ge 2$, suppose given an input datum for ($3$) for
Definition \ref{def:tuply-coloured} in the case $m=n-1$.
Then we can construct a set of data of the form required of an input
to ($2$) in Section \ref{sec:2} for $\cat{V}$, as follows.
If we define
\[
J^1_0:=I^2_0,\quad
J^0:=\int_{I^1}I^0,\quad\psi^0:=\int_{I^1}\phi^0,\quad\psi^1:=(\pi\colon
J^1_0\to\{1\}),
\]
then $J,\psi,u$ give the required form of datum.
Using this, we let $\Mul^\pi_\cat{U}[u]=\Mul^\pi_\cat{V}[u]$.

If $n\ge 1$, we let $m^\cat{U}_\nu(\pi)=m^\cat{V}_\nu(\pi)$ for
$\nu=1$ or $2$, in the similar manner, where the form of $u$ is
different in the
case $n=1$, and we are not given $u$ in the case $n=0$.

For every $k\ge 4$, the datum ($k$) for $\cat{U}$ is constructed in
the similar manner from the datum ($k-1$) for $\cat{V}$.

\subsubsection{}

It is clear that if $\cat{U}$ is an $m$-tuply coloured $n$-theory,
then its deloop $\ddeloop\cat{U}$ is obtained as an $m$-tuply coloured
($n+1$)-theory.

\subsubsection{}
\label{sec:delooping-comparison}

Our delooping construction for higher theories relates to the
categorical delooping to be recalled now.
The \emph{categorical delooping} construction
associates to a monoid $A$ (or
more generally, a monoidal $n$-category), a category (or
($n+1$)-category in the general case) $\deloop A$ with a
chosen ``base'' object, in which
\begin{itemize}
\item all objects are equivalent,
\item the endomorphism monoid (or monoidal $n$-category) of the base
  object is given an equivalence with $A$.
\end{itemize}

Note that, if $\cat{A}$ is a \emph{symmetric monoidal} $n$-category,
then $\deloop\cat{A}$ is canonically a symmetric monoidal
($n+1$)-category (with unit the base object) since the functor
$\deloop$ preserves direct products.

Let us describe one relation which will be convenient for us.

Specifically, let $n\ge 0$ be an integer, and $\cat{A}$ be a symmetric
monoidal $n$-category.
Then for an integer $m$ such that $0\le m\le n$, we would like to
relate the ($m+1$)-theory $\ddeloop\wasreteori^m_0\cat{A}$ to the
($m+1$)-theory $\wasreteori^{m+1}_0\deloop\cat{A}$, both enriched in
($n-m$)-categories.

An obvious issue for this is that $\wasreteori^{m+1}_0\deloop\cat{A}$
is fully coloured whereas $\ddeloop\wasreteori^m_0\cat{A}$ is only
$m$-tuply coloured.
However, if we restrict the data for
$\wasreteori^{m+1}_0\deloop\cat{A}$ so the only object we consider is
the base object of $\deloop\cat{A}$, then the rest of the data for
$\wasreteori^{m+1}_0\deloop\cat{A}$ is of the form of datum for an
$m$-tuply coloured ($m+1$)-theory enriched in ($n-m$)-categories.

\begin{proposition}\label{prop:deloop-comparison}
Let $n\ge 0$ be an integer, and $\cat{A}$ be a symmetric monoidal
$n$-category.
Then for every integer $m$ such that $0\le m\le n$,
$\ddeloop\wasreteori^m_0\cat{A}$ is equivalent to the
($n-m$)-categorified $m$-tuply coloured ($m+1$)-theory obtained as
above from $\wasreteori^{m+1}_0\deloop\cat{A}$ by restricting objects
to just the base object of $\deloop\cat{A}$.
\end{proposition}

\begin{proof}
The case $m=0$ is immediate from the constructions, and the general
case follows by induction, from Lemma
\ref{lem:deloop-forget-commutation} below.
\end{proof}

The lemma is as follows, and follows immediately from the
definitions.

\begin{lemma}\label{lem:deloop-forget-commutation}
For an arbitrary times categorified $n$-theory $\cat{U}$, we have
\[
\ddeloop\wasreteori_n\cat{U}\equivwith\wasreteori_{n+1}\ddeloop\cat{U}.
\]
\end{lemma}

\subsection{Theorization and lax functors}
\label{sec:functor}

\subsubsection{}

The obvious notion of functor of $n$-theories has a reasonable lax
version.
Let us start with recording the definition of a functor.

\begin{definition}\label{def:functor}
Let $n\ge 2$ be an integer, and let $\cat{U}$ and $\cat{V}$ be
$n$-theories enriched in a symmetric monoidal category $\cat{A}$.
Then a \kore{functor} $F\colon\cat{U}\to\cat{V}$ of $n$-theories
consists of data
of the forms specified below as ($0$), ``($k$)'' for every integer $k$
such that $1\le k\le n-1$, ($n$), ($n+1$), $(n+2)$, and ``($n+\ell$)''
for every integer $\ell\ge 3$.
\end{definition}

Similarly to the data for an $n$-theory, data for a functor will be
associated to input data satisfying the same elementality requirements
as before.
As before, each datum form specified needs to be extended for a
non-elemental input datum before the next datum form is specified.
The way how we do this is more or less the same as before, and we
shall do this implicitly.

The forms of data are as follows.

\begin{itemize}
\item[($0$)](\emph{Action} on objects.)
  For every object $u$ of $\cat{U}$, an object $Fu$ of $\cat{V}$.
\item[($k$)](\emph{Action} on $k$-mutimaps, inductively for $1\le k\le
  n-1$.)
  Suppose given the type $u$ of a $k$-multimap in $\cat{U}$ of arity
  given as $(I;\pi,\phi)$.
  Then for every $k$-multimap $x\in\Mul^\pi_\cat{U}[u]$, a
  $k$-multimap $Fx\in\Mul^\pi_\cat{V}[Fu]$.
\item[($n$)](\emph{Action} on $n$-multimaps.)
  Suppose given the type $u$ of an $n$-multimap in $\cat{U}$ of arity
  given as $(I;\pi,\phi)$.
  Then a map
  $\inv{m}^F_0(\pi)\colon\Mul^\pi_\cat{U}[u]\to\Mul^\pi_\cat{V}[Fu]$
  in $\cat{A}$.
\item[($n+1$)](The isomorphism of \emph{compatibility} with the
  composition.)
  Suppose given
  \begin{itemize}
  \item the arity $(I;\pi,\phi)$ of an ($n+1$)-multimap in a symmetric
    higher theory,
  \item the type $u$ of a $\phi^{n-1}$-nerve of $n$-multimaps in
    $\cat{U}$.
  \end{itemize}
  Then a $2$-isomorphism
  \[
  \inv{m}^F_1(\pi)\colon\inv{m}^F_0(\pi_!\phi^{n-1})\compose
  m^\cat{U}_1(\pi)\longequivto
  m^\cat{V}_1(\pi)\compose\inv{m}^F_0(\phi^{n-1})
  \]
  in $\cat{A}$, filling the square
  \[\begin{tikzcd}[column sep=large]
  \Mul^{\phi^{n-1}}_\cat{U}[u]\arrow{r}{m^\cat{U}_1(\pi)}\arrow{d}[swap]{\inv{m}^F_0(\phi^{n-1})}
  &\Mul^{\pi_!\phi^{n-1}}_\cat{U}[\pi_!u]\arrow{d}{\inv{m}^F_0(\pi_!\phi^{n-1})}\\
  \Mul^{\phi^{n-1}}_\cat{V}[Fu]\arrow{r}{m^\cat{V}_1(\pi)}
  &\Mul^{\pi_!\phi^{n-1}}_\cat{V}[\pi_!Fu]\kokoni{,}
  \end{tikzcd}\]
  where $\inv{m}^F_0(\phi^{n-1})$ denotes the monoidal product over
  $i\in I^n_0$ of the maps
    \[
    \inv{m}^F_0(\phi^{n-1}_i)\colon\Mul^{\phi^{n-1}_i}[u\resto{i}]\longto\Mul^{\phi^{n-1}_i}[Fu\resto{i}].
    \]
\item[($n+2$)](The isomorphism of \emph{coherence} for the
  compatibility with the composition.)
  Suppose given
  \begin{itemize}
  \item the arity $(I;\pi,\phi)$ of an ($n+2$)-multimap in a
    symmetric higher theory,
  \item the type $u$ of a $\phi^{n-1}$-nerve of $n$-multimaps in
    $\cat{U}$.
  \end{itemize}
  Then an $3$-isomorphism
  \[
  \inv{m}^F_2(\pi)\colon\inv{m}^F_1(\pi_!\phi^n)\compose
  m^\cat{U}_2(\pi)\longequivto
  m^\cat{V}_2(\pi)\compose\pi_!\inv{m}^F_1(\phi^n)
  \]
  in $\cat{A}$, filling the square
  \[\begin{tikzcd}[column sep=huge]
    \inv{m}^F_0({\pi^n}_!\phi^{n-1})\compose\pi_!m^\cat{U}_1(\phi^n)\arrow{r}{\pi_!\inv{m}^F_1(\phi^n)}\arrow{d}[swap]{m^\cat{U}_2(\pi)}
    &\pi_!m^\cat{V}_1(\phi^n)\compose\inv{m}^F_0(\phi^{n-1})\arrow{d}{m^\cat{V}_2(\pi)}\\
    \inv{m}^F_0({\pi^n}_!\phi^{n-1})\compose
    m^\cat{U}_1(\pi_!\phi^n)\arrow{r}{\inv{m}^F_1(\pi_!\phi^n)}
    &m^\cat{V}_1(\pi_!\phi^n)\compose\inv{m}^F_0(\phi^{n-1})\kokoni{,}
  \end{tikzcd}\]
  where $\inv{m}^F_1(\phi^n)$ denotes the
  $(I^{n+1}_0)^\op_{\vphantom{0}}$-nerve
  $\bigl(\inv{m}^F_1(\phi^n_i)\bigr)_{i\in I^{n+1}_0}$
  of $2$-isomorphisms in $\cat{A}$ connecting
  the family indexed by $i\in[I^{n+1}_0]$ of the maps
  \[
  \sakihe{m^\cat{V}_1\bigl(\phi^n\bigr)}{i}\compose\inv{m}^F_0\bigl({(\maekara{\phi^n}{i})}_!\phi^{n-1}\bigr)\compose\maekara{m^\cat{U}_1\bigl(\phi^n\bigr)}{i}\colon\Mul^{\phi^{n-1}}_\cat{U}[u]\longto\Mul^{\pi_!\phi^{n-1}}_\cat{V}[\pi_!Fu]
  \]
  in $\cat{A}$.
\item[($n+\ell$)](Higher coherence, inductively for $\ell\ge 3$.)
  Suppose given
  \begin{itemize}
  \item the arity $(I;\pi,\phi)$ of an ($n+\ell$)-multimap in a
    symmetric higher theory,
  \item the type $u$ of a $\phi^{n-1}$-nerve of $n$-multimaps in
    $\cat{U}$.
  \end{itemize}
  Then an ($\ell+1$)-isomorphism
  \[
  \inv{m}^F_\ell(\pi)\colon\inv{m}^F_{\ell-1}(\pi_!\phi^{n+\ell-2})\compose
  m^\cat{U}_\ell(\pi)\longequivto
  m^\cat{V}_\ell(\pi)\compose\pi_!\inv{m}^F_{\ell-1}(\phi^{n+\ell-2})
  \]
  in $\cat{A}$, filling the square
  \[\begin{tikzcd}[column sep=7.2 em]
  \begin{split}
  \inv{m}^F_{\ell-2}({\pi^{n+\ell-2}}_!\phi^{n+\ell-3})\compose\qquad\quad\\
  \pi_!m^\cat{U}_{\ell-1}(\phi^{n+\ell-2})
  \end{split}
  \arrow{r}{\pi_!\inv{m}^F_{\ell-1}(\phi^{n+\ell-2})}\arrow{d}[swap]{m^\cat{U}_\ell(\pi)}
  &
  \begin{split}
  \pi_!m^\cat{V}_{\ell-1}(\phi^{n+\ell-2})\compose\qquad\qquad\\
  {\pi^{n+\ell-2}}_!\inv{m}^F_{\ell-2}(\phi^{n+\ell-3})
  \end{split}
  \arrow{d}{m^\cat{V}_\ell(\pi)}\\
  \begin{split}
  \inv{m}^F_{\ell-2}({\pi^{n+\ell-2}}_!\phi^{n+\ell-3})\compose\qquad\quad\\
  m^\cat{U}_{\ell-1}(\pi_!\phi^{n+\ell-2})
  \end{split}
  \arrow{r}{\inv{m}^F_{\ell-1}(\pi_!\phi^{n+\ell-2})}
  &
  \begin{split}
  m^\cat{V}_{\ell-1}(\pi_!\phi^{n+\ell-2})\compose\qquad\qquad\phantom{,}\\
  {\pi^{n+\ell-2}}_!\inv{m}^F_{\ell-2}(\phi^{n+\ell-3})\kokoni{,}
  \end{split}
  \end{tikzcd}\]
  where $\inv{m}^F_{\ell-1}(\phi^{n+\ell-2})$ denotes the
  $(I^{n+\ell-1}_0)^\op_{\vphantom{0}}$-nerve
  $\bigl(\inv{m}^F_{\ell-2}(\phi^{n+\ell-2}_i)\bigr)_{i\in
    I^{n+\ell-1}_0}$ of $\ell$-isomorphisms in $\cat{A}$ connecting
  the family indexed by $i
  \in[I^{n+\ell-1}_0]$ of the ($\ell-1$)-isomorphisms
  \begin{multline*}
  \sakihe{m^\cat{V}_{\ell-1}(\phi^{n+\ell-2})}{i}\compose\\
  {\sakihe{\phi^{n+\ell-2}}{i}}_!\inv{m}^F_{\ell-2}({\maekara{\phi^{n+\ell-2}}{i}}_!\phi^{n+\ell-3})\compose\maekara{m^\cat{U}_{\ell-1}(\phi^{n+\ell-2})}{i}\colon\qquad\qquad\\
  \qquad\qquad \inv{m}^F_{\ell-3}({\pi^{n+\ell-3}}_!\phi^{n+\ell-4})\compose{\pi^{n+\ell-2}}_!m^\cat{U}_{\ell-2}(\phi^{n+\ell-3})\\
  \longequivto
  m^\cat{V}_{\ell-2}(\pi^{n+\ell-3})\compose{\pi^{n+\ell-3}}_!\inv{m}^F_{\ell-3}(\phi^{n+\ell-4})
  \end{multline*}
  in $\cat{A}$.
\end{itemize}

This completes Definition \ref{def:functor}.

\subsubsection{}

Now a lax functor can be defined as follows.

\begin{definition}\label{def:lax-functor}
Let $n\ge 2$ be an integer, and let $\cat{U}$ and $\cat{V}$ be
$n$-theories enriched in a symmetric monoidal $2$-category $\cat{A}$
(i.e., in the symmetric monoidal category underlying $\cat{A}$).
Then a \kore{lax functor} $F\colon\cat{U}\to\cat{V}$ consists of data
of the forms
\begin{itemize}
\item ($0$) through ($n$) specified above for Definition
  \ref{def:functor} for the same value of ``$n$'',
\item ($n+1$) and ($\infty$) below.
\end{itemize}
\end{definition}

\begin{itemize}
\item[($n+1$)](The map of \emph{compatibility} with the composition.)
  Suppose given
  \begin{itemize}
  \item the arity $(I;\pi,\phi)$ of an ($n+1$)-multimap in a symmetric
    higher theory,
  \item the type $u$ of a $\phi^{n-1}$-nerve of $n$-multimaps in
    $\cat{U}$.
  \end{itemize}
  Then a $2$-map
  \[
  m^F_1(\pi)\colon
  m^\cat{V}_1(\pi)\compose\inv{m}^F_0(\phi^{n-1})\longto\inv{m}^F_0(\pi_!\phi^{n-1})\compose
  m^\cat{U}_1(\pi)
  \]
  in $\cat{A}$.
  See ($n+1$) for Definition \ref{def:functor} above.
\item[($\infty$)] A datum of coherence, similar to that for a
  functor.
\end{itemize}

This completes Definition \ref{def:lax-functor}.

\subsubsection{}
\label{sec:colour-system}

The notion of lax functor can be used to define the notion of
``\emph{($n+1$)-tuply}'' \emph{coloured lax $n$-theory}, which will
allow us to describe an ($n+1$)-theory along the line discussed in
Section \ref{sec:coloured-lax-structure}, as to be done in Proposition
\ref{prop:theory-as-theorization}.
In order to do this, let us first understand an $n$-theory as a
functor of higher theories.

In order to describe the source of the functor, we need the following
definitions.
Recall that an $n$-theory $\cat{U}$ consists by definition, of data
($k$) for all integers $k\ge 0$, of the forms specified in Section
\ref{sec:symmetric-theory}.
Of these, the part $k\le n-1$ does not involve the information of
where the theory is enriched.

\begin{definition}
Let $n\ge 0$ be an integer.
Then we refer to the system consisting of data of the forms ($0$)
through ($n-1$) as specified
for Definition \ref{def:theory} of an $n$-theory, as a
system of \kore{colours up to dimension $n-1$} for a higher theory, or
a system of colours \kore{for an $n$-theory}.
\end{definition}

In particular, for every $n$-theory $\cat{U}$ and every integer $m$
such that $0\le m\le n-1$, we have a system of colours up to dimension
$m$ \kore{underlying} $\cat{U}$, consisting of the data ($0$) through
($m$) for $\cat{U}$, so $\cat{U}$ consists of this system of
colours and a structure on it.

\begin{definition}
Let $m\ge 0$ and $n\ge m+1$ be integers.
Then for the system of colours up to dimension $m$ consisting of data
of the forms ($0$) through ($m$) as specified for Definition
\ref{def:theory} of an $n$-theory, we refer to the rest of data for an
$n$-theory, consisting of data of the forms ($k$) for $k\ge m+1$, as
the \kore{structure of an $n$-theory} on the system of colours.
\end{definition}

Now the following gives an interpretation of a higher theory as a
functor of higher theories.

\begin{example}\label{ex:theory-as-functor}
Choose and fix a system of colours up to dimension $n-1$, and denote
by $\cat{T}$ the terminal object among unenriched $n$-theories
extending this system of lower colours.
Explicitly, $\cat{T}$ is such that every groupoid of $n$-multimaps in
it is contractible.

Then the structure on the chosen system of colours, of an $n$-theory
enriched in a symmetric monoidal category $\cat{A}$, is equivalent as
a datum to a lax functor $\cat{T}\to\ddeloop^n\cat{A}$ of categorified
$n$-theories.
This can also be described as a functor
$\wasreteori_n\cat{T}\to\wasreteori_n\ddeloop^n\cat{A}$ of
($n+1$)-theories.
\end{example}

The equivalence of these two descriptions will be generalized by
Theorem \ref{thm:functor-of-theorization} below.
Note that $\wasreteori_n\cat{T}$ is terminal among unenriched
($n+1$)-theories extending our system of lower colours.

\begin{remark}\label{rem:deloop-forget-commutation}
Lemma \ref{lem:deloop-forget-commutation} implies an equivalence
$\wasreteori_n\ddeloop^n\cat{A}\equivwith\ddeloop^n\wasreteori_0\cat{A}$.
\end{remark}

Let us say that an $n$-theory $\cat{T}$ as in Example
\ref{ex:theory-as-functor} is terminal \kore{on}
the chosen system of colours up to dimension $n-1$.
Considering $\cat{T}$ as a coloured variant of the terminal unenriched
uncoloured $n$-theory $\unity^n_\Com$ (where $\Com=E_\infty$ denotes
the commutative operad), one might say that an $n$-theory enriched in
$\cat{A}$ is a lax functor $\unity^n_\Com\to\ddeloop^n\cat{A}$
\emph{with strata of colours} for an $n$-theory, and similarly, also
a functor $\unity^{n+1}_\Com\to\ddeloop^n\wasreteori_0\cat{A}$ with
similar strata of colours.

\begin{definition}\label{def:theory-enriched-in-multicategory}
Let $n\ge 0$ be an integer.
Then, an \kore{$n$-theory enriched in a multicategory $\cat{M}$}, is a
functor $\unity^{n+1}_\Com\to\ddeloop^n\cat{M}$ with strata of colours
for an $n$-theory.
\end{definition}

Concretely,
\begin{itemize}
\item the colours of the functor will be the colours of the
  $n$-theory, say $\cat{U}$, enriched in $\cat{M}$, defined by the
  coloured functor,
\item $n$-multimaps in such $\cat{U}$ are similar to those in the
  case enriched in a symmetric monoidal
  category, but they `form' objects of $\cat{M}$,
\item the composition operations for $n$-multimaps in $\cat{U}$ are
  multimaps in $\cat{M}$.
\end{itemize}

Accordingly, the notion of lax functor generalizes in an obvious
manner, to that between $n$-theories enriched in a multicategory
$\cat{M}$ enriched in categories.
See Definition \ref{def:lax-functor}.

\begin{definition}\label{def:coloured-lax-theory}
Let $n\ge 0$ be an integer, and $\cat{M}$ be a multicategory enriched
in categories.
Then,
\begin{itemize}
\item a \kore{lax $n$-theory} enriched in $\cat{M}$ is a
  lax functor $\unity^{n+1}_\Com\to\ddeloop^n\cat{M}$ with strata of
  colours for an $n$-theory,
\item an \kore{($n+1$)-tuply coloured lax
    $n$-theory} enriched in $\cat{M}$ is a
  lax functor $\unity^{n+1}_\Com\to\ddeloop^n\cat{M}$ with strata
  of colours up to dimension $n$, namely, a
  lax functor $F\colon\cat{T}\to\ddeloop^n\cat{M}$, where
  $\cat{T}$ is an unenriched ($n+1$)-theory, which is terminal on a
  system of colours for an ($n+1$)-theory.
\end{itemize}
\end{definition}

We record Example~\ref{ex:theory-as-functor} using the term just
introduced.

\begin{proposition}\label{prop:theory-as-theorization}
Let $\cat{A}$ be a symmetric monoidal category.
Then, for an integer $n\ge 1$, an $n$-theory enriched in $\cat{A}$ is
equivalent as a datum to an $n$-tuply coloured lax
($n-1$)-theory enriched in $\ddeloop\cat{A}$.
\end{proposition}

\subsubsection{}
\label{sec:functor-of-theorization}

The following is a key fact on the relation between theorization and
categorification.

\begin{theorem}\label{thm:functor-of-theorization}
Let $n\ge 0$ be an integer, and $\cat{U}$ and $\cat{V}$ be
categorified $n$-theories.
Then a functor $\wasreteori_n\cat{U}\to\wasreteori_n\cat{V}$ of
uncategorified ($n+1$)-theories, is
equivalent as a datum to a lax functor $\cat{U}\to\cat{V}$.
\end{theorem}

\begin{proof}
\proofsec
We first note that the form of datum for the action of a functor
$\wasreteori_n\cat{U}\to\wasreteori_n\cat{V}$ on the colours up to
dimension $n-1$, is identical to the form of datum for the action of a
lax functor $\cat{U}\to\cat{V}$ on the colours up to the same
dimension.
Then for any fixed datum $F$ of this form, we would like to prove
that the
categories (or $2$-categories) naturally formed respectively by the
structures on these same data, of functors (of ($n+1$)-theories), and
of lax functors (of categorified $n$-theories), are equivalent.

\proofsec
The structure of a categorified $n$-theory is given on the categories
of $n$-multimaps, by the (cohrently) associative composition
functors.
A lax functor $F\colon\cat{U}\to\cat{V}$ extending the chosen action,
is then seen to be given by
\begin{enumerate}
\setcounter{enumi}{-1}
\item\label{item:action-on-n}
  the action of $F$ on the categories of $n$-multimaps through
  functors specified by the datum $\inv{m}^F_1$, and
\item\label{item:compatibility-n-action-composition}
  the (lax) compatibility specified by the data $m^F_2$ and
  $\inv{m}^F_k$ for $k\ge 3$, of the mentioned action with the
  associative composition functors for $n$-multimaps in $\cat{U}$ and
  in $\cat{V}$.
\end{enumerate}

\proofsec\label{step:lax-functor-to-functor}
Recall that we have obtained $\wasreteori_n\cat{U}$ by replacing the
composition functors of $\cat{U}$ by the
bimodules/distributors/profunctors represented by them, and similarly
for $\cat{V}$.

Given a datum as (\ref{item:action-on-n}) above for a lax functor $F$,
of an action on $n$-multimaps, a datum of compatibility as
(\ref{item:compatibility-n-action-composition}) above, of this action
with the composition functors for $n$-multimaps, can equivalently be
described as an action on the (associative) composition bimodules for
$n$-multimaps (which were represented by the composition functors) in
$\wasreteori_n\cat{U}$ and in $\wasreteori_n\cat{V}$.
Note that these bimodules (in $\wasreteori_n\cat{U}$ or in
$\wasreteori_n\cat{V}$) are formed by ($n+1$)-multimaps, and the
(coherent) associativity of the composition (for $n$-multimaps) can
equivalently be described as the (coherently associative) composition
operations for ($n+1$)-multimaps.

In this manner, we consider $F$ as a structure between
$\wasreteori_n\cat{U}$ and $\wasreteori_n\cat{V}$.

The difference of this structure between $\wasreteori_n\cat{U}$ and
$\wasreteori_n\cat{V}$, from a functor
$G\colon\wasreteori_n\cat{U}\to\wasreteori_n\cat{V}$ of
($n+1$)-theories, is that
\begin{itemize}
\item the action of $G$ on $n$-multimaps does not (explicitly) include
  the datum of functoriality included in $\inv{m}^F_1$, and
\item the action of $G$ on ($n+1$)-multimaps (given by $\inv{m}^G_1$)
  and its compatibility (given by $\inv{m}^G_k$, $k\ge 2$) with the
  composition operations for ($n+1$)-multimaps in
  $\wasreteori_n\cat{U}$ and in $\wasreteori_n\cat{V}$, do not
  explicitly include the data of compatibility with the bimodules
  structures, included in $m^F_2$ and $\inv{m}^F_{\ge 3}$.
\end{itemize}

Never the less, we obtain a functor
$\wasreteori_nF\colon\wasreteori_n\cat{U}\to\wasreteori_n\cat{V}$ by
forgetting these extra data (which will turn out in the next step to
have been actually redundant).

\proofsec\label{step:functor-to-lax-functor}
Conversely, suppose given a functor
$G\colon\wasreteori_n\cat{U}\to\wasreteori_n\cat{V}$ which extends the
chosen actions up to dimension $n-1$.
Then, from the discussions above, a lax functor
$H\colon\cat{U}\to\cat{V}$ extending the same lower
actions, such that $\wasreteori_nH=G$, is constructed if we equip the
higher data for $G$ with the following.
\begin{itemize}
\item A functoriality in the $n$-multimap $x\in\Mul^\pi_\cat{U}[u]$,
  on the $n$-multimap $Gx\in\Mul^\pi_\cat{V}[Hu]$.
\item A compatibility of this functoriality with the data
  $\inv{m}^G_i$, $i\ge 1$, of the action of $G$ (in a manner
  compatible with the (associative) composition operations for
  ($n+1$)-multimaps in $\wasreteori_n\cat{U}$ and in
  $\wasreteori_n\cat{V}$) on ($n+1$)-multimaps.
\end{itemize}

In order to obtain a functoriality of the $n$-multimap $Gx$, one notes
that a map in the category $\Mul^\pi_\cat{U}[u]$
is suitably `unary' ($n+1$)-multimaps in $\wasreteori_n\cat{U}$.
The desired functoriality is given by the restriction of the
associative action of $G$ to unary ($n+1$)-multimaps in
$\wasreteori_n\cat{U}$.

A compatibility of this functoriality with the action of $G$ means
\begin{itemize}
\item a naturality of every instance of the map
  \begin{multline*}
  \inv{m}^G_1(\pi)\colon\Map_{\Mul^{\pi_!\phi^{n-1}}_\cat{U}[\pi_!u^{\le
      n-1}]}\bigl(m^\cat{U}_1(\pi)(u^n_0),u^n_1\bigr)\\
  \longto\Map_{\Mul^{\pi_!\phi^{n-1}}_\cat{V}[\pi_!Hu^{\le
      n-1}]}\bigl(m^\cat{V}_1(\pi)(Gu^n_0),Gu^n_1\bigr)
  \end{multline*}
  in $u^n$, and
\item a compatibility of this naturality with the compatibility, given
  by $\inv{m}^G_i$, $i\ge 2$, of these maps $\inv{m}^G_1(\pi)$ with
  the composition operations in $\wasreteori_n\cat{U}$ and in
  $\wasreteori_n\cat{V}$.
\end{itemize}
We note that the functoriality in $u^n$ of the target of
$\inv{m}^G_1(\pi)$ is the composite of the functoriality of
$\Map_{\Mul^{\pi_!\phi^{n-1}}_\cat{V}[\pi_!Hu^{\le
    n-1}]}\bigl(m^\cat{V}_1(\pi)(-),-\bigr)$ in its variables, with the
functoriality of $G$.
Moreover, the compatibility of this functoriality with the composition
structures of $\wasreteori_n\cat{U}$ and of $\wasreteori_n\cat{V}$,
is also induced by the same functoriality of $G$, from the
compatibility of the functoriality of
$\Map_{\Mul^{\pi_!\phi^{n-1}}_\cat{V}[\pi_!v]}\bigl(m^\cat{V}_1(\pi)(-),-\bigr)$
with the composition operations.

Now, the functoriality, together with its compatibility with the
composition structure of $\wasreteori_n\cat{U}$, of
$\Map_{\Mul^{\pi_!\phi^{n-1}}_\cat{U}[\pi_!u^{\le
    n-1}]}\bigl(m^\cat{U}_1(\pi)(-),-\bigr)=\Mul^\pi_{\wasreteori_n\cat{U}}\bigl[u^{\le
  n-1}\bigr]\bigl(-;-\bigr)$ can alternatively be considered as obtained
by suitably restricting the composition structure of
$\wasreteori_n\cat{U}$, and similarly for
$\Map_{\Mul^{\pi_!\phi^{n-1}}_\cat{V}[\pi_!v]}\bigl(m^\cat{V}_1(\pi)(-),-\bigr)$.
Since the functoriality of $G$ on the category of $n$-multimaps is as
constructed above, we see that the compatibility of the action of $G$
on ($n+1$)-multimaps with the composition structures of
$\wasreteori_n\cat{U}$ and of
$\wasreteori_n\cat{V}$, also induces a desired datum.

\proofsec
We have thus constructed a desired lax functor
$H\colon\cat{U}\to\cat{V}$ naturally from $G$.
In order to complete the proof, we would like to verify in the case
$G=\wasreteori_nF$ for a lax functor $F\colon\cat{U}\to\cat{V}$, that
the lax functor ``$H$'' constructed from $G$ will be naturally
equivalent to $F$.
For this, it suffices to observe that the set of data which we have
added to the datum of $G$ to construct $H$ in the step
\ref{step:functor-to-lax-functor}, is naturally equivalent to the
set of data which was ``forgotten'' in the construction
$F\mapsto\wasreteori_nF$ in the step
\ref{step:lax-functor-to-functor}, but this is clear.
\end{proof}

\begin{remark}\label{rem:lax-functor-general}
We have thus described for the ($n+1$)-theory
$\cat{W}=\wasreteori_n\cat{V}$, a functor
$G\colon\wasreteori_n\cat{U}\to\cat{W}$ as a lax functor
$\cat{U}\to\cat{V}$.
Similar arguments lead to a description of $G$ for an arbitrary
($n+1$)-theory $\cat{W}$, as a structure generalizing a lax functor
$\cat{U}\to\cat{V}$.
One might call this structure between $\cat{U}$ and $\cat{W}$ a
\kore{functor} $\cat{U}\to\cat{W}$.
Thus, a functor $\wasreteori_n\cat{U}\to\cat{W}$ will be equivalent as
a datum to a functor $\cat{U}\to\cat{W}$.

For example, an $n$-theory enriched in a multicategory $\cat{M}$
(Definition~\ref{def:theory-enriched-in-multicategory}) is
equivalent as a datum to a functor $\unity^n_\Com\to\ddeloop^n\cat{M}$
with strata of colours up to dimension $n-1$ for a higher theory.
\end{remark}

Note that $n$-theories enriched in groupoids are also among
$n$-theories enriched in categories since a groupoid can be considered
as a category in which every morphism is invertible.
This kind of unenriched categorified $n$-theories is special as a
target of a
functor in that there is no difference between an ordinary functor and
a lax functor from an unenriched categorified $n$-theory to such a
target.

\begin{corollary}\label{cor:forget-to-theorization}
The functor $\wasreteori_n$ is fully faithful on $n$-theories enriched
in groupoids.
\end{corollary}

\section{Graded higher theories}
\label{sec:grading}

\subsection{Introduction}

The purpose of this section is to discuss grading of higher theory by
a higher theory (enriched in groupoids).
It turns out that this notion naturally arises by considering
theorization of algebra over a higher theory.
We shall also see interrelations between the notions, of ``graded''
lower theory over a higher theory, and of \emph{iterated monoid},
i.e., monoid (i.e., algebra in the 
category of groupoids) over a monoid over ...~over a monoid, over a
higher theory enriched in groupoids.
Some results obtained along the way will be basic.

\subsection{Algebras over a higher theory}
\label{sec:monoid-theorization}

\subsubsection{}

We have sought for the notion of $n$-theory as an interesting $n$-th
theorization of the notion of commutative algebra.
As have been discussed in Section \ref{sec:introduction}, what we
wished to get from this was for each $n$-theory to govern
\emph{algebras} over it, which naturally generalize ($n-1$)-theories.

\begin{definition}\label{def:algebra}
Let $n\ge 2$ be an integer, and let $\cat{U}$ be an unenriched
$n$-theory.
Then a \kore{$\cat{U}$-algebra} $\cat{X}$ enriched in a multicategory
$\cat{M}$ consists of data of the forms specified below as ($0$) and
($1$) (or just ($0$) if $n=2$), ``($k$)'' for every integer $k$ such
that $2\le k\le n-2$, ($n-1$), ($n$) and ($\infty$).
\end{definition}

Let us specify the forms of data for Definition.

\begin{remark}\label{rem:extend-to-non-elemental}
As before, we implicitly extend each datum form specified below for a
non-elemental input datum before proceeding to specifying the next datum
form, in a more or less similar manner as before.
\end{remark}

\begin{itemize}
\item[($0$)](\emph{Object}.)
  For every object $u$ of $\cat{U}$, a collection $\Ob_u\cat{X}$,
  whose member
  will be called an \kore{object} of $\cat{X}$ of \kore{degree} $u$.
\item[($1$)](\emph{Multimap}, in the case $n\ge 3$.)
  Suppose given ($0'$) and ($0''$) of ($1$) in Section \ref{sec:1},
  and
  \begin{itemize}
  \item[($0^\circ$)] an $S$-family $x_0$ of objects of $\cat{X}$ of
    degree $u_0$ (namely, $x_{0s}\in\Ob_{u_{0s}}\cat{X}$ for every
    $s\in S$), and an object $x_1$ of degree $u_1$.
  \end{itemize}
  Then for every multimap $v\in\Mul^\pi_\cat{U}[u]$, a collection
  $\Mul^\pi_{\cat{X},v}(x_0;x_1)$ or $\Mul^\pi_{\cat{X},v}[x]$
  for short, whose member will be called an ($S$-ary)
  (\kore{$1$-})\kore{multimap} $x_0\to x_1$ in $\cat{X}$ of
  \kore{degree} $v$.
\item[($k$)](\emph{$k$-multimap}, inductively for $2\le k\le n-2$.)
  Suppose given (\ref{item:monoid-k-1}), ($0^\circ$),
  (\ref{item:monoid-k-5}), ($k-2^\circ$) (or just
  (\ref{item:monoid-k-1}), ``($0^\circ$)'' if $k=2$), and
  ($k-1^\circ$) below:
  \begin{enumerate}
  \renewcommand{\theenumi}{\alph{enumi}}
  \item\label{item:monoid-k-1}
    the type $u$ of a $k$-multimap in $\cat{U}$ of arity given as
    $(I;\pi,\phi)$.
  \item[($0^\circ$)] an $I^0$-family $x^0$ of objects of $\cat{X}$, of
    degree $u^0$, namely, $x^0=(x^0_i)_{i\in[I^1_0]}$, where
    $x^0_i$ is an $I^0_i$-family of objects of $\cat{X}$, of degree
    $u^0_i$,
  \item\label{item:monoid-k-5} if $k\ge 4$, then ($1^\circ$) through
    ($k-3^\circ$) of ($k-1$) here,
  \item[($k-2^\circ$)](in the case $k\ge 3$)
    an $I^{k-2}$-family $x^{k-2}=(x^{k-2}_i)_{i\in[I^{k-1}_0]}$ of
    ($k-2$)-multimaps in $\cat{X}$, where $x^{k-2}_i$ is an
    ${\maekara{\phi^{k-2}}{i}}_!\phi^{k-3}$-nerve of ($k-2$)-multimaps
    in $\cat{X}$, connecting ${\maekara{\phi^{k-2}}{i}}_!x^{k-3}$ of
    degree $u^{k-2}_i$,
  \item[($k-1^\circ$)]
    \begin{itemize}
    \item a $\phi^{k-2}$-nerve $x^{k-1}_0$ of ($k-1$)-multimaps
      connecting $x^{k-2}$, of degree $u^{k-1}_0$ in $\cat{X}$,
      namely,
      $x^{k-1}_{0i}\in\Mul^{\phi^{k-2}_i}_{u^{k-1}_{0i}}[x^{\le
        k-2}\resto{i}]$ for every $i\in I^{k-1}_0$,
    \item $x^{k-1}_1\in\Mul^{\pi_!\phi^{k-2}}_{u^{k-1}_1}[\pi_!x^{\le k-2}]$.
    \end{itemize}
  \end{enumerate}
  Then for every $k$-multimap $v\in\Mul^\pi_\cat{U}[u]$, a collection
  $\Mul^\pi_{\cat{X},v}[x^{\le k-2}](x^{k-1}_0;x^{k-1}_1)$ or
  $\Mul^\pi_{\cat{X},v}[x]$ for
  short, whose member will be called a \kore{$k$-multimap}
  $x^{k-1}_0\to x^{k-1}_1$ in $\cat{X}$ of \kore{degree} $v$.
\end{itemize}

\begin{definition}
We refer to a datum of the form $x=(x^\nu)_{0\le\nu\le k-1}$ specified
by
($0^\circ$) through ($k-1^\circ$) above (by induction in $k$), as the
\kore{type of a $k$-multimap} in $\cat{X}$ of \kore{arity}
$(I;\pi,\phi)$ and of \kore{degree} $u$.
\end{definition}

\begin{remark}
Even though we have not yet specified the form of the rest of data
for $\cat{X}$, note that the notion of the type of a
$k$-multimap ``in $\cat{X}$'' makes sense as soon as data of the forms
($0$) through ($k-1$) are given ``for $\cat{X}$''.
\end{remark}

\begin{itemize}
\item[($n-1$)](The object ``formed by \emph{($n-1$)-multimaps}''.)
  Suppose given the type $u$ of an ($n-1$)-multimap in $\cat{U}$ of
  arity given as $(I;\pi,\phi)$, and the type $x$ of an
  ($n-1$)-multimap in $\cat{X}$ of the same arity of degree $u$,
  namely, a set of data similar to ($0^\circ$) through ($k-1^\circ$),
  of ``($k$)'' above, but with $k$ substituted by $n-1$, so these will
  be ($0^\circ$) through ($n-2^\circ$) here.
  Then for every ($n-1$)-multimap $v\in\Mul^\pi_\cat{U}[u]$, an object
  $\Mul^\pi_{\cat{X},v}[x^{\le n-3}](x^{n-2}_0;x^{n-2}_1)$ of
  $\cat{M}$, or $\Mul^\pi_{\cat{X},v}[x]$ for short.
  In the case where $\cat{M}$ is $\wasreteori_0\Gpd$ or some other
  multicategory so that
  $\Mul^\pi_{\cat{X},v}[x]$ can have its objects, then those objects
  will be called \kore{($n-1$)-multimaps} $x^{n-2}_0\to x^{n-2}_1$ in
  $\cat{X}$ of \kore{degree} $v$.
  For a general $\cat{M}$, we shall call $\Mul^\pi_{\cat{X},v}[x]$ the
  object ``of \emph{($n-1$)-multimaps} in $\cat{X}$ of \emph{degree}
  $v$''.

\item[($n$)](\emph{Action} of the $n$-multimaps of $\cat{U}$.)
  Suppose given
  \begin{itemize}
  \item the type $u$ of an $n$-multimap in $\cat{U}$ of arity given as
    $(I;\pi,\phi)$,
  \item ($0^\circ$) through ($n-3^\circ$) of ($n-1$) above,
  \item ($k-2^\circ$) of ``($k$)'' above, but with $k$ substituted by
    $n$, so this will be  ($n-2^\circ$) here.
  \end{itemize}
  Then a map
  \[
  \inv{m}^\cat{X}_0(\pi)\colon\Mul^\pi_\cat{U}\bigl[u\bigr]\longto\Mul_\cat{M}\bigl(\Mul^{\phi^{n-2}}_{\cat{X},u^{n-1}_0}[x];\Mul^{\pi_!\phi^{n-2}}_{\cat{X},u^{n-1}_1}[\pi_!x]\bigr)
  \]
  of groupoids, where $\Mul^{\phi^{n-2}}_{\cat{X},u^{n-1}_0}[x]$
  denotes the ($\coprod_{I^{n-1}_0}I^{n-2}$)-family
  $\coprod_{i\in
    I^{n-1}_0}\Mul^{\phi^{n-2}_i}_{\cat{X},u^{n-1}_{0i}}[x\resto{i}]$
  of objects of $\cat{M}$, and $\pi_!x$ consists of
  $x^{\le n-3}$, $\pi_!x^{n-2}$.
  For $v$ in the source of this map, the multimap
  \[
  \inv{m}_0(\pi)(v)\colon\Mul^{\phi^{n-2}}_{u^{n-1}_0}[x]\longto\Mul^{\pi_!\phi^{n-2}}_{u^{n-1}_1}[\pi_!x]
  \]
  in $\cat{M}$ will be called the \kore{composition} operation
  \kore{along} $v$ for ($n-1$)-multimaps in $\cat{X}$.
\end{itemize}

\begin{definition}
We refer to a datum of the form $x=(x^\nu)_{0\le\nu\le n-2}$ specified
by ($0^\circ$) through ($n-2^\circ$) above, as the \kore{type of a
  $\phi^{n-2}$-nerve of ($n-1$)-multimaps} in $\cat{X}$
of \kore{degree} $u^{\le n-2}$.
\end{definition}

\begin{itemize}
\item[($\infty$)] A datum of (coherent) associativity for the action
  of $n$-multimaps, similar to a datum for a functor, of
  compatibility of the action with the composition (Definition
  \ref{def:functor}).
\end{itemize}

This completes Definition \ref{def:algebra}.

\subsubsection{}

Similarly, one can define the notion of algebra over a less coloured
$n$-theory.
See Section \ref{sec:stratum-for-colour}.

\begin{example}
A $\unity^n_\Com$-algebra enriched in a multicategory $\cat{M}$
is equivalent as a datum to an ($n-1$)-theory enriched in $\cat{M}$
(Definition~\ref{def:theory-enriched-in-multicategory}).
\end{example}

This can be generalized over an arbitrary $n$-theory $\cat{U}$
after we theorize the notion of $\cat{U}$-algebra, which we shall do
next.

\subsection{Iterated theorizations of algebra}
\label{sec:iterate-theorize-algebra}

\subsubsection{}

We would like to theorize the notion of algebra.
Let us first relax the notion.

\begin{definition}\label{def:lax-monoid}
Let $n\ge 2$ be an integer, and let $\cat{U}$ be an unenriched
$n$-theory.
Then a \kore{lax $\cat{U}$-algebra} $\cat{X}$ enriched in a
categorified multicategory $\cat{M}$ consists of data of the forms
($0$) through ($n$), specified above for Definition \ref{def:algebra}
for the same value of ``$n$'', and a datum of coherent lax
associativity
for the action of $n$-multimaps, similar to a datum for a lax
functor, of coherent compatibility of the action with the composition
(Definition \ref{def:lax-functor}).
\end{definition}

This essentially contains at least the unenriched version of a
virtualization of the notion of op-lax categorified $\cat{U}$-algebra,
namely, a \emph{theorization} of the notion of $\cat{U}$-algebra.
We shall write down an enriched version of the definition explicitly,
in a form which will be convenient shortly.
(The notion will be generalized in Section \ref{sec:enrichment}.)
A natural name for the kind of thing will turn out to be ``graded
$n$-theory''.

\begin{definition}\label{def:graded}
Let $n\ge 1$ be an integer, and let $\cat{U}$ be an $n$-theory
enriched in groupoids.
Then a \kore{$\cat{U}$-graded $n$-theory} $\cat{X}$ enriched in a
multicategory $\cat{M}$
consists of data of the forms specified below as ($0$) and ($1$)
(or just ($0$) if $n=2$), ``($k$)'' for every integer $k$ such that
$2\le k\le n-2$, ($n-1$), ($n$), ($n+1$), and ($\infty$).
\end{definition}

The forms of data are as follows.
Remark \ref{rem:extend-to-non-elemental} applies here again.

\begin{itemize}
\item[($0$)](Object.)
  For every object $u$ of $\cat{U}$, a collection $\Ob_u\cat{X}$,
  whose member
  will be called an \kore{object} of $\cat{X}$ of \kore{degree} $u$.
\item[($1$)](Multimap, in the case $n\ge 2$.)
  Suppose given ($0'$) and ($0''$) of ($1$) in Section \ref{sec:1},
  and
  \begin{itemize}
  \item[($0^\circ$)] an $S$-family $x_0$ of objects of $\cat{X}$ of
    degree $u_0$, and an object $x_1$ of degree $u_1$.
  \end{itemize}
  Then for every multimap $v\in\Mul^\pi_\cat{U}[u]$, a collection
  $\Mul^\pi_{\cat{X},v}(x_0;x_1)$ or $\Mul^\pi_{\cat{X},v}[x]$
  for short, whose member will be called an ($S$-ary)
  (\kore{$1$-})\kore{multimap} $x_0\to x_1$ in $\cat{X}$ of
  \kore{degree} $v$.
\item[($k$)]($k$-multimap, inductively for $2\le k\le n-1$.)
  Suppose given the type $u$ of a $k$-multimap in $\cat{U}$ of arity
  given as $(I;\pi,\phi)$, and the type $x$ of a $k$-multimap
  in $\cat{X}$ of the same arity of degree $u$, namely, ($0^\circ$),
  (\ref{item:graded-k-5}), ($k-2^\circ$) (or just ``($0^\circ$)'' if
  $k=2$), and ($k-1^\circ$) below:
  \begin{enumerate}
  \renewcommand{\theenumi}{\alph{enumi}}
  \addtocounter{enumi}{1}
  \item[($0^\circ$)] an $I^0$-family $x^0$ of objects of
    $\cat{X}$, of degree $u^0$, namely, $x^0_i\in\Ob_{u^0_i}\cat{X}$
    for every $i\in[I^1_0]$,
  \item\label{item:graded-k-5} if $k\ge 4$, then ($1^\circ$) through
    ($k-3^\circ$) of ($k-1$) here,
  \item[($k-2^\circ$)](in the case $k\ge 3$)
    an $I^{k-2}$-family $x^{k-2}=(x^{k-2}_i)_{i\in[I^{k-1}_0]}$, where
    $x^{k-2}_i$ is an ${\maekara{\phi^{k-2}}{i}}_!\phi^{k-3}$-nerve of
    ($k-2$)-multimaps connecting ${\maekara{\phi^{k-2}}{i}}_!x^{k-3}$
    in $\cat{X}$, of degree $u^{k-2}_i$,
  \item[($k-1^\circ$)] a $\phi^{k-2}$-nerve $x^{k-1}_0$ of
    ($k-1$)-multimaps connecting $x^{k-2}$, of degree $u^{k-1}_0$ in
    $\cat{X}$, and
    $x^{k-1}_1\in\Mul^{\pi_!\phi^{k-2}}_{u^{k-1}_1}[\pi_!x^{\le
      k-2}]$.
  \end{enumerate}
  Then for every $k$-multimap $v\in\Mul^\pi_\cat{U}[u]$, a collection
  $\Mul^\pi_{\cat{X},v}[x^{\le k-2}](x^{k-1}_0;x^{k-1}_1)$ or
  $\Mul^\pi_{\cat{X},v}[x]$ for
  short, whose member will be called a \kore{$k$-multimap}
  $x^{k-1}_0\to x^{k-1}_1$ in $\cat{X}$ of \kore{degree} $v$.
\item[($n$)](Action of the $n$-multimaps of $\cat{U}$.)
  Suppose given the type $x$ of an $n$-multimap in $\cat{X}$ of
  arity and degree given respectively as $(I;\pi,\phi)$ and $u$.
  Then a functor (to the underlying category of $\cat{M}$)
  \[
  \inv{M}^\cat{X}_0(\pi)[x^{\le
    n-2}](x^{n-1}_0,x^{n-1}_1)\colon\Mul^\pi_\cat{U}[u]\longto\cat{M}
  \]
  which will also be denoted by $\inv{M}^\cat{X}_0(\pi)[x]$ for
  short.
  For $v$ in the source of this functor, we write
  \[
  \Mul^\pi_{\cat{X},v}[x]:=\Mul^\pi_{\cat{X},v}[x^{\le
    n-2}](x^{n-1}_0;x^{n-1}_1):=\inv{M}^\cat{X}_0(\pi)[x](v).
  \]
  In the case where $\cat{M}$ is $\wasreteori_0\Gpd$ or some other
  multicategory so that
  $\Mul^\pi_{\cat{X},v}[x]$ can have its objects, then those objects
  will be called \kore{$n$-multimaps} $x^{n-1}_0\to x^{n-1}_1$ in
  $\cat{X}$ of \kore{degree} $v$.
  For a general $\cat{M}$, we shall call $\Mul^\pi_{\cat{X},v}[x]$ the
  object ``of \emph{$n$-multimaps} in $\cat{X}$ of \emph{degree}
  $v$''.

\item[($n+1$)](Associativity map.)
  Suppose given
  \begin{itemize}
  \item the arity $(I;\pi,\phi)$ of an ($n+1$)-multimap in a symmetric
    higher theory,
  \item the type $u$ of a $\phi^{n-1}$-nerve of $n$-multimaps in
    $\cat{U}$,
  \item the type $x$ of a $\phi^{n-1}$-nerve of $n$-multimaps
    in $\cat{X}$ of degree $u^{\le n-1}$, namely, ($0^\circ$) through
    ($k-2^\circ$) of ``($k$)'' above, but with $k$ substituted by
    $n+1$, so this will be ($0^\circ$) through ($n-1^\circ$) here.
  \end{itemize}
  Then a multimap
  \[
  M^\cat{X}_1(\pi)\colon\inv{M}^\cat{X}_0(\phi^{n-1})[x]\longto
  \inv{M}^\cat{X}_0(\pi_!\phi^{n-1})[\pi_!x]\compose m^\cat{U}_1(\pi)
  \]
  of functors $\Mul^{\phi^{n-1}}_\cat{U}[u]\to\cat{M}$, where the
  source of this multimap is the $\coprod_{I^n_0}I^{n-1}$-family
  $\coprod_{i\in I^n_0}{\pr_i}^*\inv{M}_0(\phi^{n-1}_i)[x\resto{i}]$,
  where $\pr_i$ denotes the projection
  $\Mul^{\phi^{n-1}}_\cat{U}[u]\to\Mul^{\phi^{n-1}_i}[u\resto{i}]$.
  For $v$ in the source of this functor, we write
  \[
  m^\cat{X}_1(\pi)_v:=M^\cat{X}_1(\pi)(v)\colon\Mul^{\phi^{n-1}}_v[x]\longto\Mul^{\pi_!\phi^{n-1}}_{\pi_!v}[\pi_!x],
  \]
  where $\pi_!v:=m^\cat{U}_1(\pi)(v)$, and call it the
  \kore{composition} operation \kore{along} $v$ for $n$-multimaps.
\item[($\infty$)] A datum of coherence for the associativity,
  corresponding to that for a lax $\cat{U}$-algebra (Definition
  \ref{def:lax-monoid}).
\end{itemize}

This completes Definition \ref{def:graded}.

\subsubsection{}

Let us consider the unenriched case where
$\cat{M}=\wasreteori_0\Gpd$.

Note that in this case, $\inv{M}_0(\pi)[x]$ above can be identified
with the datum of the canonical projection map
\[
\colim_{v\in\Mul^\pi_\cat{U}[u]}\Mul^\pi_{\cat{X},v}[x]\longto\Mul^\pi_\cat{U}[u]
\]
of groupoids.
The groupoid $\colim_v\Mul^\pi_{\cat{X},v}[x]$ of $n$-multimaps of
arbitrary degrees, will further be the groupoid of $n$-multimaps in a
symmetric $n$-theory.

Indeed, the datum $M_1$ induces ``composition'' operations
$m^\cat{X}_1(\pi):=\colim_vM_1(\pi)(v)$ on these groupoids, covering
the composition operations $m^\cat{U}_1(\pi)$ in $\cat{U}$.
Writing down the form of datum for the coherence for the associativity
in $\cat{X}$
(which is straightforward), we obtain the case $m=0$ of Proposition
\ref{prop:graded-is-overlying} below (hence our term for the notion).

The construction simply remains valid in the unenriched (higher)
categorified case $\cat{M}=\wasreteori_0\Cat_m$.

\begin{proposition}\label{prop:graded-is-overlying}
Let $n\ge 1$ be an integer, and let $\cat{U}$ be an $n$-theory
enriched in groupoids.
Then for every integer $m\ge 0$, an unenriched $m$-categorified
$\cat{U}$-graded $n$-theory (i.e., $\cat{U}$-graded $n$-theory
enriched in $\wasreteori_0\Cat_m$) is equivalent as a datum to an
unenriched $m$-categorified \textbf{symmetric} $n$-theory $\cat{Y}$
equipped with a functor $\cat{Y}\to\cat{U}$.
\end{proposition}

In order to see this more precisely, let us introduce the following
terminology, which will be justified shortly.

\begin{definition}
Let $n\ge 1$ be an integer, and $\cat{U}$ be an (unenriched)
$n$-theory.
Then, for an integer $m$ such that $0\le m\le n$, we refer to data of
the forms ($0$) through ($m-1$) specified for Definition
\ref{def:graded}, as a system of \kore{colours} up to dimension $m-1$
\kore{for a $\cat{U}$-graded higher theory}, or a system of colours
\kore{for a $\cat{U}$-graded $m$-theory}.
\end{definition}

Suppose given a system $\cat{X}$ of colours up to dimension $n-2$ for
a $\cat{U}$-graded higher theory.
Then we obtain a system $\diag_!\cat{X}$ (where
$\diag\colon\cat{U}\to\unity^n_\Com$) of colours up to dimension $n-2$
for a symmetric higher theory, by inductively defining as follows.

We first define $\Ob\diag_!\cat{X}$ as the collection whose member is
a pair $(u,x)$, where $u\in\Ob\cat{U}$ and $x\in\Ob_u\cat{X}$.
If $n\ge 3$, then inductively for an integer $k$ such that $1\le k\le
n-2$, the type of a $k$-multimap in $\cat{Y}$ will be an identical
form of datum as a pair $(u,x)$, where $u$ is the type of a
$k$-multimap in $\cat{U}$, and $x$ is the type of a
$k$-multimap in $\cat{X}$ of degree $u$.
We then inductively define $\Mul^\pi_{\diag_!\cat{X}}[(u,x)]$ as the
collection whose member is a pair $(v,y)$ consisting of
$v\in\Mul^\pi_\cat{U}[u]$ and $y\in\Mul^\pi_{\cat{X},v}[x]$.
Moreover, one can associate to every member
$(v,y)\in\Mul^\pi_{\diag_!\cat{X}}[x]$ the member
$v\in\Mul^\pi_\cat{U}[u]$.

Proposition \ref{prop:graded-is-overlying} now follows since the
category (or ($m+2$)-category) of extensions of the datum
$\diag_!\cat{X}$ to an unenriched $m$-categorified symmetric
$n$-theory equipped with a functor to $\cat{U}$, gets equated by the
described construction, to the category (or ($m+2$)-category) of
extensions of the datum $\cat{X}$ to an unenriched $m$-categorified
$\cat{U}$-graded $n$-theory.

\begin{example}\label{ex:graded-multicategory}
Let $\Ini$ denote the initial uncoloured operad in groupoids.
\begin{itemize}
\item A category is equivalent as a datum to an $\Ini$-graded
  $1$-theory.
\item A planar multicategory is equivalent as a datum to an
  $E_1$-graded $1$-theory.
\item A braided multicategory (see Fiedorowicz \cite{fiedoro})
  is equivalent as a datum to an $E_2$-graded $1$-theory.
\end{itemize}
\end{example}

\begin{example}
Since $\cat{U}$-graded $n$-theory is a theorization of
$\cat{U}$-algebra, one obtains from a $\cat{U}$-monoidal category
$\cat{X}$, a $\cat{U}$-graded $n$-theory by replacing the functors
giving the composition operations by the
bimodules\slash distributors\slash profunctors corepresented by them.
We shall say that this $n$-theory is \kore{represented} by $\cat{X}$,
and shall denote it by $\wasreteori_{n-1}\cat{X}$, where the subscript
comes from the fact that $\cat{U}$-algebra is a generalization of
($n-1$)-theory from the case $\cat{U}=\unity^n_\Com$.
\end{example}

\begin{example}
Recall that we called a plain $n$-theory, i.e., an $n$-theory which is
not considered with any grading, also a \emph{symmetric} $n$-theory.
Every symmetric $n$-theory is canonically graded by the terminal
unenriched uncoloured $n$-theory $\unity^n_\Com$, and there is no
difference between a
symmetric $n$-theory and a $\unity^n_\Com$-graded $n$-theory.
\end{example}

\subsubsection{}

For an unenriched $\cat{U}$-graded $n$-theory $\cat{X}$, let us denote the
symmetric $n$-theory underlying $\cat{X}$ (which maps to $\cat{U}$;
see Proposition \ref{prop:graded-is-overlying}) by $\diag_!\cat{X}$,
where $\diag$ denotes the unique functor $\cat{U}\to\unity^n_\Com$.
For example, for the terminal unenriched uncoloured $\cat{U}$-graded
$n$-theory $\unity^n_\cat{U}$, we have that the canonical projection
functor $\diag_!\unity^n_\cat{U}\to\cat{U}$ is an equivalence.

Using this, we can obtain a compact reformulation of the notion of
algebra, as will be given now.

Suppose given a system of colours up to dimension $n-2$ for
$\cat{U}$-graded higher theory, and let $\cat{T}$ denote the terminal
unenriched $\cat{U}$-graded $n$-theory on this system of colours.
Note that one can consider the structure of a $\cat{U}$-algebra on
this system of colours.
Indeed, the structure of a $\cat{U}$-algebra enriched in a
multicategory $\cat{M}$, is equivalent as a datum to a functor
$\diag_!\cat{T}\to\ddeloop^{n-1}\cat{M}$ of $n$-theories (and
equivalently therefore, an uncoloured $\diag_!\cat{T}$-algebra
enriched in $\cat{M}$).
It is convenient to say that a $\cat{U}$-algebra enriched in $\cat{M}$
is a functor $\cat{U}\to\ddeloop^{n-1}\cat{M}$ \emph{with strata of
  colours} up to dimension $n-2$, or colours for a $\cat{U}$-algebra.

For example, a \emph{$\cat{U}$-monoid}, i.e., a $\cat{U}$-algebra
enriched in groupoids, is naturally equivalent as a datum to a
coloured functor $\cat{U}\to\ddeloop^{n-1}\wasreteori_0\Gpd$.
In this sense, the symmetric ($n+1$)-theory
$\ddeloop^{n-1}\wasreteori_0\Gpd$ classifies (uncoloured) monoids over
unenriched symmetric $n$-theories, where the universal monoid is the
uncoloured $\ddeloop^{n-1}\wasreteori_0\Gpd$-monoid $\univalg^{n-1}$
``classified'' by the identity functor of
$\ddeloop^{n-1}\wasreteori_0\Gpd$.

\begin{proposition}\label{prop:universal-theorization}
For the universal monoid $\univalg^{n-1}$, the projection functor of
the $\ddeloop^{n-1}\wasreteori_0\Gpd$-graded $n$-theory
$\wasreteori_{n-1}\univalg^{n-1}$ is equivalent to
\[
\ddeloop^{n-1}\wasreteori_0(\Gpd_*)\longto\ddeloop^{n-1}\wasreteori_0\Gpd
\]
induced from the forgetful functor $\Gpd_*\to\Gpd$, where $\Gpd_*$
denotes the Cartesian symmetric monoidal category
of pointed groupoids.
\end{proposition}

The proof is straightforward from the definitions.

\subsubsection{}
\label{sec:theorem}

Now the notion of algebra generalizes immediately as follows.

\begin{definition}
Let $n\ge 2$ be an integer, and $\cat{U}$ be an unenriched
$n$-theory.
Then a \kore{$\cat{U}$-algebra in a symmetric $n$-theory $\cat{V}$} is
a \emph{functor $\cat{U}\to\cat{V}$ with strata of colours}
for a $\cat{U}$-algebra, namely, a functor $\diag_!\cat{T}\to\cat{V}$,
where $\cat{T}$ is the terminal unenriched $\cat{U}$-graded $n$-theory
on a system of colours
up to dimension $n-2$ for a $\cat{U}$-graded higher theory.
\end{definition}

We generalize this as follows.

\begin{definition}\label{def:coloured-lax-algebra-over-theory}
Let $n\ge 1$ be an integer, and $\cat{U}$ be an unenriched
$n$-theory.
Then an \kore{$n$-tuply coloured lax $\cat{U}$-algebra} in a
categorified symmetric $n$-theory $\cat{V}$, is a lax functor
$\cat{U}\to\cat{V}$ with strata of colours up
to dimension $n-1$, namely, a lax functor $\diag_!\cat{T}\to\cat{V}$,
where $\cat{T}$ is the terminal unenriched $\cat{U}$-graded $n$-theory
on a system of
colours up to dimension $n-1$ for a $\cat{U}$-graded higher theory.
\end{definition}

We obtain from Definition \ref{def:graded} that a $\cat{U}$-graded
$n$-theory enriched in a symmetric monoidal category $\cat{A}$, i.e.,
in $\wasreteori_0\cat{A}$, is (circularly) an $n$-tuply coloured lax
$\cat{U}$-algebra in $\ddeloop^n\cat{A}$.

\begin{remark}\label{rem:coloured-lax-algebra}
Using the definition mentioned in Remark
\ref{rem:lax-functor-general}, of a ``functor'', one can define an
$n$-tuply coloured $\cat{U}$-algebra in a symmetric
\emph{($n+1$)-theory} $\cat{W}$ as a functor $\cat{U}\to\cat{W}$ with
strata of colours up to dimension $n-1$.
In particular, a $\cat{U}$-graded $n$-theory enriched in a
multicategory $\cat{M}$ will be equivalent as a datum to an $n$-tuply
coloured $\cat{U}$-algebra in $\ddeloop^n\cat{M}$.
Note the equivalence
$\wasreteori_n\ddeloop^n\cat{A}=\ddeloop^n\wasreteori_0\cat{A}$ for a
symmetric monoidal category $\cat{A}$, which follows from Lemma
\ref{lem:deloop-forget-commutation}.
\end{remark}

We obtain the following fundamental result.

\begin{theorem}\label{thm:graded-theory-is-monoid}
Let $n\ge 1$ be an integer, and let $\cat{U}$ be an $n$-theory
enriched in groupoids.
Then a $\cat{U}$-graded $n$-theory enriched in a multicategory
$\cat{M}$ is equivalent as a datum to a $\wasreteori_n\cat{U}$-algebra
enriched in $\cat{M}$.
\end{theorem}

\begin{proof}
We shall prove the case where $\cat{M}$ is represented by a symmetric
monoidal category $\cat{A}$.
The general case follows from essentially the same (and simpler)
argument, but one needs to use Remarks~\ref{rem:coloured-lax-algebra}
and \ref{rem:lax-functor-general},
of which we have omitted the details.

Let us first note that systems of colours for a $\cat{U}$-graded
$n$-theory and for
a $\wasreteori_n\cat{U}$-algebra are identical forms of datum.
Choose and fix a datum of this form.
Then we would like to show that the categories of the structures on
this same system of colours, of a $\cat{U}$-graded $n$-theory and of
$\wasreteori_n\cat{U}$-algebras, are equivalent.
Let us prove this.

Let $\cat{T}$ denote the terminal unenriched $\cat{U}$-graded
$n$-theory on the chosen system of colours.
Then the structure of a $\cat{U}$-graded $n$-theory on those strata of
colours could be described as a lax functor
$\diag_!\cat{T}\to\ddeloop^n\cat{A}$, which Theorem
\ref{thm:functor-of-theorization} and Lemma
\ref{lem:deloop-forget-commutation} equates with the datum of a
functor
$\wasreteori_n\diag_!\cat{T}\to\ddeloop^n\wasreteori_0\cat{A}$.

Let next $\cat{J}$ denote the terminal unenriched
$\wasreteori_n\cat{U}$-graded ($n+1$)-theory on the same system of
colours.
Then the structure of a $\wasreteori_n\cat{U}$-algebra on those strata
of colours can be described as a functor
${\diag'}_!\cat{J}\to\ddeloop^n\wasreteori_0\cat{A}$, where
$\diag'\colon\wasreteori_n\cat{U}\to\unity^{n+1}_\Com$.

However, it is immediate that the strata of colours up to dimension
$n-1$, of $\wasreteori_n\diag_!\cat{T}$, and of ${\diag'}_!\cat{J}$,
are identical, and these two ($n+1$)-theories on the same strata of
colours are in fact equivalent.
The result follows.
\end{proof}

\begin{definition}\label{def:graded-by-0-theory}
Let $A$ be a $0$-theory in groupoids, i.e., a commutative monoid.
Then an \kore{$A$-graded} $0$-theory is a $\wasreteori_0A$-algebra.
\end{definition}

In particular, an unenriched $A$-graded $0$-theory is equivalent as
a datum to a commutative monoid $X$ equipped with a morphism $X\to
A$.

\subsubsection{}

The notion of $\cat{U}$-graded $n$-theory can be theorized in almost
the same way as how we
have theorized the notion of algebra over an $n$-theory.
We might call the resulting object a \kore{$\cat{U}$-graded
($n+1$)-theory}.

However, this is not really a new notion as should be expected from
Theorem \ref{thm:graded-theory-is-monoid}.

\begin{theorem}\label{thm:grade-theorized-theory}
Let $n\ge 0$ be an integer, and let $\cat{U}$ be an $n$-theory
enriched in groupoids.
Then a $\cat{U}$-graded ($n+1$)-theory is equivalent as a datum to a
$\wasreteori_n\cat{U}$-graded ($n+1$)-theory.
\end{theorem}

The proof is also similar to the proof of Theorem
\ref{thm:graded-theory-is-monoid}.
We leave the details to the interested reader.

\begin{definition}\label{def:higher-graded}
Let $n\ge 0$ be an integer, and let $\cat{U}$ be an $n$-theory
enriched in groupoids.
For an integer $m\ge n+2$, a \kore{$\cat{U}$-graded $m$-theory} is a
$\wasreteori^m_n\cat{U}$-graded $m$-theory, or equivalently, a
$\wasreteori^{m+1}_n\cat{U}$-algebra.
\end{definition}

Thus, for every $\cat{U}$, $\cat{U}$-graded $m$-theories are iterated
theorizations of $\cat{U}$-algebra for $m\ge n$.

\begin{example}
An $E_1$-graded $n$-theory is equivalent as a datum to a planar
$n$-theory defined in Section \ref{sec:planar-theory}.
\end{example}

More generally, for every multicategory $\cat{U}$ enriched in
groupoids, there is a similar description of a $\cat{U}$-graded
$n$-theory.
Recall that the notion of planar $n$-theory was defined by replacing
the category $\Fin$ in the definition of a symmetric $n$-theory, with
the category $\Ord$.
For a similar description of the notion of $\cat{U}$-graded
$n$-theory, we would like to replace $\Fin$ by a category which is an
analogue for $\cat{U}$, of $\Ord$.

In order to construct this category, note that the forgetful functor
$\wasreteori_0$ from $\cat{U}$-monoidal categories to $\cat{U}$-graded
multicategories, has a left adjoint $L_\cat{U}$, and that the unique
functor $\diag\colon\cat{U}\to\Com$ of symmetric multicategories,
where $\Com=E_\infty$ denotes the commutative operad (i.e., the
terminal unenriched symmetric operad), induces a functor
$\diag_*\colon
L_\cat{U}\unity^1_\cat{U}\to\diag^*L_\Com\unity^1_\Com=\diag^*\Fin$ of
$\cat{U}$-monoidal categories, where $\unity^1_\cat{U}$ denotes the
terminal unenriched $\cat{U}$-graded multicategory.

For simplicity, suppose first that $\cat{U}$ is uncoloured.
Then we obtain a description of a $\cat{U}$-graded $n$-theory by
replacing the category $\Fin$ in the definition \ref{def:theory} of a
symmetric $n$-theory, with $\cat{L}:=L_\cat{U}\unity^1_\cat{U}$, where
\begin{itemize}
\item by a family \emph{indexed by $J\in\cat{L}$}, we mean a family
  indexed by $\diag_*J\in\Fin$, and
\item for $I\in\Ord$, we say that a $[I]$-family $J$ of objects of
  $\cat{L}$ is \emph{elemental} if $\diag_*J$ is an elemental
  $[I]$-family in $\Fin$.
\end{itemize}

If $\cat{U}$ is instead a coloured multicategory, then the description
is similar except that we would need to have objects of the ``theory''
to have ``degrees'' in $\cat{U}$.
We leave the details to the reader.

\subsection{Iterated monoids}
\label{sec:iterated-monoid}

\subsubsection{}

Let $\cat{U}$ be an $n$-theory enriched in groupoids, and $\cat{X}$ be
a $\cat{U}$-monoid.
Then since the structure of $\cat{X}$ is a generalization of the
structure of an ($n-1$)-theory, it is natural to consider the notion
of $\cat{X}$-algebra if $n\ge 2$.

The notion can actually be reduced to the same notion in the special
case where $\cat{X}$ is the unit (or terminal uncoloured)
$\cat{U}$-monoid.
Indeed, if we denote by $\diag$ the unique functor
$\cat{U}\to\unity^n_\Com$, then an $\cat{X}$-algebra will simply be an
algebra over the unit $\diag_!\wasreteori_{n-1}\cat{X}$-monoid.
(In particular, the dependence of the notion on $\cat{U}$ will be only
through the presentation of the symmetric $n$-theory
$\cat{Y}=\diag_!\wasreteori_{n-1}\cat{X}$ as underlying
$\wasreteori_{n-1}\cat{X}$.)

\subsubsection{}
\label{sec:detheorize}

Let now $\cat{X}$ be the unit $\cat{U}$-monoid (so
$\cat{U}=\diag_!\wasreteori_{n-1}\cat{X}$).
Then the notion of $\cat{X}$-algebra will be such that the notion of
$\cat{U}$-algebra theorizes it.
In fact, the notion of $\cat{U}$-algebra `detheorizes' more times, and
the most detheorized notion will at the end be equivalent to the
notion of algebra over the unit monoid over~...~over the unit monoid
over $\cat{U}$, where we should have $n-1$ unit monoids in the
expression.

We can directly define all iteratively detheorized notions as
follows.

\begin{definition}
Let $n\ge 1$ be an integer, and $\cat{U}$ be an $n$-theory enriched in
groupoids.
Then for an integer $m$ such that $0\le m\le n-1$, a
\kore{$\cat{U}$-graded $m$-theory} enriched in a multicategory
$\cat{M}$, is a functor
$\cat{U}\to\wasreteori^n_{m+1}\ddeloop^m\cat{M}$ with strata of
colours for a $\cat{U}$-graded $m$-theory.
\end{definition}

To be explicit, in the case $m=0$, there is no colours added.
In the case $m\ge 1$, given a system of colours for a $\cat{U}$-graded
$m$-theory $\cat{X}$, the structure on it, of a $\cat{U}$-graded
$m$-theory, is a functor
$\cat{T}\to\wasreteori^n_{m+1}\ddeloop^m\cat{M}$,
where $\cat{T}$ is the symmetric $n$-theory described as follows.
\begin{itemize}
\item[($0$)] $\Ob\cat{T}$ is the collection whose member is a pair
  $(u,x)$, where $u\in\Ob\cat{U}$ and $x\in\Ob_u\cat{X}$.
\end{itemize}
By induction, for $k$ such that $1\le k\le m$, the type of a
$k$-multimap in $\cat{T}$ of a given arity $(I;\pi,\phi)$, is
specified by
\begin{itemize}
\item the type $u$ of a $k$-multimap in $\cat{U}$ of arity
  $(I;\pi,\phi)$,
\item the type $x$ of a $k$-multimap in $\cat{X}$ of the same arity of
  degree $u$.
\end{itemize}
\begin{itemize}
\item[($k$)](Inductively for $k$ such that $1\le k\le m-1$.)
  Suppose given the arity $(I;\pi,\phi)$ of a $k$-multimap, and the
  type $(u,x)$ of a $k$-multimap in $\cat{T}$ of arity
  $(I;\pi,\phi)$.
  Then $\Mul^\pi_\cat{T}[(u,x)]$ is the collection whose member is a
  pair $(v,y)$, where $v\in\Mul^\pi_\cat{U}[u]$ and
  $y\in\Mul^\pi_{\cat{X},v}[x]$.
\item[($m$)] Suppose given a datum similar to an input datum for
  ``($k$)'' above, but with $k$ substituted by $m$.
  Then $\Mul^\pi_\cat{T}[(u,x)]=\Mul^\pi_\cat{U}[u]$.
\item[($\ell$)](Inductively for $\ell$ such that $m+1\le\ell\le n$.)
  Suppose given the arity $(I;\pi,\phi)$ of a $\ell$-multimap, and the
  type of an $\ell$-multimap in $\cat{T}$ of arity $(I;\pi,\phi)$,
  which by induction, will be specified by
  \begin{itemize}
  \item the type $u$ of a $\ell$-multimap in $\cat{U}$ of arity
    $(I;\pi,\phi)$,
  \item the type $x$ of a $\phi^{m-1}$-nerve of $m$-multimaps in
    $\cat{X}$ of degree $u^{\le m-1}$.
  \end{itemize}
  Then $\Mul^\pi_\cat{T}[(u,x)]=\Mul^\pi_\cat{U}[u]$.
\end{itemize}
The composition is given by the composition in $\cat{U}$.

Thus, the notion of $\cat{U}$-graded ($n-1$)-theory coincides with the
notion of $\cat{U}$-algebra, and for all $m\ge 0$ (which may be $\ge
n$), the notion of $\cat{U}$-graded $m$-theory is iteratively an
$m$-th theorization of the notion of $\cat{U}$-graded $0$-theory.

\begin{lemma}
A $\cat{U}$-graded $m$-theory is equivalent as a datum to a
$\wasreteori_n\cat{U}$-graded $m$-theory.
\end{lemma}

The proof of this is direct from Corollary
\ref{cor:forget-to-theorization}.

\begin{definition}
Let $n\ge 0$ be an integer, and $\cat{U}$ be an $n$-theory enriched in
groupoids.
Then we refer to a $\cat{U}$-graded $1$-theory also as a
$\cat{U}$-graded \kore{multicategory}.
\end{definition}

\begin{definition}
Let $n\ge 1$ be an integer, and $\cat{U}$ be an $n$-theory enriched in
groupoids.
Let $m\ge 0$ and $\ell\ge 1$ be integers.
Then, for an $\ell$-categorified $\cat{U}$-graded $m$-theory
$\cat{X}$, we denote by $\wasreteori_m\cat{X}$, the
($\ell-1$)-categorified $\cat{U}$-graded ($m+1$)-theory
\emph{represented} by $\cat{X}$, namely, obtained by
replacing the structure functors of $\cat{X}$ by the
bimodules\slash distributors\slash profunctors corepresented by them.
\end{definition}

If $\cat{X}$ is enriched in $\ell$-categories, then it can be regarded
as $k$-categorified for any $k\ge\ell$.
However, resulting $\wasreteori_m\cat{X}$ is independent of $k$.
In particular, $\wasreteori_m\cat{X}$ is also defined for an
$\cat{U}$-graded $m$-theory $\cat{X}$ enriched in groupoids, and it is
an uncategorified $\cat{U}$-graded ($m+1$)-theory, which can be
considered as as highly categorified (trivially) as one wishes to.

\begin{definition}
Let $n\ge 1$ be an integer, and $\cat{U}$ be an $n$-theory enriched in
groupoids.
Let $m\ge 0$ be an integer, and $\cat{X}$ be a $\cat{U}$-graded
$m$-theory which is possibly higher categorified.
Then, for an integer $\ell\ge 0$, we define a (less or un-
categorified) $\cat{U}$-graded ($m+\ell$)-theory
$\wasreteori^{m+\ell}_m\cat{X}$ by the inductive relations
\[
\wasreteori^{m+\ell}_m\cat{X}=
\begin{cases}
\cat{X}&\text{if $\ell=0$},\\
\wasreteori_{m+\ell-1}\wasreteori^{m+\ell-1}_m\cat{X}&\text{if $\ell\ge 1$}.
\end{cases}
\]
\end{definition}

\subsubsection{}

The following definition essentially achieves (more than) our initial
goal of the discussions here.

\begin{definition}
Let $n\ge 0$ be an integer, and $\cat{U}$ be an $n$-theory enriched in
groupoids.
Let $m\ge 0$ be an integer, and $\cat{X}$ be a
$\cat{U}$-graded $m$-theory enriched in groupoids.
Then for an integer $\ell\ge 0$, an \kore{$\cat{X}$-graded}
$\ell$-theory is a $\diag_!\wasreteori^N_m\cat{X}$-graded
$\ell$-theory, where $N:=\max\{m,n\}$, and
$\diag\colon\wasreteori^N_n\cat{U}\to\unity^N_\Com$.
\end{definition}

\begin{example}\label{ex:terminally-graded}
Let $\cat{X}$ be the terminal uncoloured $\cat{U}$-graded $m$-theory
enriched in groupoids.
Then we have $\diag_!\wasreteori^N_m\cat{X}=\wasreteori^N_n\cat{U}$,
so an $\cat{X}$-graded
$\ell$-theory is equivalent as a datum to a $\cat{U}$-graded
$\ell$-theory, as had been predicted.
\end{example}

\subsubsection{}

In order to analyse the notions further, we shall next consider how
gradings can be altered.

Given a functor $F\colon\cat{V}\to\cat{U}$ of $n$-theories enriched in
groupoids, it
is clear that a $\cat{U}$-graded $m$-theory $\cat{X}$ gets pulled back
by $F$.
Let us denote the resulting $\cat{V}$-graded $m$-theory by
$F^*\cat{X}$.
In the case where $m=n$, and $\cat{X}$ is unenriched, the projection
$(\diag_\cat{V})_!F^*\cat{X}\to\cat{V}$ (where
$\diag_\cat{V}\colon\cat{V}\to\unity^n_\Com$) is the base
change of the projection $(\diag_\cat{U})_!\cat{X}\to\cat{U}$ by $F$
in a suitable sense.
In the case where $m\le n-1$, one has
$\maeoki{^\cat{V}}{\wasreteori}^n_m(F^*\cat{X})=F^*(\maeoki{^\cat{U}}{\wasreteori}^n_m\cat{X})$,
where the superscripts to $\wasreteori^n_m$ indicate the gradings
considered.
If $m\ge n+1$, then $F^*$ is the pull-back by
$\wasreteori^m_nF\colon\wasreteori^m_n\cat{V}\to\wasreteori^m_n\cat{U}$.

\begin{example}
For an $n$-theory $\cat{V}$ enriched in groupoids, and an uncoloured
$\cat{V}$-monoid $\cat{X}$ defined by a functor
$F\colon\cat{V}\to\ddeloop^{n-1}\wasreteori_0\Gpd$,
we have an equivalence
$\wasreteori_{n-1}\cat{X}=F^*(\wasreteori_{n-1}\univalg^{n-1})$ of
(simply coloured) $\cat{V}$-graded $n$-theories, where
$\wasreteori_{n-1}\univalg^{n-1}$ has been described in Proposition
\ref{prop:universal-theorization}.
\end{example}

Let again $\cat{U}$ be an $n$-theory enriched in groupoids.
For an integer $m\ge 0$, suppose that $\cat{V}$ is a $\cat{U}$-graded
$m$-theory enriched in groupoids, and denote by $P$, the projection
$\diag_!\wasreteori^N_m\cat{V}\to\wasreteori^N_n\cat{U}$, where $N\ge
m,n$, and $\diag\colon\cat{U}\to\unity^n_\Com$.
Then also clearly, one obtains from an unenriched (possibly higher
categorified) $\cat{V}$-graded $m$-theory
$\cat{Y}$, its push forward $P_!\cat{Y}$ as a $\cat{U}$-graded
$m$-theory, generalizing the construction $\diag_!$.

In the case $m=n$, $P_!\cat{Y}$ has the same underlying symmetric
$n$-theory as that of $\cat{Y}$, with projection to $\cat{U}$ given by
the projection to $\diag_!\cat{V}$ composed with
$P\colon\diag_!\cat{V}\to\cat{U}$.

In the case $m\le n-1$, the description of $P_!\cat{Y}$ is as
follows.
\begin{itemize}
\item[($0$)] For an object $u\in\Ob\cat{U}$, $\Ob_uP_!\cat{Y}$ is the
  collection whose member is a pair $(v,y)$, where $v\in\Ob_u\cat{V}$
  and $y\in\Ob_v\cat{Y}$.
\end{itemize}
Let $k$ be an integer such that $1\le k\le m$, and suppose given the
type $u$ of a $k$-multimap in $\cat{U}$ whose arity is given as
$(I;\pi,\phi)$.
Then by induction, the type of a $k$-multimap in $P_!\cat{X}$ of the
same arity of degree $u$, is specified by
\begin{itemize}
\item the type $v$ of a $k$-multimap in $\cat{V}$ of the same arity of
  degree $u$,
\item the type $y$ of a $k$-multimap in $\cat{Y}$ of the same arity of
  degree $v$.
\end{itemize}
\begin{itemize}
\item[($k$)](Inductively for $1\le k\le m-1$.)
  Suppose given
  \begin{itemize}
  \item a $k$-multimap $u^k$ in $\cat{U}$ whose arity and type
    are given respectively as $(I;\pi,\phi)$ and $u^{\le
      k-1}=(u^\nu)_{0\le\nu\le k-1}$,
  \item the type $(v,y)$ of a $k$-multimap in $P_!\cat{X}$ of the same
    arity of degree $u^{\le k-1}$.
  \end{itemize}
  Then $\Mul^\pi_{P_!\cat{Y},u^k}[(v,y)]$ is the collection whose
  member is a pair $(w,z)$, where $w\in\Mul^\pi_{\cat{V},u^k}[v]$ and
  $z\in\Mul^\pi_{\cat{Y},w}[y]$.
\item[($m$)] Suppose given a datum similar to an input datum for
  ``($k$)'' above, but with $k$ substituted by $m$.
  Then
  \[
  \Mul^\pi_{P_!\cat{Y},u^m}[(v,y)]=\colim_{w\in\Mul^\pi_{\cat{V},u^m}[v]}\Mul^\pi_{\cat{Y},w}[y].
  \]
\end{itemize}
The composition is given by the composition in $\cat{Y}$.

It is immediate to verify that we have an equivalence
$\maeoki{^\cat{U}}{\wasreteori}^n_m(P_!\cat{Y})\equivwith
P_!\maeoki{^\cat{V}}{\wasreteori}^n_m\cat{Y}$.

In the case $m\ge n+1$, the description above applies after $\cat{U}$
is replaced by $\wasreteori^m_n\cat{U}$.

In general, between suitable categories, the construction $P_!$
gives a left adjoint to the functor $P^*$.

\begin{remark}
A construction $P_*$, which suitably gives a \emph{right} adjoint of
$P^*$, will be considered later in Section
\ref{sec:right-adjoint-of-restriction}.
This construction is not as obvious, and will in general, only produce
an ($m+1$)-theory from an $m$-theory.
\end{remark}

\begin{example}\label{ex:push-terminal-forward}
$P_!\unity^m_\cat{V}=\cat{V}$ for the terminal unenriched uncoloured
$\cat{V}$-graded $m$-theory $\unity^m_\cat{V}$.
\end{example}

\begin{example}
In the case where $\cat{V}$ is the terminal uncoloured
$\cat{U}$-graded $m$-theory enriched in groupoids, $P$ is an
equivalence, and $P^*$ and
$P_!$ are the mutually inverse equivalences giving the identification
of Example \ref{ex:terminally-graded}.
\end{example}

\begin{example}\label{ex:left-push-kan-extension}
Consider the case where $\cat{U}$ is an initially graded $1$-theory
enriched in groupoids, and $m=0$.
In this case, a $\cat{U}$-monoid $\cat{V}$ is a groupoid-valued
functor on $\cat{U}$, and the projection
$P\colon(\diag_\cat{U})_!\wasreteori_0\cat{V}\to\cat{U}$ is the
corresponding op-fibration (fibred in groupoids).
For a $\cat{V}$-graded $0$-theory $\cat{X}$, the op-fibration
\[
(\diag_\cat{U})_!P_!\wasreteori_0X\longto\cat{U}
\]
describing the $\cat{U}$-monoid $P_!\cat{X}$ is the composite of the
op-fibrations $(\diag_\cat{V})_!\wasreteori_0X\to\wasreteori_0\cat{V}$
and $P$, agreeing with the alternative description of the
$\cat{U}$-module $P_!\cat{X}$ as the left Kan extension along $P$ of
the $(\diag_\cat{U})_!\wasreteori_0\cat{V}$-module $X$.
\end{example}

The version of this connection for $\cat{U}=\unity^n_\Com$, has been
concretely expressed in Proposition
\ref{prop:universal-theorization}.

Proposition \ref{prop:graded-is-overlying} generalizes as follows.

\begin{proposition}\label{prop:graded-projection}
Let $n\ge 0$ be an integer, $\cat{U}$ be an $n$-theory enriched in
groupoids, $m\ge 0$ be an integer, and $\cat{V}$ be a $\cat{U}$-graded
$m$-theory enriched in groupoids.
Then an unenriched $\cat{V}$-graded $m$-theory is equivalent as a
datum to an unenriched
$\cat{U}$-graded $m$-theory $\cat{X}$ equipped with a ``projection''
functor $\cat{X}\to\cat{V}$.
\end{proposition}

\begin{proof}
We may assume $m\le n$ without loss of generality, and the case $m=n$
may not be tautologous, but is obviously true.

For the case $m\le n-1$, denote by $P$, the projection
${(\diag_\cat{U})}_!\wasreteori^n_m\cat{V}\to\cat{U}$, where
$\diag_\cat{U}\colon\cat{U}\to\unity^n_\Com$.
Then, from an unenriched $\cat{V}$-graded $m$-theory $\cat{Y}$, we
obtain a $\cat{U}$-graded $m$-theory $P_!\cat{Y}$ equipped with the
functor
\[
P_!\diag_\cat{Y}\colon P_!\cat{Y}\longto P_!\unity^m_\cat{V}=\cat{V},
\]
where $\diag_\cat{Y}$ denotes the unique functor
$\cat{Y}\to\unity^m_\cat{V}$.

Conversely, given an unenriched $\cat{U}$-graded $m$-theory $\cat{X}$
with
projection $Q\colon\cat{X}\to\cat{V}$, one obtains a $\cat{V}$-graded
$m$-theory
$\bigl({(\diag_\cat{U})}_!\wasreteori^n_mQ\bigr)_!\unity^m_\cat{X}$.

We would like to show that these constructions are inverse to each
other.

The functor
$P_!\diag\colon
P_!\bigl({(\diag_\cat{U})}_!\wasreteori^n_mQ\bigr)_!\unity^m_\cat{X}\to
P_!\unity^m_\cat{V}$ is equivalent to $Q\colon\cat{X}\to\cat{V}$.

It therefore, suffices to show a natural equivalence
$\bigl({(\diag_\cat{U})}_!\wasreteori^n_mP_!\diag_\cat{Y}\bigr)_!\unity^m_{P_!\cat{Y}}\equivwith\cat{Y}$.
However, the functor
${(\diag_\cat{U})}_!\wasreteori^n_mP_!\diag_\cat{Y}$ can be identified
with the functor
\[
{(\diag_\cat{U}\compose\wasreteori^n_mP)}_!\maeoki{^\cat{V}}{\wasreteori}^n_m\diag_\cat{Y}\colon{(\diag_\cat{U}\compose\wasreteori^n_mP)}_!\maeoki{^\cat{V}}{\wasreteori}^n_m\cat{Y}\longto{(\diag_\cat{U}\compose\wasreteori^n_mP)}_!\maeoki{^\cat{V}}{\wasreteori}^n_m\unity^m_\cat{V},
\]
so we obtain
\begin{align*}
\maeoki{^\cat{V}}{\wasreteori}^n_m\bigl({(\diag_\cat{U})}_!\wasreteori^n_mP_!\diag_\cat{Y}\bigr)_!\unity^m_{P_!\cat{Y}}
&=\bigl({(\diag_\cat{U})}_!\wasreteori^n_mP_!\diag_\cat{Y}\bigr)_!\maeoki{^{P_!\cat{Y}}}{\wasreteori}^n_m\unity^m_{P_!\cat{Y}},\\
&=\bigl({(\diag_\cat{U}\compose\wasreteori^n_mP)}_!\maeoki{^\cat{V}}{\wasreteori}^n_m\diag_\cat{Y}\bigr)_!\unity^n_\cat{Y}\\
&=\maeoki{^\cat{V}}{\wasreteori}^n_m\cat{Y},
\end{align*}
from which the result follows.
\end{proof}

\section{Enrichment of higher theories}
\label{sec:enrichment}

\subsection{Introduction}

The subject of this section will be a general notion of enrichment for
graded higher theories.
We shall show how some previous notions such as grading by a graded
higher theory, can be compactly understood using the new notions
introduced in this section.
We shall also show how the new framework helps with considering
push-forward construction `on the right side', along some functors of
higher theories.
We shall also discuss a construction for higher theories which is
related to the Day convolution for symmetric monoidal categories, and
leads to a notion of algebra over an enriched higher theory.

\subsection{Enriched theories}
\label{sec:enriched-theory}
\subsubsection{}

Given a kind of algebraic structure, one motivation for theorizing it
was to generalize it to similar structure definable in a theorized
form of the same kind of algebraic structure.
Specifically, we would like to define the kind in question of
algebraic structure in a theorized structure $\cat{V}$, as a
(coloured) morphism to $\cat{V}$
from the terminal unenriched theorized structure.

For example, for an $n$-theory $\cat{U}$ enriched in groupoids, we
would like to define the notion of $\cat{U}$-algebra in a
$\cat{U}$-graded $n$-theory, since $\cat{U}$-graded $n$-theory is the
kind of structure which theorize the structure of a $\cat{U}$-algebra.
More generally, the following definition seems reasonable, which can
be considered as giving a quite general manner of \emph{enrichment},
of the notion of graded higher theory.

\begin{definition}\label{def:enrich}
Let $n\ge 0$ be an integer, and $\cat{U}$ be an $n$-theory enriched in
groupoids.
Let $m\ge 0$ and $M\ge m+1$ be integers, and $\cat{V}$ be a
$\cat{U}$-graded $M$-theory.
Then an \kore{$m$-theory in $\cat{V}$} is a functor
$\unity^M_\cat{U}\to\cat{V}$ of $\cat{U}$-graded $M$-theories with
strata of colours for a $\cat{U}$-graded $m$-theory.
Namely, it consists of
\begin{itemize}
\item a system of colours up to dimension $m-1$ for a $\cat{U}$-graded
  higher theory,
\item a functor $\wasreteori^M_{m+1}\cat{T}\to\cat{V}$, where
  $\cat{T}$ denotes the terminal unenriched $\cat{U}$-graded
  ($m+1$)-theory on the chosen system of colours.
\end{itemize}

The system of colours will be called the system of colours \kore{of}
the $m$-theory defined.

In the case $m=n-1$, $m$-theory will also be called an
\kore{algebra}.
\end{definition}

\begin{remark}\label{rem:terminology-redundancy-enriched-algebra}
In order to avoid redundancy, we call the defined kind of object
simply an ``$m$-theory'' (or ``algebra'') instead of a
``$\cat{U}$-graded $m$-theory'' (or $\cat{U}$-algebra) when it is
understood that a $\cat{U}$-grading is contained is the datum
$\cat{V}$.
We may explicitly refer to the $m$-theory as a $\cat{U}$-graded
$m$-theory when ($M\ge n$ and) $\cat{V}$ is a priori, just a symmetric
$M$-theory, and different higher theories may be being considered
over which we would like to grade $\cat{V}$, perhaps including the
commutative operad $\Com$.
This would clarify that a $\cat{U}$-grading is considered for
$\cat{V}$.
Note however, that this would be still ambiguous if more than one
$\cat{U}$-gradings are being considered for $\cat{V}$.
\end{remark}

Similarly to the definitions in Section \ref{sec:stratum-for-colour},
there are also less coloured versions of the notion of enriched graded
theory.
These are the cases where the $\cat{U}$-graded theory $\cat{T}$ which
were considered in Definition \ref{def:enrich}, are less coloured.

There is also a lax generalization.

\begin{definition}
Let $n\ge 0$ be an integer, and $\cat{U}$ be an $n$-theory enriched in
groupoids.
Let $m\ge 0$ and $M\ge m+1$ be integers, and let $\cat{V}$ be a
possibly higher categorified $\cat{U}$-graded $M$-theory.
Then for an integer $\ell\ge 0$, an \kore{$\ell$-lax} $m$-theory
in $\cat{V}$ is an $m$-theory in $\wasreteori^{M+\ell}_M\cat{V}$.
\end{definition}

$\cat{V}$ needs to be at least $\ell$-categorified in order for this
to be properly more general than the ($\ell-1$)-lax notion.

If $\cat{V}$ is represented by a once more categorified
$\cat{U}$-graded ($M-1$)-theory $\cat{W}$, then an
$\ell$-lax $m$-theory in $\cat{V}=\wasreteori_{M-1}\cat{W}$, is also
an ``$m$-tuply coloured'' ($\ell+1$)-lax ($m-1$)-theory in $\cat{W}$.

\subsubsection{}

As we have discussed in Section \ref{sec:introduction}, we obtain the
following in a low ``theoretic'' level of algebra.
For a functorial formulation of the following, see Proposition
\ref{prop:enriching-multicategory}.

\begin{proposition}
For every coloured symmetric operad $\cat{U}$ in groupoids, the notion
of coloured
$\cat{U}$-graded operad in a symmetric monoidal category has a
generalization in a ($\cat{U}\tensor E_1$)-monoidal category.
Namely, there is a notion of coloured $\cat{U}$-graded operad in a
($\cat{U}\tensor E_1$)-monoidal category, such that the notion of
coloured $\cat{U}$-graded operad in a symmetric monoidal category
coincides with the notion of coloured $\cat{U}$-graded operad in its
underlying ($\cat{U}\tensor E_1$)-monoidal category.
\end{proposition}

Indeed, for a ($\cat{U}\tensor E_1$)-monoidal category $\cat{A}$, it
sufficed to define a coloured $\cat{U}$-graded operad in $\cat{A}$ as
a $1$-theory in the $\cat{U}$-graded $2$-theory
$\wasreteori^2_0\deloop\cat{A}$, where $\deloop\cat{A}$ denotes the
$\cat{U}$-monoidal $2$-category obtained by categorically delooping
$\cat{A}$ using the $E_1$-monoidal structure.
To be more precise, the $1$-theory in $\wasreteori^2_0\deloop\cat{A}$
should satisfy the following condition as a functor
\[
  F\colon\cat{T}\longto\wasreteori^2_0\deloop\cat{A},
\]
where $\cat{T}$ denotes the unenriched $\cat{U}$-graded $2$-theory
which is terminal on the colours of the $1$-theory.
The condition is that $F$ should send any object of $\cat{T}$ (i.e.,
colour of the $1$-theory) to the base object of $\deloop\cat{A}$.

In the case where $\cat{U}=\Ini, E_1, E_2$, this coincides with the
familiar notion.
See Example \ref{ex:graded-multicategory}.

\begin{remark}
It seems natural that the notion of $\cat{U}$-graded multicategory can
further be generalized to be enriched in a ($\cat{U}\tensor
E_1$)-graded multicategory.
We hope to come back to this in a sequel to this work.
\end{remark}

\subsubsection{}

We can more generally consider the similarly general enrichement of
the notion of $\cat{U}$-graded higher theory in the
case where $\cat{U}$ is graded by an $N$-theory, where $N\ge n+1$.
In this case, a $\cat{U}$-graded higher theory in the previous sense
was simply a $\diag_!\wasreteori^N_n\cat{U}$-graded higher theory.

The notion is therefore a special case of the notion defined in
Definition \ref{def:enrich}.
To be explicit, we have the following.
(There will be changes in the notation.)

Let
\begin{itemize}
\item $n\ge 0$ be an integer, and $\cat{U}$ be an $n$-theory enriched
  in groupoids,
\item $m\ge 0$ be an integer, and $\cat{X}$ be a
  $\cat{U}$-graded $m$-theory enriched in groupoids,
\item $\ell\ge 0$ and $L\ge\ell+1$ be integers, and $\cat{Y}$ be an
  $\cat{X}$-graded $L$-theory.
\end{itemize}
Then $\cat{Y}$ is a $\diag_!\wasreteori^N_m\cat{X}$-graded $L$-theory,
where $N:=\max\{m,n\}$, and $\diag\colon\cat{U}\to\unity^n_\Com$, so
an $\ell$-theory \emph{in $\cat{Y}$}
makes sense according to Definition \ref{def:enrich}.

\begin{definition}\label{def:algebra-in-graded-theory}
Let
\begin{itemize}
\item $n\ge 0$ be an integer, and $\cat{U}$ be an $n$-theory enriched
  in groupoids,
\item $m\ge 0$ be an integer, and $\cat{X}$ be a $\cat{U}$-graded
  $m$-theory enriched in groupoids,
\item $\ell\ge 0$ and $L\ge\ell+1$ be integers, and $\cat{V}$ be a
  $\cat{U}$-graded $L$-theory.
\end{itemize}
Then an \kore{$\cat{X}$-graded} $\ell$-theory in $\cat{V}$ is an
$\ell$-theory in the $\cat{X}$-graded $L$-theory $P^*\cat{V}$, where
$P$ denotes the projection
$\diag_!\wasreteori^N_m\cat{X}\to\wasreteori^N_n\cat{U}$, where
$N:=\max\{m,n\}$, and
$\diag\colon\cat{U}\to\unity^n_\Com$.

In the case $\ell=m-1$, an $\cat{X}$-graded $\ell$-theory will also be
called an \kore{$\cat{X}$-algebra}.
\end{definition}

Thus, a $\cat{U}$-graded $m$-theory in $\cat{V}$ is an uncoloured
$\cat{T}$-algebra in $\cat{V}$ for an unenriched $\cat{U}$-graded
($m+1$)-theory
$\cat{T}$ which is terminal on a system of colours up to
dimension $m-1$.

\begin{remark}
For the notion of Definition \ref{def:algebra-in-graded-theory}, we
are refraining from saying ``\emph{$\cat{U}$-graded}
$\cat{X}$-graded'' theory when we do not intend to emphasize
the $\cat{U}$-grading, and this is part of our convention on the
terminology.
See Remark \ref{rem:terminology-redundancy-enriched-algebra}.
\end{remark}

\begin{example}\label{ex:enrichment-in-groupoid}
For $\cat{U}$ and $\cat{X}$ as in Definition
\ref{def:algebra-in-graded-theory}, a $\cat{X}$-graded $\ell$-theory
enriched in a multicategory $\cat{M}$, is equivalent as a datum to an
$\cat{X}$-graded $\ell$-theory in the $\cat{U}$-graded
($\ell+1$)-theory $\diag^*\ddeloop^\ell\cat{M}$, where
$\diag\colon\cat{U}\to\unity^n_\Com$.
\end{example}

\begin{example}
In Definition \ref{def:algebra-in-graded-theory}, if $\cat{X}$ is the
terminal uncoloured $\cat{U}$-graded $m$-theory enriched in groupoids,
then $P$ is an
equivalence, and an $\cat{X}$-graded $\ell$-theory in $\cat{V}$ is
equivalent as a datum to an $\ell$-theory in $\cat{V}$.
\end{example}

\subsection{Graded theories as lifts of an algebra}
\label{sec:graded-theory-as-lift}

\subsubsection{}

Using Definition \ref{def:algebra-in-graded-theory}, an
$\cat{X}$-graded $\ell$-theory in $\cat{V}$ in the notation there, can
be written as an appropriate coloured functor of $\cat{U}$-graded
higher theories.
On the other hand, $\cat{X}$ itself may be defined by a coloured
functor
$F\colon\unity^{m+1}_\cat{U}\to\diag^*\ddeloop^m\wasreteori_0\Gpd$ of
$\cat{U}$-graded theories.
In this situation, one might wish to describe a higher theory graded
by $\cat{X}$, directly in terms of $F$.
We shall show that it can indeed be described as a coloured lift of
$F$.

Note that we may assume $m=L$($\ge \ell+1$) without loss of
generality.
We shall first discuss a result in this situation (\emph{Proposition
\ref{prop:graded-low-theory-as-lift}}).

For application in Section \ref{sec:right-adjoint-of-restriction}, we
shall also give an analogous result \emph{Proposition
\ref{prop:graded-theory-as-lift}} for the case ``$\ell=m$''.

\subsubsection{}

Let us first recall common notation.

If $\cat{C}$ is a category and $x$ is an object of $\cat{C}$,
then we denote by $\cat{C}_{/x}$, the category of objects of
$\cat{C}$ lying over $x$, i.e., equipped with a map to $x$.
More generally, if a category $\cat{D}$ is equipped with a functor to
$\cat{C}$, then we define
$\cat{D}_{/x}:=\cat{D}\times_{\cat{C}}\cat{C}_{/x}$.
Note here that $\cat{C}_{/x}$ is mapping to $\cat{C}$ by the functor
which forgets the structure map to $x$.
Note that the notation is abusive in that the name of the functor
$\cat{D}\to\cat{C}$ is dropped from it.
In order to avoid this abuse from causing any confusion, we shall use
this notation only when the considered functor $\cat{D}\to\cat{C}$ is
clear from the context.

\subsubsection{}

To get on the task now, let us denote by $\CAT$ the very large
category of large categories.
Consider this as symmetric monoidal by the Cartesian structure, and
give the functor category $\Fun(\Gpd^\op,\CAT)$ the structure of a
multicategory by the Day convolution.
Namely, we consider the structure of a multicategory underlying (or
``represented'' by) the symmetric monoidal structure given by the Day
convolution.
Then the Grothendieck construction defines a functor
$G\colon\Fun(\Gpd^\op,\CAT)\to\wasreteori_0(\CAT_{/\Gpd})$ of
multicategories, where $\CAT_{/\Gpd}$ is made symmetric monoidal by
the structure
induced from the (Cartesian) symmetric monoidal structure of $\Gpd$.

Furthermore, considering $\Gpd$ as a symmetric monoidal (full)
subcategory of $\CAT$, we obtain the composite
\[
\wasreteori_0\Gpd\xlongrightarrow{\mathrm{Yoneda}}\Fun(\Gpd^\op,\Gpd)\xlongrightarrow{\mathrm{inclusion}}\Fun(\Gpd^\op,\CAT)\xlongrightarrow{G}\wasreteori_0(\CAT_{/\Gpd}),
\]
where the structures of multicategories on the functor categories are
by the Day convolution.
We denote this functor by
$\gy\colon\wasreteori_0\Gpd\to\wasreteori_0(\CAT_{/\Gpd})$,
$X\mapsto\gy_X$, so $\gy_X$ is defined by the forgetful functor
$\gyzen X:=\Gpd_{/X}\to\Gpd$.

Now let $n\ge 0$ be an integer, and $\cat{U}$ be an $n$-theory.
Then for an integer $m\ge 0$, we obtain from a $\cat{U}$-graded
$m$-theory $\cat{V}$ enriched in $\Gpd$, a $\cat{U}$-graded $m$-theory
$\gy_*\cat{V}$ enriched in $\CAT_{/\Gpd}$, by postcomposition with
$\ddeloop^m\gy\colon\ddeloop^m\wasreteori_0\Gpd\to\ddeloop^m\wasreteori_0(\CAT_{/\Gpd})$.

$\gy_*\cat{V}$ can be considered as a $\cat{U}$-graded $m$-theory
$\gyzen_*\cat{V}$ enriched in $\CAT$, equipped with a functor
$\gyzen_*\cat{V}\to\diag^*\ddeloop^m\Gpd$, where
$\diag\colon\cat{U}\to\unity^n_\Com$.
In particular, we obtain a functor
$\wasreteori_m\gyzen_*\cat{V}\to\diag^*\ddeloop^m\wasreteori_0\Gpd$ by
Lemma \ref{lem:deloop-forget-commutation}.

\begin{proposition}\label{prop:graded-low-theory-as-lift}
Let $n\ge 0$ be an integer, and $\cat{U}$ be an $n$-theory enriched in
groupoids.
Let $m\ge 1$ be an integer, and $\cat{T}$ be a
$\cat{U}$-graded ($m+1$)-theory enriched in groupoids.
Suppose that an uncoloured $\cat{T}$-monoid $\cat{X}$ is defined by a
functor
$F\colon\cat{T}\to\diag^*\ddeloop^m\wasreteori_0\Gpd$,
where $\diag\colon\cat{U}\to\unity^n_\Com$, of $\cat{U}$-graded
($m+1$)-theories (e.g., these data may be specifying a
$\cat{U}$-graded $m$-theory).

Then, for an integer $\ell$ such that $0\le\ell\le m-1$, an
$\cat{X}$-graded $\ell$-theory in a $\cat{U}$-graded $m$-theory
$\cat{V}$ enriched in groupoids, is equivalent as a datum to a lift
with strata of colours up to dimension $\ell-1$, of $F$ to
$\wasreteori_m\gyzen_*\cat{V}$.
\end{proposition}

\begin{proof}
Since $\cat{X}$ is uncoloured, a system of colours for an
$\cat{X}$-graded $\ell$-theory (up to dimension $\ell-1$) is
equivalent as a datum to a system of colours for an $\cat{T}$-graded
$\ell$-theory.
If $\cat{J}$ is the terminal unenriched $\cat{T}$-graded
($\ell+1$)-theory on a
system of colours up to dimension $\ell-1$, then, from the
definitions, a correspondence is immediate between the structures on
this system of colours, of $\cat{X}$-graded $\ell$-theories in
$\cat{V}$, and lifts of $F$ to functors
$P_!\wasreteori^{m+1}_{\ell+1}\cat{J}\to\wasreteori_m\gyzen_*\cat{V}$,
where $P$ denotes the projection
$\diag_!\wasreteori^N_{m+1}\cat{T}\to\wasreteori^N_n\cat{U}$,
$N\ge m+1,n$.
\end{proof}

\begin{corollary}\label{cor:ab-ba}
Let $F$, $\cat{X}$ be as in Proposition.
Then, for an integer $\ell$ such that $0\le\ell\le m-1$,
an $\cat{X}$-graded $\ell$-theory enriched
in a symmetric multicategory $\cat{M}$ is equivalent as a datum to a
lift with strata of colours up to dimension $\ell-1$, of $F$ to
$\diag^*\ddeloop^\ell\wasreteori_{m-\ell}\gyzen_*\wasreteori^{m-\ell}_1\cat{M}$.
\end{corollary}

\begin{proof}
An $\cat{X}$-graded $\ell$-theory enriched in a symmetric
multicategory $\cat{M}$, is an $\cat{X}$-graded $\ell$-theory in
$\diag^*\cat{W}$, where $\cat{W}$ denotes the symmetric
($\ell+1$)-theory $\ddeloop^\ell\cat{M}$.
Proposition identifies this with a coloured lift of $F$ to
$\diag^*\wasreteori_m\gyzen_*\wasreteori^m_{\ell+1}\cat{W}$.
Moreover, there is an equivalence
\[
\wasreteori_m\gyzen_*\wasreteori^m_{\ell+1}\cat{W}=\ddeloop^\ell\wasreteori_{m-\ell}\gyzen_*\wasreteori^{m-\ell}_1\cat{M}.
\]
of unenriched ($m+1$)-theories lying over
$\ddeloop^m\wasreteori_0\Gpd=\ddeloop^\ell\ddeloop^{m-\ell}\wasreteori_0\Gpd$.
\end{proof}

\subsubsection{}

In Section \ref{sec:right-adjoint-of-restriction}, we shall use a
natural analogue of Proposition \ref{prop:graded-low-theory-as-lift}
for $\ell=m$.
The way how we obtain it will be by restricting the context.

Let us denote by $\CCAT$ the very large \emph{$2$-category} of large
categories extending the $1$-category $\CAT$.
In order to formulate the result, we first extend $\gy$ to the
composite
\begin{multline*}
\gy\colon\wasreteori_0\CCAT\xlongrightarrow{\mathrm{Yoneda}}\FFun(\CCAT^\op,\CCAT)\\
\xlongrightarrow{\mathrm{restriction}}\FFun(\Gpd^\op,\CCAT)\xlongrightarrow{G}\wasreteori_0(\CCAT_{/\Gpd})
\end{multline*}
of functors of categorified multicategories, where $\FFun$ indicates
the $2$-categories of functors extended to categorified
multicategories by the Day convolution, and $\Gpd$ is considered as a
symmetric monoidal subcategory of $\CCAT$.
Thus, for $\cat{C}\in\CCAT$, the object $\gy_\cat{C}\in\CCAT_{/\Gpd}$
is defined by the forgetful functor
$\gyzen\cat{C}:=\Gpd_{/\cat{C}}\to\Gpd$.

Now let $n\ge 0$ be an integer, and $\cat{U}$ be an $n$-theory.
Then as before, for an integer $m\ge 0$, we obtain from a
$\cat{U}$-graded $m$-theory $\cat{V}$ enriched in $\CCAT$,
\begin{itemize}
\item a $\cat{U}$-graded $m$-theory $\gyzen_*\cat{V}$ enriched in
  $\CCAT$,
\item a functor
$\wasreteori_m\gyzen_*\cat{V}\to\diag^*\ddeloop^m\wasreteori_0\Gpd$,
\end{itemize}
by considering $\gy_*\cat{V}$.

\begin{proposition}\label{prop:graded-theory-as-lift}
For $\cat{U}$, $F$, $\cat{X}$ as in Proposition
\ref{prop:graded-low-theory-as-lift}, an $\cat{X}$-graded $m$-theory
in the ($m+1$)-theory represented by a $\cat{U}$-graded $m$-theory
$\cat{V}$ enriched in $\CCAT$, is equivalent as a datum to a lift with
strata of colours up to dimension
$m-1$, of $F$ to $\wasreteori_m\gyzen_*\cat{V}$.
\end{proposition}

The proof is similar to the proof of Proposition
\ref{prop:graded-low-theory-as-lift}.

\begin{corollary}\label{cor:graded-theory-as-lift}
For $F$, $\cat{X}$ as in Proposition
\ref{prop:graded-low-theory-as-lift}, an $\cat{X}$-graded $m$-theory
enriched in a symmetric monoidal category $\cat{A}$, is equivalent as
a datum to a lift with strata of colours up to dimension $m-1$, of $F$
to $\diag^*\ddeloop^m\wasreteori_0\gyzen_*\cat{A}$.
\end{corollary}

\begin{proof}
This is the case $\cat{V}=\diag^*\ddeloop^m\cat{A}$ of Proposition
since we have an equivalence
\[
\gyzen_*\ddeloop^m\cat{A}=\ddeloop^m\gyzen_*\cat{A}
\]
of ($m+1$)-theories enriched in $\CAT$, lying over $\ddeloop^m\Gpd$.
\end{proof}

\subsection{The right adjoint of the restriction of degrees}
\label{sec:right-adjoint-of-restriction}

\subsubsection{}
\label{sec:right-push-begin}

In Section \ref{sec:iterated-monoid}, we have considered the
left adjoint to the functor `restricting' the degrees for graded
theories.
We would like to consider the \emph{right} adjoint in the following
situation.

Let
\begin{itemize}
\item $n\ge 0$ be an integer, and $\cat{U}$ be an $n$-theory enriched
  in groupoids,
\item $m\ge 0$ be an integer, and $\cat{V}$ be an
  $\cat{U}$-graded $m$-theory enriched in groupoids,
\item $\cat{X}$ be an unenriched $\cat{V}$-graded $m$-theory.
\end{itemize}
Denote by $P$ the projection functor
$\diag_!\wasreteori^N_m\cat{V}\to\wasreteori^N_n\cat{U}$, where $N\ge
m,n$, and $\diag\colon\cat{U}\to\unity^n_\Com$.
Then it turns out that we can construct a $\cat{U}$-graded
\emph{($m+1$)-theory} $P_*\cat{X}$ having an appropriate universal
property.

We shall do this construction in two steps.
The key observation is that we can reduce the situation above into two
simpler situations according to the factorization of $P$ as
\[
\diag_!\wasreteori^N_m\cat{V}\xlongrightarrow{R}\diag_!\wasreteori^N_m\cat{T}\xlongrightarrow{Q}\wasreteori^N_n\cat{U},
\]
where $\cat{T}$ denotes the $\cat{U}$-graded $m$-theory enriched in
groupoids which is
terminal on the system of colours of $\cat{V}$, so $\cat{V}$ can be
identified with an uncoloured $\cat{T}$-graded $m$-theory.

The two cases which we shall treat separately below, will imply
respectively
that we obtain a $\cat{T}$-graded ($m+1$)-theory $R_*\cat{X}$, and
that we obtain from this, a $\cat{U}$-graded ($m+1$)-theory
$Q_*R_*\cat{X}$.
Moreover, it will follow that the construction $P_*:=Q_*R_*$ has a
universal property desired of the right push-forward.

\begin{remark}
As we shall see in the construction, $Q_*$ will produce an
($m+1$)-theory in which the groupoids of ($m+1$)-multimaps may not
necessarily be small, even if the groupoids of $m$-multimaps are all
small in $\cat{X}$.
\end{remark}

Let us see the constructions $Q_*$ and $R_*$.

\subsubsection{}

In order to construct $R_*$ above, it suffices to construct a
$\cat{U}$-graded ($m+1$)-theory $P_*\cat{X}$ in the spacial case of
the original situation where $\cat{V}$ is an \emph{uncoloured}
$\cat{U}$-graded $m$-theory.
This can be done with the following construction at the universal
level.

Recall from Section \ref{sec:iterate-theorize-algebra} and Corollary
\ref{cor:graded-theory-as-lift} that $\cat{V}$ corresponds to a
functor
$F_\cat{V}\colon\wasreteori^N_n\cat{U}\to\wasreteori^N_{m+1}\ddeloop^m\wasreteori_0\Gpd$,
where $N\ge n,m+1$, and $\cat{X}$ corresponds then to a functor
$F_\cat{X}\colon\cat{J}\to{F_\cat{V}}^*\ddeloop^m\wasreteori_0\gyzen_*\Gpd$
of unenriched $\cat{U}$-graded ($m+1$)-theories, where $\cat{J}$ is
terminal on the system of colours of $\cat{X}$.

Let $\CCocor$ denote the symmetric monoidal $2$-category of groupoids
and cocorrespondences.
Thus, its object is a groupoid, and for groupoids $X,Y$, the category
$\Map_\CCocor(X,Y)$ is the category formed naturally by the
diagrams of the form
\[
X\longto M\longfrom Y
\]
in $\Gpd$, where the groupoid $M$ is allowed to vary arbitrarily.
Composition is done by the obvious push-out operation.
The symmetric monoidal structure is induced from the Cartesian product
in $\Gpd$.

Then there is a symmetric monoidal lax functor
$\Gamma\colon\gyzen_*\Gpd\to\CCocor$ which sends
\begin{itemize}
\item an object of $\gyzen_*\Gpd$ given by $X\in\Gpd$ and a functor
  $F\colon X\to\Gpd$, to the groupoid $\lim_XF$,
\item a map $(X,F)\to(Y,G)$ in $\gyzen_*\Gpd$ given by a map $f\colon X\to
  Y$ and a map $\lift{f}\colon F\to f^*G$, to the map in $\CCocor$
  corresponding to the diagram
  \[
  \lim_XF\xlongrightarrow{\lift{f}}\lim_Xf^*G\xlongleftarrow{f^*}\lim_YG
  \]
  in $\Gpd$,
\end{itemize}
and extends these data naturally.
From this, we obtain the induced functor
$\Gamma\colon\wasreteori^2_0\gyzen_*\Gpd\to\wasreteori^2_0\CCocor$ of
$2$-theories, and hence
\[
\ddeloop^m\Gamma\colon\wasreteori_{m+1}\ddeloop^m\wasreteori_0\gyzen_*\Gpd=\ddeloop^m\wasreteori^2_0\gyzen_*\Gpd\longto\ddeloop^m\wasreteori^2_0\CCocor.
\]

Denote by $\CCocor_*$ the symmetic monoidal $2$-category of
co-correspondences in the category $\Gpd_*$ of pointed groupoids.
The forgetful functor $\Gpd_*\to\Gpd$ (which preserves push-outs and
direct products) induces a symmetric monoidal functor
$\CCocor_*\to\CCocor$.
In particular, we can consider the ($m+2$)-theory
$\wasreteori_{m+1}\ddeloop^m\wasreteori_0\CCocor_*$ as graded by
$\wasreteori_{m+1}\ddeloop^m\wasreteori_0\CCocor=\ddeloop^m\wasreteori^2_0\CCocor$.
This $\wasreteori_{m+1}\ddeloop^m\wasreteori_0\CCocor$-graded
($m+2$)-theory is in fact representable by a ($m+1$)-theory.
This will be an instance of the following.

\begin{lemma}
Let $\cat{C}$ and $\cat{D}$ be $n$-theories enriched in $\CAT$, and
suppose given a functor $P\colon\cat{D}\to\cat{C}$ of categorified
$n$-theories.
Then the $\wasreteori_n\cat{C}$-graded ($n+1$)-theory corresponding to
the induced functor
$P\colon\wasreteori_n\cat{D}\to\wasreteori_n\cat{C}$ of symmetric
($n+1$)-theories, is representable by a categorified
$\wasreteori_n\cat{C}$-graded $n$-theory if the functors induced by
$P$ on the categories of $n$-multimaps, are all op-fibrations.

Moreover, the representing $\wasreteori_n\cat{C}$-graded $n$-theory in
this situation, is enriched in fact in groupoids if and only if the
op-fibrations are all fibred in groupoids.
\end{lemma}

\begin{proof}
Let $y$ be the type of an $n$-multimap in $\cat{D}$, and $x$ be an
$n$-multimap in $\cat{C}$ of type $Py$.
Then the category of $n$-multimaps in the representing $n$-theory, of
type $y$ of degree $x$, will be the fibre of the op-fibration over $x$
in the category of $n$-multimaps of type $y$ in $\cat{D}$.
We would like to let ($n+1$)-multimaps in $\wasreteori_n\cat{C}$ act
on these categories, but an ($n+1$)-multimap in $\wasreteori_n\cat{C}$
is a morphism in the category of appropriate $n$-multimaps in
$\cat{C}$, which acts on the fibres of the op-fibrations.
It is straightforward to check that this extends to the structure of a
categorified $\wasreteori_n\cat{C}$-graded $n$-theory which represents
the $\wasreteori_n\cat{C}$-graded ($n+1$)-theory
$\wasreteori_n\cat{D}$.

The construction also shows the second statement.
\end{proof}

\begin{remark}
The condition is also necessary.

To see this, note that the category of $n$-multimaps in $\cat{C}$ of a
given arity can be recovered from $\wasreteori_n\cat{C}$, as the
category formed by $n$-multimaps in $\wasreteori_n\cat{C}$ of the same
arity, under the unary ($n+1$)-multimaps between them in
$\wasreteori_n\cat{C}$.
Moreover, we obtain a similar description of the category of
$n$-multimaps in $\cat{D}$ of a given arity, as a category \emph{lying
  over} the corresponding category for $\cat{C}$, in view of
the noted description of the latter category.
On the other hand, for a categorified $\wasreteori_n\cat{C}$-graded
$n$-theory $\cat{E}$, the category formed by $n$-multimaps in
$\diag_!\wasreteori_n\cat{E}$ of a given arity (where
$\diag\colon\wasreteori_n\cat{C}\to\unity^{n+1}_\Com$) under the unary
($n+1$)-multimaps between them in $\diag_!\wasreteori_n\cat{E}$, can
be seen to be lying over the corresponding category in
$\cat{F}:=\wasreteori_n\cat{C}$, as the op-fibration corresponding to
the action of the unary ($n+1$)-multimaps in $\cat{F}$ between those
$n$-multimaps, on the categories of $n$-multimaps in $\cat{E}$ of the
same arity of appropriate degrees.
Therefore, if $\wasreteori_n\cat{D}$ corresponds to the
$\cat{F}$-graded ($n+1$)-theory $\wasreteori_n\cat{E}$, then we
conclude that the assumption of Proposition is satisfied by
$\cat{D}$.
\end{remark}

We can indeed apply Lemma to the forgetful functor
$\ddeloop^m\wasreteori_0\CCocor_*\to\ddeloop^m\wasreteori_0\CCocor$ of
categorified ($m+1$)-theories, since the forgetful functor
$\CCocor_*\to\CCocor$ is such that the functors induced on the
categories of $1$-morphisms are easily seen to be op-fibrations fibred
in groupoids.

Let us identify $\ddeloop^m\wasreteori^2_0\CCocor_*$ with the
$\ddeloop^m\wasreteori^2_0\CCocor$-graded ($m+1$)-theory enriched in
groupiods representing it.
From this, we obtain a simply coloured $\cat{J}$-graded
\emph{($m+1$)-theory}
${F_\cat{X}}^*(\ddeloop^m\Gamma)^*(\ddeloop^m\wasreteori^2_0\CCocor_*)$.

\begin{definition}
Let $\cat{U}$, $\cat{X}$, $P$, $\cat{J}$, $F_\cat{X}$ be as above.
Denote by $Q$ the projection
$\diag_!\wasreteori^N_{m+1}\cat{J}\to\wasreteori^N_n\cat{U}$, where
$N:=\max\{m+1,n\}$, and $\diag\colon\cat{U}\to\unity^n_\Com$.

Then we define a $\cat{U}$-graded ($m+1$)-theory $P_*\cat{X}$ as
$Q_!{F_\cat{X}}^*(\ddeloop^m\Gamma)^*(\ddeloop^m\wasreteori^2_0\CCocor_*)$.
\end{definition}

\begin{example}\label{ex:right-push-by-identity}
In the case where $\cat{V}$ is the terminal unenriched uncoloured
$\cat{U}$-graded $m$-theory $\unity^m_\cat{U}$, $\cat{X}$ can be
identified with a $\cat{U}$-graded $m$-theory, and it follows from Proposition
\ref{prop:universal-theorization} that $P_*\cat{X}=\wasreteori_m\cat{X}$.
\end{example}

\subsubsection{}

In order to do the other construction, let us introduce a notation.

Suppose given a collection $\Lambda$, and a family
$X=(X_\lambda)_{\lambda\in\Lambda}$ of groupoids parametrized by
$\Lambda$.
Then by $\prod_\Lambda X=\prod_{\lambda\in\Lambda}X_\lambda$, we
denote the not necessarily small groupoid whose truncated $n$-type is
$\prod_\Lambda X^{\le n}$ naturally formed by associations $\sigma$ to
every member $\lambda\in\Lambda$, of an object $\sigma(\lambda)\in
X_\lambda^{\le n}$, where $X_\lambda^{\le n}$ denotes the truncated
$n$-type of $X_\lambda$.

Note that $\prod_\Lambda X^{\le n}$ may not be small, but is
well-defined as a homotopy $n$-type by induction on $n$.
For example, we define $\sigma,\tau\in\prod_\Lambda X^{\le 0}$ to be
\emph{equal} if and only if $\sigma(\lambda)=\tau(\lambda)$ for every
$\lambda\in\Lambda$, and then this equality relation is an equivalence
relation, so the members of $\prod_\Lambda X^{\le 0}$ form a possibly
large $0$-type under this relation of equality.

\subsubsection{}

In order to construct $Q_*$ of Section \ref{sec:right-push-begin}, it
suffices to construct a $\cat{U}$-graded ($m+1$)-theory $P_*\cat{X}$
in the original situation \emph{modified} as follows.
\begin{itemize}
\item $\cat{V}$ is terminal on the system of colours for a
  $\cat{U}$-graded $m$-theory.
\item $\cat{X}$ is a $\cat{V}$-graded \emph{($m+1$)-theory}.
\end{itemize}
The construction is as follows. 
If necessary, Theorem \ref{thm:graded-theory-is-monoid} allows us to
replace $\cat{U}$ by $\wasreteori^N_n\cat{U}$ for any $N>n$, so we
shall assume without loss of generality,
that the dimension of $\cat{U}$ is sufficiently high.

For an object $u\in\Ob\cat{U}$, we let an object
$\sigma\in\Ob_uP_*\cat{X}$ of $P_*\cat{X}$ of degree $u$ be an
association to every $v\in\Ob_u\cat{V}$ of an object
$\sigma(v)\in\Ob_v\cat{X}$.

Let $k$ be an integer such that $1\le k\le m-1$.
Then the collections of $k$-multimaps in $P_*\cat{X}$ will inductively
be as follows.
Suppose given a $k$-multimap $u^k$ in $\cat{U}$ of arity and type
given respectively as $(I;\pi,\phi)$ and $u^{\le
  k-1}=(u^\nu)_{0\le\nu\le k-1}$, and the type
$\sigma=(\sigma^\nu)_{0\le\nu\le k-1}$ of a $k$-multimap in
$P_*\cat{X}$ of the same arity of
degree $u^{\le k-1}$ (see Definition \ref{def:graded}).
Then, we let a $k$-multimap $\tau\in\Mul^\pi_{P_*\cat{X},u^k}[\sigma]$
be an association to every $k$-multimap $v^k$ in $\cat{V}$ of arity
$(I;\pi,\phi)$ and degree $u^k$, of a $k$-multimap
$\tau(v^k)\in\Mul^\pi_{\cat{X},v^k}[\sigma(v^{\le k-1})]$, where
$v^{\le k-1}=(v^\nu)_{0\le\nu\le k-1}$ is the type of
$v^k$ (of degree $u^{\le k-1}$), and $\sigma\bigl(v^{\le
  k-1}\bigr):=\bigl(\sigma^\nu(v^\nu)\bigr)_{0\le\nu\le k-1}$ (where
$\sigma^\nu\bigl(v^\nu\bigr):=\bigl(\sigma^\nu_i(v^\nu_i)\bigr)_{i\in[I^{\nu+1}_0]}$
if $\nu\le k-2$, etc.) is by induction, the type of a
$k$-multimap in $\cat{X}$ of the same arity of degree $v^{\le k-1}$.

The collections of $m$-multimaps in $P_*\cat{X}$ will be as follows.
Suppose given an input datum similar to above, but with $k$ replaced
by $m$.
Then we let an $m$-multimap $\tau\in\Mul^\pi_{P_*\cat{X},u^m}[\sigma]$
be an association to every one of the types $v$ of $m$-multimaps in
$\cat{V}$ of arity $(I;\pi,\phi)$ and degree $u^{\le m-1}$, of an
$m$-multimap $\tau[v]\in\Mul^\pi_{P_!\cat{X},u^m}[\sigma(v)]$.

The groupoids of ($m+1$)-multimaps in $P_*\cat{X}$ will be as
follows.
Suppose given a ($m+1$)-multimap $u^{m+1}$ in $\cat{U}$ of arity and
type given respectively as $(I;\pi,\phi)$ and $u^{\le m}$, and the
type $\sigma$ of a ($m+1$)-multimap in $P_*\cat{X}$ of the same
arity of degree $u^{\le m}$.
Then we let
\[
\Mul^\pi_{P_*\cat{X},u^{m+1}}[\sigma]=\prod_v\Mul^\pi_{P_!\cat{X},u^{m+1}}[\sigma^{\le
  m-1}(v)](\sigma^m_0[v];\sigma^m_1[\pi_!v]),
\]
where $v$ runs through all the types of $\phi^{m-1}$-nerves of
$m$-multimaps in $\cat{V}$ of degree $u^{\le m-1}$.

The action of ($m+2$)-multimaps of $\cat{U}$ on the groupoids of
($m+1$)-multimaps in $P_*\cat{X}$, will be as follows.
Suppose given a ($m+2$)-multimap $u^{m+2}$ in $\cat{U}$ of arity and
type given respectively as $(I;\pi,\phi)$ and $u^{\le m+1}$, and the
type $\sigma$ of a $\phi^m$-nerve of ($m+1$)-multimaps in
$P_*\cat{X}$ of degree $u^{\le m}$.
Then we let the action
$\Mul^{\phi^m}_{P_*\cat{X},u^{m+1}_0}[\sigma]\to\Mul^{\pi_!\phi^m}_{P_*\cat{X},u^{m+1}_1}[\pi_!\sigma]$
of $u^{m+2}$ be given by composing the following two maps, namely,
\begin{itemize}
\item the map
  \begin{multline*}
  \prod_{i\in
    I^{m+1}_0}\prod_{w_i}\Mul^{\phi^m_i}_{P_!\cat{X},u^{m+1}_{0i}}[\sigma^{\le
    m-2}(w^{\le m-2}_i)]\\
  \bigl[\bigl((\maekara{\phi^m}{i-1})_!\sigma^{m-1}\bigr)\bigl(w^{m-1}_i\bigr)\bigr]\bigl(\sigma^m_{i-1}[w_i];\sigma^m_i[(\phi^m_i)_!w_i]\bigr)\\
  \\
  \longto\prod_v\prod_{i\in
    I^{m+1}_0}\Mul^{\phi^m_i}_{P_!\cat{X},u^{m+1}_{0i}}\bigl[\sigma^{\le
    m-2}(v^{\le
    m-2})\bigr]\bigl[\bigl(\maekara{\phi^m}{i-1}\bigr)_!\bigl(\sigma^{m-1}(v^{m-1})\bigr)\bigr]\\
  \bigl(\sigma^m_{i-1}[(\maekara{\phi^m}{i-1})_!v^{m-1}];\sigma^m_i[(\maekara{\phi^m}{i})_!v^{m-1}]\bigr),
  \end{multline*}
  where
  \begin{itemize}
  \item the source is simply the result of expanding the factors of
    the product
    $\Mul^{\phi^m}_{P_*\cat{X},u^{m+1}_0}[\sigma]=\prod_{i\in
      I^{m+1}_0}\Mul^{\phi^m_i}_{P_*\cat{X},u^{m+1}_{0i}}[\sigma\resto{i}]$,
  \item $v$ runs through all the types of
    $\phi^{m-1}$-nerves of $m$-multimaps in $\cat{V}$ of degree
    $u^{\le m-1}$,
  \item the map is induced from the correspondence $v\mapsto
    w_i=(\maekara{\phi^m}{i-1})_!v$, and
  \end{itemize}
\item the map
  \begin{multline*}
  \prod_v\Mul^{\phi^m}_{P_!\cat{X},u^{m+1}_0}\bigl[\sigma^{\le
    m-1}(v)\bigr]\bigl[\bigl(\sigma^m_i[(\maekara{\phi^m}{i})_!v^{m-1}]\bigr)_{i\in[I^m+1_0]}\bigr]\\
  \longto\prod_v\Mul^{\pi_!\phi^m}_{P_!\cat{X},u^{m+1}_1}[\sigma^{\le
    m-1}(v)](\sigma^m_0[v];\sigma^m_1[\pi_!v]),
  \end{multline*}
  where $v$ runs through the same range, and the map is given by the
  action of $u^{m+2}$ in $P_!\cat{X}$.
\end{itemize}

It is straightforward to extend these data naturally to a full datum
for a $\cat{U}$-graded ($m+1$)-theory $P_*\cat{X}$.

\subsubsection{}

From these constructions, we in particular have obtained the
constructions $R_*$ and $Q_*$ in the situation of Section
\ref{sec:right-push-begin}.
Therefore, we can extend the previous constructions to this situation
by defining $P_*:=Q_*R_*$.
Note Example \ref{ex:right-push-by-identity}.
$P_*$ will have the following universal property.

\begin{proposition}\label{prop:right-adjoint-to-pull-back}
Let
\begin{itemize}
\item $n\ge 0$ be an integer, and $\cat{U}$ be an $n$-theory enriched
  in groupoids,
\item $m\ge 0$ be an integer, and $\cat{V}$ be a
  $\cat{U}$-graded $m$-theory enriched in groupoids,
\item $\cat{X}$ be an unenriched $\cat{V}$-graded $m$-theory.
\end{itemize}
Denote by $P$ the projection functor
$\diag_!\wasreteori^N_m\cat{V}\to\wasreteori^N_n\cat{U}$, where $N\ge
m,n$, and $\diag\colon\cat{U}\to\unity^n_\Com$.

Then, for a $\cat{U}$-graded ($m+1$)-theory $\cat{Z}$, a functor
$\cat{Z}\to P_*\cat{X}$ of $\cat{U}$-graded ($m+1$)-theories is
naturally equivalent as a datum to a functor
$P^*\cat{Z}\to\wasreteori_m\cat{X}$ of $\cat{V}$-graded
($m+1$)-theories.
\end{proposition}

Indeed, this is an immediate consequence of the similar universal
properties of the constructions $R_*$ and $Q_*$, which can be verified
easily from the constructions.

\begin{example}\label{ex:algebra-in-right-push}
Let $\cat{U}$, $\cat{V}$ be as in Proposition, but suppose moreover
that $\cat{V}$ is uncoloured.
Then a system of colours for a $\cat{V}$-graded higher theory is the
same as a system of colours for a $\cat{U}$-graded higher theory.
Let $\ell\ge 0$ be an integer, and suppose given such a system of
colours up to dimension $\ell-1$.
Let $\cat{T}$ denote the terminal unenriched $\cat{V}$-graded
($\ell+1$)-theory on this system of colours.
We would like to consider for an integer $L\ge\ell+1, m$, Proposition
for an unenriched $\cat{V}$-graded $L$-theory $\cat{X}$ and the
$\cat{U}$-graded ($L+1$)-theory
$\cat{Z}:=P_!\wasreteori^{L+1}_{\ell+1}\cat{T}$, where $P$ is
as in Proposition.

We obtain that a functor $\cat{Z}\to P_*\cat{X}$ is equivalent as a
datum to a functor
$\wasreteori^L_{\ell+1}\cat{T}=P^*P_!\wasreteori^L_{\ell+1}\cat{T}\to\cat{X}$,
which simply defines the structure of an $\ell$-theory in $\cat{X}$,
on the considered system of colours.
\end{example}

\begin{corollary}\label{cor:base-change-right}
In addition to $\cat{U}$, $\cat{V}$, $\cat{X}$ of Proposition, suppose
given a $\cat{U}$-graded $m$-theory $\cat{W}$ enriched in groupoids.
Let $\lift{P}\colon P_!P^*\cat{W}\to\cat{W}$ be the counit for the
adjunction, and
$Q\colon\diag_!\wasreteori^N_{m+1}\cat{W}\to\wasreteori^N_n\cat{U}$
and $\lift{Q}\colon P_!P^*\cat{W}\to\cat{V}$ be respective projections.

Then there is a natural equivalence
$Q^*P_*\cat{X}\equivwith\lift{P}_*\lift{Q}^*\cat{X}$ of
$\cat{W}$-graded ($m+1$)-theories.
\end{corollary}

\begin{proof}
This follows immediately from Proposition and the following lemma.
\end{proof}

\begin{lemma}
Let
\begin{itemize}
\item $n\ge 0$ be an integer, and $\cat{U}$ be an $n$-theory enriched
  in groupoids,
\item $m\ge 0$ be an integer, and $\cat{V}$ and
  $\cat{W}$ be $\cat{U}$-graded $m$-theories enriched in groupoids.
\end{itemize}
Let $P$, $\lift{P}$, $Q$, $\lift{Q}$ be similar to those in
Proposition and Corollary \ref{cor:base-change-right}.

Then, for a $\cat{W}$-graded $m$-theory $\cat{Y}$, there is a natural
equivalence $P^*Q_!\cat{Y}\equivwith\lift{Q}_!\lift{P}^*\cat{Y}$ of
$\cat{V}$-graded $m$-theories.
\end{lemma}

\begin{proof}
Straightforward from the definitions.
\end{proof}

\subsection{Convolution for higher theories}
\label{sec:convolution}

In Example \ref{ex:algebra-in-right-push}, if $\cat{X}$ is of the form
$P^*\cat{Y}$ for a $\cat{U}$-graded $L$-theory $\cat{Y}$, then we
obtain that a $\cat{V}$-graded $\ell$-theory in $\cat{Y}$ can
equivalently be written as an $\ell$-theory in the $\cat{U}$-graded
($L+1$)-theory $P_*P^*\cat{Y}$.

There is also a coloured generalization of this.
We change the notations, and consider the following situation.

Let
\begin{itemize}
\item $n\ge 0$ be an integer, and $\cat{U}$ be an $n$-theory enriched
  in groupoids,
\item $m\ge 0$ be an integer, and $\cat{T}$ be a
  $\cat{U}$-graded ($m+1$)-theory enriched in groupoids, which is
  terminal on a system of colours up to dimension $m-1$,
\item $\cat{X}$ be an uncoloured $\cat{T}$-monoid, so together with
  $\cat{T}$, this is defining a $\cat{U}$-graded $m$-theory,
\item $\ell\ge 0$ and $L\ge\ell+1,m$ be integers, and $\cat{Y}$ be an
  unenriched $\cat{U}$-graded $L$-theory.
\end{itemize}
For these, let
$P\colon\diag_!\wasreteori^N_{m+1}\cat{T}\to\wasreteori^N_n\cat{U}$
(where $N\ge m+1,n$, and $\diag\colon\cat{U}\to\unity^n_\Com$) and
$R\colon P_!\wasreteori_m\cat{X}\to\cat{T}$ be the respective
projections.

Then the $\cat{T}$-graded ($L+1$)-theory $R_*R^*P^*\cat{Y}$ is such
that an $\cat{X}$-graded $\ell$-theory in $\cat{Y}$ is equivalent as
a datum to an $\ell$-theory in $R_*R^*P^*\cat{Y}$.
Therefore, $R_*R^*P^*\cat{Y}$, considered as a construction between
the $\cat{U}$-graded $m$-theories $\cat{X}$ and $\cat{Y}$, is in a way
analogous to the Day convolution for monoidal categories \cite{day}.

\begin{remark}
In the case $\ell=0$, we obtain that the datum of an $\cat{X}$-graded
$0$-theory in $\cat{Y}$ is further
equivalent to the datum of a $0$-theory in $(PR)_*(PR)^*\cat{Y}$.
This may be closer to the conventional contexts for the Day
convolution.
\end{remark}

The purpose of this section is to obtain an enriched generalization of
this for the case $\ell=m-1$, where enrichment is in a symmetric
monoidal higher category.

Let us start with the following.

\begin{definition}
Suppose given
\begin{itemize}
\item an integer $d\ge 0$, and an $d$-theory $\cat{T}$ enriched in
  groupoids,
\item an integer $k$ such that $0\le k\le d$, and a symmetric
  monoidal $k$-category $\cat{A}$,
\item functors
  $F,G\colon\cat{T}\to\ddeloop^{d-k}\wasreteori^k_0\cat{A}$, or
  equivalently,
  $(F,G)\colon\cat{T}\to\ddeloop^{d-k}\wasreteori^k_0(\cat{A}\times\cat{A})$.
\end{itemize}
Then we define a $\cat{T}$-graded $d$-theory $\Fun(F,G)$ as
$(F,G)^*\ddeloop^{d-k}\Fun_\cat{A}$, where
$\Fun_\cat{A}$ denotes $\wasreteori^k_0\Fun([1],\cat{A})$ (where the
$k$-category $\Fun([1],\cat{A})$ is given the object-wise symmetric
monoidal structure) considered as a
$\wasreteori^k_0(\cat{A}\times\cat{A})$-graded $k$-theory via
the symmetric monoidal functor
\[
(d_1,d_0)\colon\Fun([1],\cat{A})\longto\cat{A}\times\cat{A}.
\]
induced from the simplicial coface operators
$d^1,d^0\colon[0]\to[1]$.
\end{definition}

We would like to generalize the construction ``$\Fun$'' for
\emph{coloured} theories.
The following definition includes this.

\begin{definition}
Let
\begin{itemize}
\item $n\ge 0$ be an integer, and $\cat{U}$ be an $n$-theory enriched
  in groupoids with the unique functor
  $\diag\colon\cat{U}\to\unity^n_\Com$,
\item $m\ge 0$ be an integer, and $\cat{T},\cat{J}$ be
  $\cat{U}$-graded ($m+1$)-theory enriched in groupoids,
\item $k$ be an integer such that $0\le k\le m+1$, and $\cat{A}$ be a
  symmetric monoidal $k$-category,
\item
  $F\colon\cat{T}\to\diag^*\ddeloop^{m+1-k}\wasreteori^k_0\cat{A}$
  and
  $G\colon\cat{J}\to\diag^*\ddeloop^{m-k+1}\wasreteori^k_0\cat{A}$
  be functors of $\cat{U}$-graded ($m+1$)-theories.
\end{itemize}
For these, let
$P\colon\diag_!\wasreteori^N_{m+1}\cat{T}\to\wasreteori^N_n\cat{U}$
(where $N:=\max\{m+1,n\}$) and
$\lift{Q}\colon P_!P^*\cat{J}\to\cat{T}$ be the respective
projections, and $\lift{P}\colon P_!P^*\cat{J}\to\cat{J}$ be the
counit for the adjunction.

Then we define the $\cat{T}$-graded ($m+1$)-theory
$\Fun\bigl((\cat{T},F),(\cat{J},G)\bigr)$ as
$\lift{Q}_!\Fun(F\lift{Q},G\lift{P})$, where
\[
F\lift{Q},G\lift{P}\colon\diag_!\wasreteori^N_{m+1}P_!P^*\cat{J}\longto\wasreteori^N_{m-k+1}\ddeloop^{m-k+1}\cat{A}
\]
are the indicated functors of symmetric $N$-theories.
\end{definition}

This indeed gives the desired enriched generalization of the previous
construction done using the right push-forward construction.
Namely, we obtain the following in the special case where $k=1$ and
$\cat{A}=\Gpd$ of the construction here.

\begin{proposition}
Let
\begin{itemize}
\item $n\ge 0$ be an integer, and $\cat{U}$ be an $n$-theory enriched
  in groupoids with the unique functor
  $\diag\colon\cat{U}\to\unity^n_\Com$,
\item $m\ge 0$ be an integer, and $\cat{T}$ and $\cat{J}$ be
  $\cat{U}$-graded ($m+1$)-theories enriched in groupoids, each of
  which is terminal on a system of colours up to dimension $m-1$,
\item $F\colon\cat{T}\to\diag^*\ddeloop^m\wasreteori_0\Gpd$ and
  $G\colon\cat{J}\to\diag^*\ddeloop^m\wasreteori_0\Gpd$ be functors
  of $\cat{U}$-graded ($m+1$)-theories.
\end{itemize}
Denote by $\cat{X}$ and $\cat{Y}$, the $\cat{U}$-graded $m$-theories
defined by $F$ and $G$.

Then
$\Fun\bigl(\cat{X},\cat{Y}\bigr):=\Fun\bigl((\cat{T},F),(\cat{J},G)\bigr)$
is equivalent to $R_*R^*P^*\cat{Y}$, where
$R\colon\wasreteori_m\cat{X}\to\cat{T}$ and
$P\colon\diag_!\wasreteori^N_{m+1}\cat{T}\to\wasreteori^N_n\cat{U}$
(where $N\ge m+1,n$) are the respective projections.
\end{proposition}

The proof is straightforward by direct inspection of the
constructions.

It follows that, if $m\ge 1$, an $\cat{X}$-algebra in $\cat{Y}$ is
equivalent as a datum to an ($m-1$)-theory in
$\Fun(\cat{X},\cat{Y})$.
Since the latter notion makes sense for any symmetric monoidal
$k$-category $\cat{A}$, where $0\le k\le m+1$, we can think of
it as the definition in such an enriched context, of
an $\cat{X}$-algebra in $\cat{Y}$.

\section{Higher theorization of symmetric monoidal functor}
\label{sec:generalization}

\subsection{The definition}
\label{sec:generalized-theory-definition}

\subsubsection{}

We have so far considered iterated theorizations of algebra over a
multicategory.
One might wonder whether there are iterative theorizations of
symmetric monoidal functor on a fixed symmetric
monoidal category $\cat{B}$, to a varying target symmetric monoidal
category.

This is actually a generalization of the previous case.
Indeed, the functor $\wasreteori_0$ from symmetric monoidal categories
to multicategories, has a left adjoint $L$, so, for a multicategory
$\cat{U}$, a symmetric monoidal functor on $L\cat{U}$ is the same
as a $\cat{U}$-algebra, which we have already theorized.

We would like to show here that, if a symmetric monoidal category
$\cat{B}$ admits a certain concrete form of description, then we
indeed obtain iterated theorizations of symmetric monoidal functor on
$\cat{B}$, by replacing in Definition \ref{def:theory} of a symmetric
higher theory, the coCartesian symmetric monoidal category $\Fin$ by
$\cat{B}$.

It turns out that $\cat{B}$ may more generally be a symmetric monoidal
higher category satisfying certain conditions.

\begin{remark}\label{rem:inversion-restriction}
When the dimension of the category $\cat{B}$ is at least $2$, then a
source of difficulty for having interesting symmetric monoidal
functors $\cat{B}\to\cat{A}$, where $\cat{A}$ is a symmetric monoidal
category, is that such a functor must invert all maps
in $\cat{B}$ in dimensions $\ge 2$.
However, this restriction will be discarded in one dimension at a
time as theorization (in particular, relaxation) of the notion is
iterated.
Indeed, in the $n$-dimensional theory which we define below, the
inversion of maps of $\cat{B}$ will be forced only in dimensions $\ge
n+2$.
\end{remark}

Let $\cat{B}$ be a symmetric monoidal higher category, and denote
the underlying higher category of $\cat{B}$ by $\cat{C}$, so
$\cat{B}$ is $\cat{C}$ equipped with a symmetric monoidal structure.
The conditions we would like to impose on $\cat{B}$ are ($0$) and
``($k$)'' below for all integers $k\ge 1$.

\begin{itemize}
\item[($0$)] The groupoid $\Ob\cat{C}$ of objects of $\cat{C}$ is free
  as a commutative monoid on a groupoid of generators.
\item[($k$)] Suppose given data of the forms ($k-1'$) through ($1'$)
  of ($k$) in Section \ref{sec:k}, or equivalently, the arity of a
  $k$-multimap in an $\Ini$-graded higher theory.
  Then the groupoid formed by all $k$-multimap of the specified arity
  in the $\Ini$-graded $k$-theory $\wasreteori^k_1\cat{C}$, is free on
  a groupoid as a commutative monoid under the symmetric monoidal
  structure of $\cat{B}$.
\end{itemize}

Note that the groupoid of free generators is then the full subgroupoid
consisting of \emph{indecomposable} objects (from which we exclude the
unit(s) by definition).

\begin{remark}\label{rem:free-generation-finite-dim}
If $\cat{B}$ is in fact a symmetric monoidal \emph{$d$-category} for a
finite value of $d$, then the groupoid of ($d+1$)-multimaps, and more
gerenerally, of $\phi^d$-nerves of ($d+1$)-multimaps in
$\wasreteori^{d+1}_1\cat{C}$ (see Definition \ref{def:theory}), will
be equivalent to the groupoid of $\phi^{d-1}$-nerves of $d$-morphisms
in $\wasreteori^d_1\cat{C}$.
In particular, the conditions ($k$) for all integers $k\ge 0$ hold in
this case if the conditions hold for all $k\le d$, and the groupoid
of $\phi^{d-1}$-nerves of $d$-morphisms is free as a commutative
monoid for every specification of the arity.
\end{remark}

In addition to $\cat{B}=\Fin$, the following are examples of $\cat{B}$
satisfying these conditions.
The symmetric monoidal structures are all given by the ``disjoint
union''.

\begin{itemize}
\item The $2$-category $\CCorr(\Fin)$ of correspondences of finite
  sets, and its underlying $1$-category $\Corr(\Fin)$.
\item The category $\Bord_1$ of compact $0$-dimensional manifolds and
  (the groupoids of) $1$-dimensional bordisms between them.
  Any choice of tangential structure on manifolds.
\item The category $\End^{d-1}_{\Bord_d}(\kara)$ of closed
  ($d-1$)-manifolds and (the groupoids of) $d$-dimensional bordisms.
  Any choice of tangential structure.
\item The \emph{fully extended cobordism $d$-category} $\Bord_d$ of
  bordisms up to dimension $d$.
  Any choice of tangential structure.
\item For an integer $k$ such that $2\le k\le d-1$, the $k$-category
  $\End^{d-k}_{\Bord_d}(\kara)$ of endomorphisms in $\Bord_d$ of the
  empty ($d-k-1$)-dimensional cobordism.
\item The ($d+1$)-category of $d$-th iterated cocorrespondences in
  $\Fin$, and its underlying
  $d$-category.
  See e.g., Lurie \cite[Section 3.2]{lurie-tft} for the idea, and
  Ben-Zvi and Nadler \cite[Remark 1.17]{bzn} for an explicit
  discussion of a definition which applies readily here.
  See also Calaque \cite{calaque}.
\end{itemize}

\begin{remark}
One can also define versions of the fully extended cobordism category
where each bordism is given a codimension $n$ embedding into the
Euclidean space.
These bordisms will form an $E_n$-monoidal $d$-category.

While our technique so far does not seem to apply directly for
theorizing the notion of topological field theory on such a category,
this kind of category seems close to satisfying an $E_n$ analogue of
the assumption required for our technique.
We hope to treat theorization of these topological field theories in a
sequel to this work.
\end{remark}

\begin{remark}
Some other non-embedded (symmetric monoidal) variants of the cobordism
category, such as discussed by Lurie in \cite[Section 4]{lurie-tft},
also satisfy the conditions of Remark
\ref{rem:free-generation-finite-dim}.
\end{remark}

\subsubsection{}

Let now $\cat{B}$ be a symmetric monoidal higher category
satisfying the conditions ($k$) above for every integer $k\ge 0$.
We shall obtain an $n$-th theorization of symmetric monoidal functor
(to a variable symmetric monoidal $1$-category) on $\cat{B}$.
(See Remark \ref{rem:inversion-restriction}.)
Our $n$-theorized objects will be called ``$\cat{B}$-graded
$n$-theories''.

A \kore{$\cat{B}$-graded $n$-theory} will consist of data similar to
the data ($k$) for $k\ge 0$, for a symmetric (i.e., ``$\Fin$-graded'')
$n$-theory (see Section \ref{sec:symmetric-theory}), but with
appropriate modifications applied as follows.

Firstly, the form of datum ($0$) (in the case $n\ge 1$) will be as
follows.
\begin{itemize}
\item[($0$)] For every indecomposable object $b$ of $\cat{B}$, a
  collection $\Ob_b\cat{U}$, whose member will be called an
  \kore{object} of $\cat{U}$ of \kore{degree} $b$.
\end{itemize}

We extend this for an arbitrary object $b\in\cat{B}$ as follows.
Namely, we let $\Ob_b\cat{U}$ be the collection of \emph{$b$-families}
of objects of $\cat{U}$, defined as follows.

\begin{definition}
Let $b$ be an object of $\cat{B}$.
Then a \kore{$b$-family} of objects of a $\cat{B}$-graded higher
theory $\cat{U}$, is a pair consisting of
\begin{itemize}
\item a decomposition $b\equivwith\bigtensor_Sc$, where $S$ is a
  finite set, and $c=(c_s)_{s\in S}$ is an $S$-family of
  indecomposable objects of $\cat{B}$, and
\item a \emph{$c$-family} $u\in\Ob_c\cat{U}$, by which we mean that
  $u$ is an $S$-family $(u_s)_{s\in S}$, where
  $u_s\in\Ob_{c_s}\cat{U}$.
\end{itemize}
\end{definition}

Let us next describe the type of a $1$-multimap.
This is where the true difference of a general $\cat{B}$-graded theory
from the case $\cat{B}=\Fin$ is seen.
Namely, for a general $\cat{B}$, a multimap in a $\cat{B}$-graded
theory will in general, not only accept multiple inputs, but also emit
multiple outputs.

Thus, the \emph{type of a multimap in a $\cat{B}$-graded higher
theory $\cat{U}$} consists of
\begin{itemize}
\item[($0'$)] a map $b^1\colon b^0_0\to b^0_1$ in $\cat{B}$ which is
  indecomposable with respect to the commutative monoid structure,
\item[($0''$)] a $b^0$-family $u$ of objects of $\cat{U}$, namely,
  $u=(u_i)_{i=0,1}$, where $u_i$ is a $b^0_i$-family of objects of
  $\cat{U}$.
\end{itemize}

In general, we modify Definition \ref{def:theory} of a symmetric
$n$-theory in the following two respects.
\begin{itemize}
\item We modify the form of datum ($k$) for every $k\ge 1$ as
  follows.
  We
  \begin{itemize}
  \item keep the forms of input data ($k-1'$) through ($1'$)
    unchanged,
  \item replace the nerve in $\Fin$ in ($0'$) with a $k$-multimap in
    $\wasreteori^k_1\cat{C}$ (where $\cat{C}$ denotes the higher
    category underlying $\cat{B}$ as
    before) of the specified arity, which is moreover, indecomposable
    with respect to the commutative monoid structure,
  \item modify the form of the rest of input data accordingly,
  \item let the similar datum as associated in ($k$) of Definition
    \ref{def:theory}, be associated to the input data of these
    modified forms.
  \end{itemize}
\item We generalize the process of extension of the datum ($k$) to a
  process of extension from indecomposable to arbitrary
  $k$-multimaps, using
  the free decomposition of $k$-multimaps just similarly to before.
\end{itemize}

\begin{example}\label{ex:properad}
For the symmetric monoidal $1$-category $\cat{B}=\Cocor(\Fin)$
underlying the symmetric monoidal $2$-category of cocorrespondences in
$\Fin$, the notion of $\cat{B}$-graded $1$-theory enriched in a
symmetric monoidal category $\cat{A}$, coincides with the notion of
coloured \emph{properad} of Vallette in $\cat{A}$ \cite{vallette}.
Thus, the notion of (coloured) properad is a theorization of the
notion of $\cat{B}$-algebra, and $\cat{B}$-graded higher theories give
further theorizations.
\end{example}

\subsection{Symmetric monoidal functors as algebras in a theory}
\label{sec:monoidal-functor-as-theory}

\subsubsection{}

We would like to see that, if $\cat{B}$ is a symmetric monoidal
$d$-category which satisfies our assumptions, then symmetric monoidal
functors $\cat{B}\to\cat{A}$ with $\cat{A}$ any symmetric
monoidal $d$-category, are indeed included in our framework.

\subsubsection{}

Let us start with the following.
Let $\cat{B}$ be a symmetric monoidal $d$-category which satisfies our
assumptions ($k$) for all integers $k\ge 0$.
Then we would like to construct a $\cat{B}$-graded $d$-theory from a
symmetric monoidal $d$-category $\cat{E}$ equipped with a symmetric
monoidal functor $P\colon\cat{E}\to\cat{B}$.
The $d$-theory, which we shall denote by $\cat{E}^\theta$, is as
follows.

For a indecomposable object $b\in\cat{B}$, an object of
$\cat{E}^\theta$ of degree $b$ is an object of $\cat{E}$ in the fibre
over $b$.

For every integer $k$ such that $1\le k\le d-1$, the datum ($k$) which
we specify for $\cat{E}^\theta$ is inductively as follows.
Denote by $\cat{C}$ the $d$-category, i.e., ($d-1$)-categorified
$\Ini$-graded $1$-theory, underlying $\cat{B}$.
Suppose given
\begin{itemize}
\item an indecomposable $k$-multimap $b^k$ in $\wasreteori^k_1\cat{C}$
  of arity specified by data of the forms ($k-1'$) through ($1'$) of
  ($k$) in Section \ref{sec:k}, and of type
  given as $b^{\le k-1}=(b^\nu)_{0\le\nu\le k-1}$,
\item a type $e=(e^\nu)_{0\le\nu\le k-1}$ of a $k$-multimap in the
  $\cat{B}$-graded theory $\cat{E}^\theta$ of the same arity of degree
  $b^{\le k-1}$.
\end{itemize}
By induction, $e^{k-1}_i$ (where $i=0,1$) will be a family consisting
of lifts in $\cat{E}$ of the factors\slash components (in the unique
decomposition in $\cat{B}$) of the (nerve of) ($k-1$)-multimaps (or
objects if $k=1$) $b^{k-1}_i$, so $\Tensor e^{k-1}_i$ in $\cat{E}$
lifts $b^{k-1}_i$, where $\Tensor$ indicates taking the monoidal
product of the members of the family (which is $e^{k-1}_i$ here) in
$\cat{E}$.
Moreover, if $k\ge 2$, then the ($k-1$)-morphisms $\pi_!\bigl(\Tensor
e^{k-1}_0\bigr)$ and $\Tensor e^{k-1}_1$ in $\cat{E}$ have common
source and target by induction.

Given these data, we define a $k$-multimap $e^{k-1}_0\to e^{k-1}_1$ in
$\cat{E}^\theta$ of degree $b^k$, to be a lift of $b^k$ to a
$k$-morphism $\pi_!\bigl(\Tensor e^{k-1}_0\bigr)\to\Tensor e^{k-1}_1$
(or $\Tensor e^0_0\to\Tensor e^0_1$ if $k=1$) in $\cat{E}$, completing
the induction.

Similarly, the groupoids of $n$-multimaps in $\cat{E}^\theta$ will be
the groupoids of similar lifts, and $n$-multimaps in $\cat{E}^\theta$
compose by the composition of $n$-multimaps in
$\wasreteori^n_0\cat{E}$.

Thus we have constructed a $\cat{B}$-graded $d$-theory
$\cat{E}^\theta$.

\begin{remark}\label{rem:right-localized}
This construction is not faithful in $\cat{E}$ (equipped with
$P\colon\cat{E}\to\cat{B}$).
Instead, the construction $\blank^\theta$ gives (non-trivial) right
localization functors of suitable categories.
\end{remark}

\subsubsection{}

Let us denote by $\unity^d_\cat{B}$, the terminal unenriched
uncoloured $\cat{B}$-graded $d$-theory.
Inspecting the construction above, it is easy to see that a $0$-theory
in $\cat{E}^\theta$, or a functor $\unity^d_\cat{B}\to\cat{E}^\theta$,
is equivalent as a datum to a section to the symmetric monoidal
functor
$P\colon\cat{E}\to\cat{B}$, which commutes with the symmetric monoidal
structures, but is ($d-1$)-lax as a functor.

Let now $\cat{A}$ be a symmetric monoidal $d$-category.
Then we have the projection functor $\cat{A}\times\cat{B}\to\cat{B}$,
which is symmetric monoidal.
It follows that a $0$-theory in $(\cat{A}\times\cat{B})^\theta$ is a
symmetric monoidal ($d-1$)-lax functor $\cat{B}\to\cat{A}$.

\begin{remark}
Even though the construction above has thus captured symmetric
monoidal functors $\cat{B}\to\cat{A}$ for every symmetric monoidal
\emph{$1$-category} $\cat{A}$, this is not the most interesting target
if $d\ge 2$.
However, in the case where $d\ge 2$ and the dimension of $\cat{A}$ is
also $d$, the reader may be unsatisfied for the laxness which has
crept in.
(In relation to Remark \ref{rem:right-localized}, this laxness is due
to the way how $\cat{B}$ as the terminal one among symmetric monoidal
$d$-categories lying over $\cat{B}$, can fail to be local with respect
to the right localization if $d\ge 2$.)
Compared with $d$-lax symmetric monoidal functors $\cat{B}\to\cat{A}$,
which could be captured in the framework of
$\wasreteori^d_0\cat{B}$-graded $d$-theories in the previous approach,
our new approach here has only eliminated the laxness with respect to
the symmetric monoidal structure.
However, in order to deal with the remaining laxness with a similar
technique, we would need to assume a finer version of the unique
decomposition, which would be more difficult to be satisfied.
\end{remark}

\subsection{An example of different nature}
\label{sec:example-different-nature}
In the case $\cat{B}=\Bord_1$, there is an example of a $1$-theory
which is associated to a category rather than a symmetric monoidal
category.
An algebra in it will appear very different from a $1$-dimensional
field theory in the usual sense.
Let us sketch these.
The tangential structure we consider is framing, or equivalently,
orientation.

Let $\cat{C}$ be a category.
Then we can use it to construct a $\Bord_1$-graded $1$-theory as
follows.

Firstly, we need to associate to every indecomposable object of
$\Bord_1$, a collection to be the collection of objects of that
degree.
To every $0$-dimensional manifold consisting of one point $\pt$ with
any framing of $\pt\times\R^1$, we associate the collection
$\Ob\cat{C}$.

Next, we need to associate to every indecomposable map in $\Bord_1$, a
groupoid to be the groupoid of $1$-multimaps of that degree.
\begin{itemize}
\item If the bordism is diffeomorphic to the interval as a manifold,
  then to every object $x\in\Ob\cat{C}$ at the incoming (relatively to
  the orientation) end point, and every object $y\in\Ob\cat{C}$ at the
  outgoing end point, we associate the groupoid $\Map_\cat{C}(x,y)$.
\item If the bordism is diffeormorphic to the circle, then, for
  simplicity, we associate the terminal groupoid (but note Remark
  \ref{rem:theory-with-hochschild} below).
\end{itemize}

Finally, we need to define the composition operations.
This can be given by the composition of maps in $\cat{C}$ (and its
associativity).

Thus, we have sketched a construction of a $\Bord_1$-graded
$1$-theory.
Let us denote this theory by $\cat{Z}_\cat{C}$.

Note that, in the case where $\cat{C}$ is the unit category $\unity$,
we obtain $\cat{Z}_\unity=\unity^1_{\Bord_1}$, the terminal
$\Bord_1$-graded $1$-theory.
It follows that, in general, any object $x\in\cat{C}$, or
equivalently, a functor $\unity\to\cat{C}$, induces a functor
$\cat{Z}_x\colon\unity^1_{\Bord_1}\to\cat{Z}_\cat{C}$, which, by
definition, is a \emph{field theory in $\cat{Z}_\cat{C}$}.

The contractibility of the diffeomorphism group of the framed interval
implies that $\cat{Z}_x$ for $x\in\cat{C}$ exhaust all field theories
in $\cat{Z}_\cat{C}$.

\begin{remark}\label{rem:theory-with-hochschild}
There is another version of $\cat{Z}_\cat{C}$, in which the groupoid
which we associate to the circle is the Hochschild homology
$\HH_\dot\cat{C}\equivwith\int^{x\in\cat{C}}\Map_\cat{C}(x,x)$.
All the claims above also hold for this version of $\cat{Z}_\cat{C}$,
as a result of the following observation (and simple computations).

The observation is as follows.
Let $\ptn$, $\ptp$ denote the two $1$-framed points of opposite
framings, and let $I$ denote the terminal category (i.e.,
$\Ini$-graded $1$-theory) on the two colours ``$\ptn$'' and
``$\ptp$'' (so, as a category, $I$ is a contractible groupoid).
Let $\Theory_I$ and $\Theory_\Bord$ respectively denote the categories
of $I$-graded and of $\Bord_1$-graded $1$-theories.
There is an obvious adjunction
$\diag_!\colon\Theory_I\rightleftarrows\Cat\noloc\diag^*$, where
$\diag\colon\{\ptn,\ptp\}\to\ten$.
The contractibility of the diffeomorphism group of the framed interval
implies that there is also a functor $\Theory_\Bord\to\Theory_I$ with
left adjoint $\close{\cat{Z}}$ satisfying
$\cat{Z}=\close{\cat{Z}}\compose\diag^*$.
\end{remark}

\appendix

\section{A comparison to the work of Beaz and Dolan}
\renewcommand*{\thesubsubsection}{\thesection.\arabic{subsubsection}}

\subsubsection{}

We have been informed of resemblance between our work and the
beautiful
pioneering work as early as about two decades ago, of Baez and Dolan
\cite{bd}.
Since some of our purposes overlap theirs, and the methods also has
great similarity, we think that a comparison of two works would be
worthwhile.

Specifically, Baez and Dolan introduce what they call the ``slice
operad'' construction for the purpose of defining the notion of
``opetopic set'', which they use to give a definition of an
$n$-category.
The slice operad construction is not just interesting and powerful,
but some of the ideas which they have developed for this construction,
are quite close to some of the ideas which we have used for our work.
There seems to be no doubt therefore, that our work was shaped by the
great influences from some ideas which go back to their work (or were
at least popularized quite possibly as we imagine,
through their work).

The ``slice'' construction constructs from a multicategory $O$, a new
multicategory $O^+$.
This looks close to our $\wasreteori O$, even though $\wasreteori O$
is a $2$-theory, rather than a $1$-theory.
Since Baez and Dolan constructed $O^+$ as a multicategory, they did
not need to introduce a new concept like our concept of higher
theory.
Moreover, iteration of their construction is automatic, unlike
iteration of the process of theorization, which was the first main
theme of our work.
We recognize this as a great advantage of their construction.

\subsubsection{}

This does not mean, however, that staying in the world of
multicategories is necessary or desirable in all respects.
We have succeeded after all, in defining all the higher notions of
theory, and the new framework accommodates simpler approaches to some
issues.
For example, the case $n=1$ of our Theorem
\ref{thm:graded-theory-is-monoid} implies that the $2$-theory
$\wasreteori O$ is such that an uncoloured $\wasreteori O$-algebra is
precisely the same as an $O^+$-algebra (described in the quotation
below).
Moreover, the construction of $\wasreteori O$ from $O$ was direct and
immediate, while the construction of $O^+$ is one of the major steps
in Baez and Dolan's work.

\begin{example}
In the case where $O$ is the associative operad, an $O^+$-algebra is
precisely an uncoloured planar operad, while a $\wasreteori
O$-algebra (which naturally has ``colours'' in general) is precisely
a planar multicategory, i.e., a coloured planar operad.
\end{example}

The notion of theorization also clarify the work of Baez and Dolan
conceptually.
Let us first hear the description of the slice operad in the
inventors' own words.
We shall quote from \cite{bd}.
In the context at hand, their term ``operad'' means coloured operad
(in sets, over which ``algebras'' are also considered in sets), and
``type'' means colour in our terminology.
\begin{quotation}
``We define the ``slice operad'' $O^+$ of an operad $O$ in such a way
that an algebra of $O^+$ is precisely an operad over $O$, i.e., an
operad with the same set of types as $O$, equipped with an operad
homomorphism to $O$. Syntactically, it turns out that:
\begin{enumerate}
\renewcommand{\labelenumi}{\theenumi.}
\item The types of $O^+$ are the operations of $O$.
\item The operations of $O^+$ are the reduction laws of $O$.
\item The reduction laws of $O^+$ are the ways of combining reduction
  laws of $O$ to give other reduction laws.
\end{enumerate}
This gets at the heart of the process of ``categorification,'' in
which laws are promoted to operations and these operations satisfy new
coherence laws of their own. Here the coherence laws arise simply from
the ways of combining the the old laws.''
[\textit{sic}]\\
\mbox{}\hfill(John C.~Baez and James Dolan \cite[Section 1]{bd})
\end{quotation}
In their work, they observe the points (1), (2), (3) from the actual
construction of $O^+$, but do not explicitly discuss the conceptual
reason for why $O^+$ had to be related to categorification.
The notion of theorization sheds light on this.
Indeed, ``an operad over $O$'' in their definition, is precisely an
(uncoloured) theorized $O$-monoid (or algebra ``enriched'' in sets),
as those authors may have known in some formulation.

\subsubsection{}

There are also other advantages in employing higher theories.
For example, recall that the ultimate goal of Baez and Dolan was to
give a definition of an $n$-category.
For this, they needed a few more steps after defining the slice operad.
On the other hand, a version of $n$-categories are already among the
$n$-theories.
To examine the difference closely, we generally consider an $n$-theory
formed not just by the $n$-multimaps, but with strata of colours
consisting of objects to ($n-1$)-multimaps (which is the usual
``colour'' in an operad in the case $n=1$), and in a special case, the
structure of an $n$-category is formed by these objects as objects,
and the unary higher multimaps as higher morphisms.
Contrary to this, Baez and Dolan do not consider an algebraic
structure having more than one layer of colours since they consider
only multicategories.
This is the reason why they needed to find another route which might
appear like a detour from the point of view of higher theories.
(However, some opetopic sets appear to be modeling a version of
initially graded higher theories, so their method merely does not
appear as direct as one can wish to make it.)

The flexibility coming from the rooms for strata of colours, is also
important for considering enrichment, since possibility for more
interesting enrichment requires more strata of colours.
Even though the purposes of Baez and Dolan did not motivate them to
consider a very general notion of enrichment, their framework as built
may not support a very interesting notion of enrichment, either.
Note also that, even if one intends to work only with mutlicategories,
enrichment of multicategories is most generally done along
$2$-theories.

\subsubsection{}

From the quotation of Baez and Dolan's words above, the idea expressed
in the final two sentences is remarkable.
Indeed, it is exactly the idea which we have described in Section
\ref{sec:definition}, and used in our definition of an $n$-theory,
except for two differences.

One difference is that we see the same, more generally at the heart of
theorization.
The other is that we have a simpler understanding of the ``new
coherence law'', in terms of the theorized form of the structure.

Now, our version of their idea has led to a process which keeps the
complexity of structures from increasing rapidly by instead raising
the theoretic order, and the resulting simplicity helped us enormously
with various constructions concerning higher theories.
(In those constructions, roles were also played by the flexibility
from the rooms for strata of colours.)

Our version of their idea also helped us with treating some systems of
operations with multiple inputs and multiple outputs.

\subsubsection{}

To summarize, our work has benefited from the fruits of the
developments which were initiated by such prominent works as Beaz and
Dolan's.

\oldsubsection*{Acknowledgement}

I found the first hint for this work in interesting discussions years
ago with Kevin Costello, who was then my adviser at Northwestern
University.
I am grateful to him for sharing his great wisdom with me for those
years.
I am grateful to John Baez and James Dolan for their seminal works.
I am grateful to Ross Street and Isamu Iwanari for their comments on
the manuscripts for what has become Section \ref{sec:introduction}.
I am grateful to the participants of my informal talks on this work
for many helpful comments and questions, especially to Sh\={o}
Sait\={o}, K\={o}tar\={o} Kawatani, Isamu Iwanari, Hiroyuki Minamoto,
Norihiko Minami.
I am grateful to Isamu Iwanari and Norihiko Minami also for their
encouragement on various occasions.

\renewcommand{\section}{\oldsection}


\begin{thebibliography}{99}
\setcounter{enumiv}{-1}
\bibitem{bd}Baez, John C.; Dolan, James. \emph{Higher-dimensional
  algebra. III: $n$-categories and the algebra of opetopes.}
Adv.~Math.~\textbf{135}, 1998, No.~2, 145--206.

\bibitem{benabou}B\'enabou, Jean. \emph{Introduction to bicategories.}
  Reports of the Midwest Category Seminar, 1--77. Springer Berlin
  Heidelberg, 1967.

\bibitem{bzn}Ben-Zvi, David; Nadler, David. \emph{Nonlinear traces.}
  arXiv:1305.7175.

\bibitem{homotopy-everything}Boardman, J.~M.; Vogt,
  R.~M. \emph{Homotopy-everything H-spaces.}
  Bull.~Am.~Math.~Soc.~\textbf{74}, 1968, 1117--1122.

\bibitem{homotopy-invariant}Boardman, J.~M.; Vogt,
  R.~M. \emph{Homotopy Invariant Algebraic Structures on Topological
    Spaces.} Lecture Notes in Mathematics
  \textbf{347}. Springer-Verlag, Berlin-Heidelberg-New York,
  1973. X$+$257 pp.

\bibitem{calaque}Calaque, Damien. \emph{Lagrangian structures on
    mapping stacks and semi-classical TFTs.} arXiv:1306.3235.

\bibitem{crane}Crane, Louis. \emph{Clock and category: Is quantum
    gravity algebraic?} J.~Math.~Phys.~\textbf{36}, 1995, No.~11,
  6180--6193.

\bibitem{crane-frenke}Crane, Louis; Frenkel, Igor
  B. \emph{Four-dimensional topological quantum field theory, Hopf
    categories, and the canonical bases.} J.~Math.~Phys.~\textbf{35},
  1994, No.~10, 5136--5154.

\bibitem{day}Day, Brian J. \emph{Construction of biclosed categories.}
  Thesis for Doctor of Philosophy, University of New South
  Wales, 1970.

\bibitem{fiedoro}Fiedorowicz, Z. \emph{The symmetric bar
    construction.} Preprint available at
  https:\slash\slash people.math.osu.edu\slash fiedorowicz.1\slash.

\bibitem{gro-semi-bour}Grothendieck, Alexander. \emph{Technique de
    descente et th\'eor\`emes d'existence en g\'eom\'etrie
    alg\'ebrique. I: G\'en\'eralit\'es. Descente par morphismes
    fid\`element plats.} S\'em.~Bourbaki \textbf{12} (1959/60),
  299--327, No.~190, 1960.

\bibitem{sga1-6}Grothendieck, Alexander. \emph{Cat́\'egories fibr\'ees
    et descente.} S\'eminaire de G\'eom\'etrie Alg\'ebrique du Bois
  Marie 1960/61 (SGA 1), \emph{Rev\^etements \'Etales et Groupe
    Fondamental.} Expos\'e VI. Lecture Notes in
  Mathematics. \textbf{224}. Springer-Verlag, Berlin-Heidelberg-New
  York, 1971.

\bibitem{joyal-street}Joyal, Andr\'e; Street, Ross. \emph{Braided
    tensor categories.} Adv.~Math.~\textbf{102}, 1993, No.~1, 20--78.

\bibitem{lambek}Lambek, J. \emph{Deductive systems and
    categories. II. Standard constructions and closed categories.}
  Category Theory, Homology Theory and Their Applications I,
  Proc.~Conf.~Seattle Res.~Center Battelle Mem.~Inst.~\textbf{1968},
  1969, 1, 76--122.

\bibitem{lawvere}Lawvere, F.~William. \emph{Functorial semantics of
    algebraic theories.} Ph.D. thesis, Columbia University, 1963.

\bibitem{lurie-tft}Lurie, Jacob. \emph{On the classification of
    topological field theories.} Current developments in mathematics,
  \textbf{2008}, 129--280, Int.~Press, Somerville, MA, 2009.

\bibitem{mac-algebra}MacLane, S. \emph{Categorical algebra.}
  Bull.~Am.~Math.~Soc.~\textbf{71}, 1965, 40--106.

\bibitem{vallette}Vallette, Bruno. \emph{A Koszul duality for props.}
  Trans.~Am.~Math.~Soc.~\textbf{359}, 2007, No.~10, 4865--4943.
\end{thebibliography}
\end{document}